\documentclass[11pt]{article}

\usepackage{graphicx}
\usepackage{latexsym,amssymb}
\usepackage{amsthm}
\usepackage{indentfirst}
\usepackage{amsmath}

\newtheorem{theoremalph}{Theorem}

\newtheorem{Theorem}{Theorem}[section]
\newtheorem*{Theorem A}{Theorem A}
\newtheorem*{Thm}{Theorem}

\newtheorem*{Conjecture}{Conjecture}
\newtheorem{Definition}[Theorem]{Definition}
\newtheorem{Proposition}[Theorem]{Proposition}
\newtheorem{Lemma}[Theorem]{Lemma}

\newtheorem*{Remark}{Remark}
\newtheorem{Corollary}[Theorem]{Corollary}

\newtheorem*{Claim}{Claim}

 \def\NN{{\mathbb N}} 

 \def\RR{{\mathbb R}} 
\def\TT{{\mathbb T}}

   \def\cN{{\cal N}} 
    \def\cU{{\cal U}}

    \def\cX{{\cal X}}

\def\triangleq{\stackrel{\triangle}{=}}

\def\dim{\operatorname{dim}}

\def\Sing{\operatorname{Sing}}
\def\orb{\operatorname{Orb}}

\def\e{{\varepsilon}}

\begin{document}

\title{Morse-Smale systems and horseshoes for three dimensional singular flows}

\author{Shaobo Gan \and Dawei Yang\footnote{D. Yang is the corresponding author. S. G. is supported by 973 project 2011CB808002, NSFC 11025101 and 11231001. D.Y. was partially supported by NSFC 11001101, Ministry of Education of P. R. China 20100061120098.}}

\date{}

\maketitle


\begin{abstract}
We prove that for every three-dimensional vector field, either it can be accumulated by Morse-Smale ones, or it can be accumulated by ones with a transverse homoclinic intersection of some hyperbolic periodic orbit in the $C^1$ topology.

\end{abstract}

\section{Introduction}

\subsection{The main result}

One of the main subjects in differentiable dynamical systems is to describe the dynamics of ``most'' dynamical
systems. These theories were established in the last century. See \cite{Ano05} for instance. An important progress is due to Peixoto \cite{Pei62}:
\begin{Thm}[Peixoto]
Assume that $M^2$ is a closed surface. A $C^1$ vector field on $M^2$ is $C^1$ structurally stable vector field iff it is Morse-Smale and every
vector field could be accumulated by a structurally stable one in the $C^1$ topology.
%
\end{Thm}

Smale was interested in the generalization of Peixoto's result and he asked whether Morse-Smale vector fields are dense in the space of vector fields.
Soon, Levinson and Thom pointed out that Morse-Smale vector fields would not be dense (without a rigorous proof). See \cite{Ano05}. Smale
noticed the point and he constructed his famous horseshoe (for two dimensional diffeomorphisms or
three-dimensional vector fields) \cite{Sma65} which shows that the dynamics may be very complicated and
Morse-Smale systems would not be dense in the space of diffeomorphisms or vector fields. As in \cite[Page
16]{Ano05}: ``At that moment the world turned upside down \dots, and a new life began''.

Actually, Poincar\'e found an important phenomenon in his famous work \cite{Poi90} on celestial mechanics, which
was called ``doubly asymptotical solution''. Nowadays mathematicians call it \emph{transverse homoclinic
intersection}. Smale found that his horseshoe is closely related to transverse homoclinic intersections. Three classical results are known:
\begin{itemize}

\item Poincar\'e showed that transverse homoclinic intersections can survive under small perturbations.
Moreover, if a system has one transverse homoclinic intersection, then it has infinitely many transverse
homoclinic intersections \cite{Poi90}.

\item Birkhoff showed that if a plane system has one transverse homoclinic intersection, then it has infinitely many hyperbolic periodic orbits \cite{Bir35}.

\item Smale proved that the existence of transverse homoclinic intersection is equivalent to the existence of horseshoe \cite{Sma65}.

\end{itemize}

Hence there are two kinds of typical dynamical systems: Morse-Smale system or
a system with a horseshoe. Their dynamical behavior is quite different:
\begin{itemize}

\item The dynamics of Morse-Smale system is very simple: the chain recurrent set of a Morse-Smale system is a set containing finitely many hyperbolic periodic orbits or singularities. The topological entropy is robustly zero.

\item The dynamics of a system with a horseshoe is very complicated: its chain recurrent set contains a non-trivial basic set with dense periodic orbits. The topological entropy is robustly positive.

\end{itemize}

Is there other typical dynamics beyond the above two ones? Palis formulated the idea for
diffeomorphisms, and he conjectured that
\begin{Conjecture}[Palis \cite{Pal00,Pal05,Pal08}]
Every system can be approximated either by Morse-Smale systems or by systems exhibiting a horseshoe (non-trivial hyperbolic
basic set).

\end{Conjecture}

In this paper, we manage to prove such kind of results for three dimensional vector fields.

\begin{theoremalph}[Main Theorem]\label{Thm:main}

Every three dimensional vector field can be $C^1$ approximated by Morse-Smale ones or by ones exhibiting a Smale horseshoe. In other words, Morse-Smale vector fields and vector fields with a non-trivial hyperbolic set (Smale horseshoe) form a dense open set in the space of $C^1$ three dimensional vector fields.

\end{theoremalph}

Important progress has been made for the conjecture of Palis for diffeomorphisms: in $C^1$ topology, Pujals-Sambarino \cite{PuS00} proved it for two-dimensional diffeomorphisms (as a corollary of a stronger result); Bonatti-Gan-Wen \cite{BGW07} gave a prove for three-dimensional diffeomorphisms; and finally Crovisier \cite{Cro10} proved the conjecture for \emph{any}-dimensional diffeomorphisms.

Comparing with diffeomorphism case, singularities of vector fields bring more difficulties. This prevents one to use some techniques of diffeomorphisms to singular vector fields, such as
Crovisier's central model. By considering the sectional Poincar\'e maps of the flows, sometimes one can get some (not all)
similar properties between $d$-dimensional vector fields and $(d-1)$-dimensional diffeomorphisms. But singular
vector field displays different dynamics, e.g., the famous \emph{Lorenz
attractor} \cite{Lor63}. In the spirit of Lorenz attractor, \cite{ABS77,Guc76,GuW79} constructed \emph{geometric
Lorenz attractor} in a theoretical way. Roughly, geometric Lorenz attractor is a robust attractor of
three-dimensional vector field, and it contains a hyperbolic singularity which is accumulated by hyperbolic
periodic orbits in a robust way.

Lorenz attractor is not hyperbolic because of the existence of singularity. On the other hand, Ma\~n\'e \cite{Man82} showed
that a robust attractor of a surface diffeomorphism is hyperbolic. This implies that the dynamics of 3-dimensional singular flows
are different from 2-dimensional diffeomorphisms.
Morales-Pacifico-Pujals \cite{MoP03,MPP98,MPP04} studied geometric Lorenz attractor in an abstract
way. They found the right concept, i.e., \emph{singular hyperbolicity}, to describe the hyperbolicity of Lorenz attractor, and they proved that a robust transitive set of a three-dimensional vector field is singular hyperbolic. But the dynamics of singular hyperbolic set are not as clear as hyperbolic set. For instance, there is no shadowing lemma of Anosov-Bowen type.

The dynamics of general transitive sets with singularities for three-dimensional vector fields are even more unclear for us than singular hyperbolic sets, even if the transitive sets have some dominated splitting with respect to the linear Poincar\'e flow. These are main difficulties that we encounter. For a non-trivial transitive set \emph{without singularities} of a generic any-dimensional vector field, one can adapt Crovisier's central model \cite{Cro10} to get a transverse homoclinic intersection of a hyperbolic
periodic orbit.

Let us be more precise. Let $M^d$ be a $d$-dimensional $C^\infty$ compact Riemannian manifold without boundary. Denote by ${\cal X}^1(M^d)$ the space of $C^1$ vector fields on $M^d$. Given $X\in{\cal X}^1(M^d)$, denote by $\phi_t=\phi^X_t$ the $C^1$ flow generated by $X$ and $\Phi_t={\rm d}\phi_t: TM^d\to TM^d$ the tangent flow on the tangent bundle $TM^d$. If $X(\sigma)=0$, then $\sigma$ is called a \emph{singularity} of $X$. Other points are called \emph{regular}. Let ${\rm
Sing}(X)$ be the set of singularities of $X$. For a regular point $p$, if $\phi_t(p)=p$ for some $t>0$, then $p$ is called \emph{periodic}. Let ${\rm Per}(X)$ be the set of periodic points of $X$. If
$x\in{\rm Sing}(X)\cup{\rm Per}(X)$, then $x$ is called a \emph{critical point} of $X$ and $\orb(x)$ is called a \emph{critical orbit} or \emph{critical element} of $X$.

For an invariant set $\Lambda$ and a $\Phi_t$-invariant bundle $E\subset T_\Lambda M^d$, we say that $E$ is \emph{contracting} (w.r.t. the tangent flow $\Phi_t$) if there are constants $C\ge 1, \lambda>0$ such that $\|\Phi_t|_{E(x)}\|\le C{\rm e}^{-\lambda t}$ for every $x\in \Lambda$ and $t\ge 0$; we say that $E$ is \emph{expanding} if it is contracting for $-X$.

An invariant set $\Lambda$ of $X$ is \emph{hyperbolic} if $TM^d$ has
continuous $\Phi_t$-invariant splitting
$$
T_{\Lambda}M^d= E^s\oplus \langle X\rangle\oplus E^u
$$
(where the fibre $\langle X(x)\rangle$ at $x$ is 0-dimensional or 1-dimensional according to
$x$ is a singularity or not), such that $E^s$ is contracting and $E^u$ is expanding.
If $\dim E^s$ is independent of $x\in\Lambda$, then $\dim E^s$ is called the \emph{index} of $\Lambda$.

For a critical point $x$, if $\orb(x)$ is a hyperbolic set, then we say that $x$ or $\orb(x)$ is hyperbolic. One
can define its index as the index of the hyperbolic set $\orb(x)$.

Recall that a $C^1$ vector field $X$ is \emph{Morse-Smale} if the non-wandering set $\Omega(X)$ of $X$ consists
of only finitely many hyperbolic critical elements and their stable and unstable manifolds intersect
transversely. We use ${\cal MS}$ to denote the set of Morse-Smale vector fields in ${\cal X}^1(M^d)$. For a
hyperbolic periodic orbit $\gamma$, define
\begin{eqnarray*}
W^s(\gamma)&=&\{x\in M^d: \lim_{t\to+\infty}d(\phi_t(x),\gamma)=0.\}\\
W^u(\gamma)&=&\{x\in M^d: \lim_{t\to-\infty}d(\phi_t(x),\gamma)=0.\}
\end{eqnarray*}

We know (\cite{HPS77}) that $W^s(\gamma)$ and $W^u(\gamma)$ are submanifolds, which are called the stable and unstable manifolds of $\gamma$. If $W^s(\gamma)\pitchfork W^u(\gamma)\setminus\gamma\neq\emptyset$, then one says that $\gamma$ has a \emph{transverse homoclinic orbit}\footnote{Singularities cannot have transverse homoclinic intersections.}. One says that \emph{$X$ has a transverse homoclinic
orbit} if for some hyperbolic periodic orbit $\gamma$ of $X$, $\gamma$ has a transverse homoclinic orbit. Recall that: Birkhoff-Smale theorem asserts that the existence of transverse homoclinic orbits is equivalent to the
existence of Smale's horseshoe (non-trivial hyperbolic basic set). We denote:
$${\cal HS}=\{X\in{\cal X}^1(M^d):~X~\textrm{has a transverse homoclinic orbit}\}.$$
One can restate Theorem A as:

\indent{\bf ${\cal MS}\cup{\cal HS}$ is open and dense in ${\cal X}^1(M^3)$.}

\subsection{More on three-dimensional flows}

Given a vector field $X$, let $\phi_t$ be the flow generated by $X$. For any $\varepsilon>0$,
$\{x_0,x_1,\cdots,x_n\}$ is called an \emph{$\varepsilon$-chain}(or $\varepsilon$-pseudo-orbit) from $x_0$ to $x_n$ if there are $t_i\ge 1$ such that
$d(\phi_{t_i}(x_i),x_{i+1})<\varepsilon$ for any $0\le i\le n-1$. For $x,y\in M^d$, one says that $y$ is
\emph{chain attainable} from $x$ if for any $\varepsilon>0$, there is an $\varepsilon$-chain
from $x$ and $y$. If $x$ is chain attainable from itself, then $x$ is called a
\emph{chain recurrent point}. The set of chain recurrent points is called \emph{chain recurrent set} of $X$, and denoted by ${\rm CR}(X)$. Chain bi-attainability is a closed equivalence relation in ${\rm CR}(X)$. For each $x\in{\rm CR}(X)$, the equivalent class containing $x$ is called the chain recurrent class of $x$, and denoted by $C(x)$ or $C({\rm Orb}(x))$. These are standard by Conley's theory \cite{Con78}.

Bonatti and Crovisier \cite{BoC04} extended the $C^1$ connecting lemma to pseudo-orbits. An application of their result gave a useful classification of chain recurrent classes for \emph{$C^1$-generic} vector fields\footnote{Their results are stated for diffeomorphisms. The proof can be adapted to the case of vector field by a parallel way.}: if a chain recurrent class contains a periodic orbit, then it is the \emph{homoclinic class} of this periodic orbit; otherwise, it is called an \emph{aperiodic class}. Here, the \emph{homoclinic class} of a hyperbolic periodic orbit is defined to be the closure of all transverse homoclinic orbits of this periodic orbit.

Hyperbolic periodic orbits may have non-transverse homoclinic intersections, which are called \emph{homoclinic tangencies}. Newhouse \cite{New70,New74,New79} studied the bifurcations of homoclinic tangencies crucially, which generate rich dynamics. Newhouse phenomena give typical dynamics beyond uniformly hyperbolic dynamics. There are many results for diffeomorphisms far away from ones with a homoclinic tangency. One can see the introduction of \cite{CSY11}.

In this work, we can prove that every \emph{non-trivial} chain recurrent class is a homoclinic class for $C^1$ generic vector fields which are far away from homoclinic tangencies. Here, a chain recurrent class is called \emph{non-trivial} if it is not reduced to be a critical orbit.
\begin{theoremalph}\label{Thm:homoclinicclass}
There is a dense ${G}_\delta$ set ${\cal R}\subset {\cal X}^1(M^3)$ such that, for every $X\in{\cal R}$, if $X$ cannot be accumulated by ones with a homoclinic tangency, then every non-trivial chain recurrent class of $X$ is a homoclinic class.

\end{theoremalph}

Theorem~\ref{Thm:homoclinicclass} is stronger than Theorem~\ref{Thm:main}. We will see this point in Section~\ref{Sec:reduction}.

An important conjecture made by Palis for surface diffeomorphisms is: every two-dimensional diffeomrophism can be accumulated either by ones with a homoclinic tangency, or by uniformly hyperbolic ones. This was proved by Pujals-Sambarino \cite{PuS00} in the $C^1$ topology. For three-dimensional vector fields, as mentioned in \cite{Pal91}, excluding homoclinic tangencies and uniform hyperbolic systems, the typical dynamics may include the homoclinic orbits of singularities or Lorenz-like attractors. Arroyo-Rodriguez \cite{ArR03} proved the conjecture of Palis if homoclinic orbits of singularities were involved for three-dimensional vector fields. It is still an open problem about the density of Lorenz-like attractors or repellers beyond uniform hyperbolicity and homoclinic bifurcations of periodic orbits (even in the $C^1$ topology).

%

Morales-Pacifico-Pujals \cite{MPP98,MPP04} defined what is ``Lorenz-like'' in a dynamical way. In \cite{MoP03}, Morales-Pacifico gave the notion of
singular Axiom A without cycle. Let's be more precise.

We say that a continuous invariant splitting $T_\Lambda M^d=E\oplus F$ w.r.t. the tangent flow over a compact invariant set $\Lambda$ is a \emph{dominated splitting} with respect to the tangent flow $\Phi_t$ if  there are constants $C\ge 1, \lambda>0$, such that
$\|\Phi_t|_{E(x)}\|\|\Phi_{-t}|_{F(\phi_t(x))}\|\le C{\rm e}^{-\lambda t}$ for every $x\in \Lambda$ and $t\ge 0$. For a compact invariant set $\Lambda$, we say that $\Lambda$ admits a \emph{partially hyperbolic splitting} if there is a continuous invariant splitting $T_\Lambda M^d=E^s\oplus E^c\oplus E^u$ w.r.t. $\Phi_t$, where $E^s$ is contracting, $E^u$ is expanding, and both $E^s\oplus (E^c\oplus E^u)$ and $(E^s\oplus E^c)\oplus E^u$ are dominated splittings. In the above definition, $E^s$ or $E^u$ is allowed to be trivial.

\begin{Definition}
A transitive set $\Lambda$ of $X\in{\cal X}^1(M^3)$ is called a \emph{singular hyperbolic attractor} if
\begin{enumerate}

\item There is a neighborhood $U$ of $\Lambda$ such that
$$\Lambda=\bigcap_{t\ge 0}\phi_t(U).$$

\item $\Lambda$ contains a singularity of index $2$, and every singularity in $\Lambda$ has index 2.

\item $\Lambda$ admits a partially hyperbolic splitting $T_\Lambda M^3=E^s\oplus E^{cu}$, where $\dim E^s=1$ and $E^{cu}$ area-expanding: there are constants $C\ge 1$, $\lambda>0$ such that for any $x\in\Lambda$ and for any $t\ge 0$, one has $|\det \Phi_{-t}|_{F(x)}|\le C {\rm e}^{-\lambda t}$.

\end{enumerate}

$\Lambda$ is called a \emph{singular hyperbolic repeller} if it is a singular hyperbolic attractor for $-X$.

\end{Definition}

One knows that the geometric Lorenz attractors as in \cite{Guc76,GuW79,ABS77} are singular hyperbolic attractors. $X\in{\cal X}^1(M^3)$ is called \emph{singular Axiom A without cycle} as in \cite{MoP03} if the chain recurrent set of $X$ contains only finitely many chain recurrent classes; moreover each chain recurrent class is a hyperbolic basic set, or a singular hyperbolic attractor, or a singular hyperbolic repeller.

Singular Axiom A flows include Lorenz-like flows. \cite{ArR03,MoP03} asked whether singular Axiom A flows and flows with a homoclinic tangency are typical phenomena.
\begin{Conjecture}
Every $X\in{\cal X}^1(M^3)$ can be accumulated either by vector fields with a homoclinic tangency, or by singular Axiom A vector fields without cycle.

\end{Conjecture}

We get some progress on this conjecture.

\begin{theoremalph}\label{Thm:dominationtosingular}
There is a dense ${G}_\delta$ set ${\cal R}\subset {\cal X}^1(M^3)$ such that for any $X\in{\cal R}$ and $\sigma$, if the chain recurrent class $C(\sigma)$ is nontrivial and admits a dominated splitting $T_{C(\sigma)} M^3=E\oplus F$ w.r.t. the tangent flow, then $C(\sigma)$ is a singular hyperbolic attractor or a singular hyperbolic repeller.

\end{theoremalph}

Notice that in a joint work with C. Bonatti \cite{BGY11}, we proved this result by adding an additional assumption that: $C(\sigma)$ contains a periodic orbit. So, according to this result, to prove the above theorem, we assume that $C(\sigma)$ contains no periodic orbits. Then we can consider the return map of (singular) cross-sections. After a sequence of perturbations, we will get a contradiction. This is one of the main points of this work.

\subsection{Entropy of flows}
The entropy of a flow is defined to be the entropy of the time-one map of the flow. The definition meets some
problems: there are two topological equivalent flows, one has zero entropy and the other one has positive
entropy. This pathology happens because of the existence of singularities. See \cite{Ohn80,Tho90,SYZ09} for references. But using the main theorem of the paper, we can prove:

\begin{Theorem}
There is a dense open set ${\cal U}\subset {\cal X}^1(M^3)$ such that for any $X\in{\cal U}$, for any $Y$
topological equivalent to $X$, one has $h(X)=0$ iff $h(Y)=0$.
\end{Theorem}

\begin{proof}
By Theorem~\ref{Thm:main}, there is a dense open set ${\cal U}\subset{\cal X}^1(M^3)$ such that for any
$X\in{\cal U}$, either $X$ is Morse-Smale, or $X$ has a non-trivial hyperbolic basic set. Thus for any
$X\in{\cal U}$, one has
\begin{itemize}

\item either $h(X)=0$, then $X$ is Morse-Smale, thus for any $Y$ which is topological equivalent to $X$, for any
point $x$, the forward iteration and backward iteration of $x$ with respect to $\phi_t^Y$ go to a critical element. This feature implies that $h(Y)=0$.

\item or $h(X)>0$, then $X$ has a non-trivial hyperbolic basic set.  Since for non-singular equivalent flows $X, Y$, $h(X)>0$ iff $h(Y)>0$. We have that $h(Y)>0$.
\end{itemize}

\end{proof}

For the relationship between zero-entropy vector fields and Morse-Smale vector fields, one has

\begin{Theorem}

If a three-dimensional vector field $X\in{\cal X}^1(M^3)$ can be accumulated by $C^1$ robustly zero-entropy vector fields, then it can be $C^1$ accumulated by Morse-Smale vector fields.

\end{Theorem}

\begin{proof}
By the assumptions, for any $C^1$ neighborhood ${\cal U}$ of $X$, there is an open set ${\cal V}\subset {\cal U}$ such that every vector field $Y$ in ${\cal V}$ has zero-entropy. By Theorem~\ref{Thm:main}, by reducing ${\cal V}$ if necessary, one can assume that for every $Y\in\cal V$, either it is Morse-Smale, or it has a non-trivial hyperbolic basic set. Since $Y\in\cal V$ can only have zero-entropy, one has $Y$ is Morse-Smale. This ends the proof.

\end{proof}

\subsection{Organization of this paper}

The proof of the theorems is not short. Especially for vector fields, they involve more definitions and
notations.

\begin{enumerate}

\item In section~\ref{Sec:differentflows}, we give various kinds of definitions of flows associated to a vector field $X$. Liao defined these flows in a very abstract way. In fact, all these flows have their geometric meanings. We will deal with dominated splittings for the tangent flow and the linear Poincar\'e flow. For the estimations stated in this section (which is crucial for singular flow), Liao had very original ideas by a sequence of papers. We restate some of them and give the proof ourself. The proof is more intuitive.

\item In Section~\ref{Sec:generic}, we will study generic properties by connecting lemmas, ergodic closing lemma. \cite{BoC04} gave a $C^1$ connecting lemma for pseudo-orbits which helps us to obtain generic results for chain recurrent classes. We notice that Lyapunov stable chain recurrent classes with a critical element are robust for generic vector fields. The proof is not difficult, but it
opens a new door: Lyapunov stable chain recurrent class will survive under generic small perturbations.

\item In Section~\ref{Sec:reduction} we first give the proof of Theorem~\ref{Thm:main} by assuming Theorem~B. Then we reduce the proofs of Theorem~B and Theorem~C to several sub-results.

\item In Section~\ref{Sec:lorenz-like}, we prove that the Lyapunov stable chain recurrent class admits a partially hyperbolic splitting $E^{ss}\oplus E^{cu}$ and every singularity in the chain recurrent class is Lorenz-like. In this section, the main novelty of this paper is that we use some uniform estimation on vector fields away from homoclinic tangencies and a suitable application of Liao's shadowing lemma.

\item In Section~\ref{Sec:birthhomo}, we prove that every nontrivial partially hyperbolic chain recurrent class with singularities contains periodic orbits for generic three-dimensional vector fields. We notice that it contains a periodic orbit iff it is singular hyperbolic. When the chain recurrent class is not singular hyperbolic, it is not singular hyperbolic robustly. Then by a sequence of perturbations, the continuation of the chain recurrent class intersects the closure of the basin of some sink. This implies that the continuation of the chain recurrent class is not Lyapunov stable. We can get a contradiction by Lemma~\ref{Lem:stablepropertyofquasiattractor}. The difficulty we encounter is similar to the case of one-dimensional endomorphisms with singularities, where ``singularities'' means that the points where the endomorphisms fail to be a local diffeomorphisms. For flows, for every central unstable curve in the cross-section, in principle we will know that its length will grow near the local stable manifold of the singularities. But when it is cut by local stable manifold of singularities, its image under return map will be disconnected. This facts make the dynamics unclear. We have a good control in this section for this phenomenon.
\end{enumerate}

\section{Flows associated to a vector field and dominated splittings}\label{Sec:differentflows}

\subsection{Tangent flow, linear Poincar\'e flow and their extensions}\label{Sub:flows}

Given $X\in{\cal X}^1(M^d)$, $X$ generates a $C^1$ flow $\phi_t:M^d\to M^d$, and the \emph{tangent flow}
$\Phi_t={\rm d}\phi_t:TM^d\to TM^d$. Denote by $\pi:~TM^d\to M^d$ the bundle projection.

Denote the normal bundle of $X$ by
$$
{\cal N}={\cal N}^X=\bigcup_{x\in M^d\setminus\Sing (X)} {\cal N}_x,
$$
where ${\cal N}_x$ is the orthogonal complement of the flow direction $X(x)$, i.e.,
$$
{\cal N}_x=\{v\in TM^d: v\perp X(x)\}.
$$
Given $x\in M^d\setminus\Sing(X)$ and $v\in {\cal N}_x$, $\psi_t(v)$ is the orthogonal projection of $\Phi_t(v)$ on
${\cal N}_{\phi_t(x)}$ along the flow direction, i.e.,
$$
\psi_t(v)=\Phi_t(v)-\frac{\langle \Phi_t(v), X(\phi_t(x))\rangle}{|X(\phi_t(x))|^2}X(\phi_t(x)),
$$
where $\langle\cdot, \cdot\rangle$ is the inner product on $T_xM$ given by the Riemannian metric.

By the definition, $\|\psi_t\|$ is uniformly bounded for $t$ in any bounded interval although it is just defined on the regular set which is not compact in general.

This flow could also be defined in a more general way by Liao \cite{Lia89}. \cite{LGW05} used the terminology of ``extended linear Poincar\'e flow''.
For every point $x\in M^d$, one could define the sphere fiber at $x$ by
$$S_x M^d=\{v:~v\in T_x M^d,~|v|=1\}.$$
The sphere bundle $SM^d=\bigcup_{x\in M^d} S_x M^d$ is compact. One can define the \emph{unit tangent flow}
$$\Phi^I_t:~SM^d\to SM^d$$
as
$$\Phi^I_t(v)=\frac{\Phi_t(v)}{|\Phi_t(v)|}$$
for any $v\in SM^d$.

Given a compact invariant set $\Lambda$ of $X$, denote
$$\widetilde\Lambda= {\rm Closure}\left(
\bigcup_{x\in \Lambda\setminus{\rm Sing}(X)}\frac{X(x)}{|X(x)|}\right).
$$
in $SM^d$. Thus the essential difference between $\widetilde\Lambda$ and $\Lambda$ is on the singularities. We have more information on $\widetilde\Lambda$: it tells us how regular points in $\Lambda$ accumulate singularities.

For any $x\in M^d$, and any two orthogonal vectors $v_1,v_2\in T_x M^d$, if $|v_1|\neq0$, one can define

$$\chi_t(v_1,v_2)=(\Phi_t(v_1),\Phi_t(v_2)-\frac{\langle \Phi_t(v_1),\Phi_t(v_2)\rangle}{|\Phi_t(v_1)|^2}\Phi_t(v_1)).$$

By definition the two components of $\chi_t$ are still orthogonal. If one denotes $$\chi_t=({\rm
proj}_1(\chi_t),{\rm proj}_2(\chi_t)),$$ for any regular point $x\in M^d$ and any vector $v\in {\cal N}_x$, one
has
$$\psi_t(v)={\rm proj}_2\chi_t(X(x),v).$$

One can normalize the first component of $\chi_t$: for any $x\in M^d$, and any two orthogonal vectors $v_1,v_2\in T_x M^d$, if $|v_1|=1$, one can define
$$\chi_t^\#(v_1,v_2)=(\Phi_t^I(v_1),~\Phi_t(v_2)-\frac{\langle \Phi_t(v_1),\Phi_t(v_2)\rangle}{|\Phi_t(v_1)|^2}\Phi_t(v_1)).$$
$\chi_t^\#$ is also a continuous flow, and for any regular point $x$ and any $v\in{\cal N}_x$, one has
$$\chi_t^\#(\frac{X(x)}{|X(x)|},v)=(\Phi_t^I(\frac{X(x)}{|X(x)|}),~\psi_t(v)).$$

By the continuity of $\chi_t$, one can extend the definition of $\psi_t$ ``to singularities'': for any $u\in \widetilde\Lambda$, one defines $\widetilde{\cal N}_u=\{v\in T_{\pi(u)} M:~\langle u,v\rangle=0\}$. $\widetilde{\cal N}$ is a $(d-1)$-dimensional vector bundle on the base space $\widetilde\Lambda$ (for a formal discussion, see \cite{LGW05}). For $u\in\widetilde\Lambda$ and $v\in \widetilde{\cal N}_{u}$, one can define $\widetilde {\psi}_t(v)={\rm proj}_2\chi_t(u,v)$. By the definition we know that ${\rm proj}_2\chi_t$ is a continuous flow defined on $\widetilde \Lambda$. Thus, $\widetilde\psi_t$ can be viewed as a compactification of $\psi_t$.

\subsection{Scaled linear Poincar\'e flow $\psi_t^*$ and Liao's estimation}
For our purpose, we need another flow $\psi_t^*:{\cal N}\to{\cal N}$ (called {\em scaled linear Poincar\'e
flow}). Given $x\in M^d\setminus\Sing(X)$, and $v\in {\cal N}_x$,
$$
\psi_t^*(v)=\frac{|X(x)|}{|X(\phi_t(x))|}\psi_t(v)=\frac{\psi_t(v)}{\|\Phi_t|_{\langle X(x)\rangle}\|}.
$$
In our case,
this scaled linear Poincar\'e flow will help us to overcome some difficulties produced by singularities. It gives uniform estimations on some non-compact sets.

\begin{Lemma}\label{Lem:starbounded}
For any $\tau>0$, there is $C_\tau>0$ such that for any $t\in[-\tau,\tau]$,
$$\|\psi_t^*\|\le C_\tau,$$
where $\|\psi_t^*\|=\sup\{|\psi_t^*(v)|: v\in {\cal N} {\rm \ and\ } |v|=1\}.$

\end{Lemma}
\begin{proof}
For $t\in[-\tau,\tau]$, first we know that $\Phi_t$ is uniformly bounded from 0 and $\infty$; from the definition of the linear Poincar\'e flow, we know that $\psi_t$ is uniformly bounded. Thus, $\psi_t^*(x)=\psi_t(x)/\|\Phi_t|_{\left< X(x) \right>}\|$ is uniformly bounded.
\end{proof}

For each $\beta>0$, one can define the
normal manifold $N_{x}(\beta)$ of $x$ as the following:
$$N_{x}(\beta)=\exp_x({\cal N}_x(\beta)),$$
where ${\cal N}_x(\beta)=\{v\in{\cal N}_x: |v|\le\beta\}.$
Take $\beta_*>0$ small enough such that for any $x\in M$, $\exp_x$ is a diffeomorphism from ${\cal N}_x(\beta_*)$ to its image $N_{x}({\beta}_{*}).$

To study the dynamics in a small neighborhood of a periodic orbit of a vector field, Poincar\'e defined the sectional return map of a cross section of a periodic point. By generalizing this idea to every regular point, one can define the sectional Poincar\'e map for any two points in the same regular orbit. For our convenience, we define the sectional Poincar\'e map in the normal bundle.

Given $T>0$ and $x\in M^d\setminus{\rm Sing}(X)$, the flow $\phi_t$ defines a local holonomy map $P_{x,\phi_T(x)}$ from $N_x(\beta_*)$ to $N_{\phi_T(x)}(\beta_*)$ in a small neighborhood of $x$. Hence its lift map in the normal bundle gives a map ${\cal P}_{x,\phi_T(x)}:~U\to {\cal N}_{\phi_T(x)}(\beta_*)$, where $U$ is a small neighborhood of $x$ in ${\cal N}_x(\beta_*)$ and ${\cal P}_{x,\phi_T(x)}=\exp_{\phi_T(x)}^{-1}\circ P_{x,\phi_T(x)}\circ \exp_x$. Note that when $T'>T>0$, the domain of ${\cal P}_{x,\phi_{T'}(x)}$ is contained in the domain of ${\cal P}_{x,\phi_T(x)}$.
Usually the size of $U$ depends on the orbit of $x$: if $x$ is very close to a singularity, then $U$ should be very small. But after scaling, we have the following uniform estimation for the relative size of $U$.

\begin{Lemma}\label{Lem:domainofPoincare}
Given $X\in{\cal X}^1(M^d)$ and $T>0$, there is $\beta_T>0$ such that for any regular point $x$, ${\cal P}_{x,\phi_T(x)}$ is well defined on ${\cal N}_{x}(\beta_T |X(x)|)$.
\end{Lemma}

\begin{proof}
After taking an orthonormal basis $\{e_1,e_2,\cdots, e_d\}$ of $T_xM$, we get a coordinate system:
$$
{\rm Exp}_x: {\mathbb{R}}^d\to M^d,
$$
such that
$$
{\rm Exp}_x(z)=\exp_x(\sum_{i=1}^dz_ie_i).
$$
In the coordinate, the flow generated by the vector field $X$ satisfies the following differential equation:
$$
\frac{{\rm d} z}{{\rm d} t}=\hat{X}(z),
$$
where
$$
\hat{X}(z)=D{\rm Exp}_x^{-1}\circ X({\rm Exp}_x(z)).
$$
Note that $\hat{X}$ depends on $x\in M^d$.

Assume that $\beta_*$ is small enough such that for any $x\in M^d$, $|z|\le \beta_*$, and any two unit vectors $v$ and $w$,
$$
0.999<| D{\rm Exp}_x(z)v|<1.001,~~~|\angle(D{\rm Exp}_x(z)v,D{\rm Exp}_x(z)w)-\angle(v,w)|<0.001.
$$
We will prove the lemma in the local coordinate. Note that ${\rm Exp}_x(0)=x$ and
$|X(x)|=|\hat{X}(0)|$.

Denote by $$K=\sup_{x\in M^d, |z|\le\beta_*}\{|\hat{X}(z)|, \|D\hat{X}(z)\|\}.$$
Since $D{\rm Exp}_x$ and $D{\rm Exp}_x^{-1}$ are uniformly bounded with respect to $x$, we have
$$K<\infty.$$
Assume $\beta_0<\min\{1,\beta_*\}/(1000K)$. For any regular point $x$, take $e_1=X(x)/|X(x)|$. Then
$\hat{X}(0)=(|X(x)|, 0, \cdots, 0)$. For $|z|\le\beta_0|X(x)|$, according to the mean-value theorem, there exists $\xi\in{\mathbb R}^d, |\xi|\le\beta_0|X(x)|$,
$$
|\hat{X}(z)|=|\hat{X}(0)+D\hat{X}(\xi)z|\ge |X(x)|-K\beta_0|X(x)|\ge 0.999|X(x)|>0.
$$
This implies that $N_{x}(\beta_0|X(x)|)\cap{\rm Sing}(X)=\emptyset.$

Denote by $\hat{\phi}_t(z)=(\hat{\phi}^1, \cdots, \hat{\phi}^d)$ the solution of $\hat{X}(z)$ such that $\hat{\phi}_0(z)=z$. Then for $0\le t\le \beta_0$, if $|z|\le \frac{\beta_0}{3}|X(x)|$ and $|\phi_s(z)|\le\beta_0|X(x)|$ for $s\in[0,t]$, we have
$$
|\hat{\phi}_t(z)|=|z+\int_0^t \hat{X}(\hat{\phi}_t(z)){\rm d}t|\le
(\frac{\beta_0}{3}+1.001t)|X(x)|.
$$
Let $t$ be the time such that $|\hat{\phi}_s(z)|\le\beta_0|X(x)|$ for $s\in[0,t)$ and
$|\hat{\phi}_t(z)|=\beta_0|X(x)|.$ From the above estimation, we have that
$$
t\ge \frac 23\beta_0/1.001>\frac 12\beta_0.
$$

By reducing $\beta_0$ if necessary, for $|z|\le\beta_0|X(x)|$,
$$\sup_{t\in(-\beta_0,\beta_0)}\frac{|\hat{X}(z)|}{|\hat{X}(\hat{\phi}_t(z))|}<\frac{1}{1000K},~~~\sup_{t\in(-\beta_0,\beta_0)}\angle(\hat{X}(z),~\hat{X}(\hat{\phi}_t(z)))<\frac{1}{1000K}$$

Since {\bf $e_1=X(x)/|X(x)|$}, ${\cal N}_x=\{z: z_1=0\},$ and $\hat{X}(0)=(|X(x)|,0,\cdots,0),$
we have that for $|z|\le\beta_0|X(x)|$, $\hat{X}_1(z)\in[0.999|X(x)|,1.001|X(x)|]$.

\begin{Claim}
For $0<r<\beta_0$, for any $z=(0,y)$, $y\in{\mathbb R}^{d-1}$, $|y|\le r|X(x)|/2$,  for any $t\in[\beta_0/3, 2\beta_0/3]$, there exists a unique $\tau=\tau(t,y)\in(0,\beta_0]$ such that $\hat{\phi}_\tau(0,y)\in  \hat{N}_{\phi_t(x)}(r)$, where
$$
\hat{N}_x(r)={\rm Exp}^{-1}({N}_x(r)).
$$
\end{Claim}

\begin{proof}[\bf Proof of Claim.]
Since $|z|<1000\beta_0 K$, $t\in[0,\beta_0]$ and $\beta_0$ is small, $\mathrm{Exp}_x$ is almost an isometry, and ${N}_{\phi_t(x)}$ is the graph of a map $f=f_t: \mathbb{R}^{d-1}\to {\mathbb{R}}$ with $|\frac{\partial{f}}{\partial y}|<0.001$.
Let $z=(0,y)$ with $|y|\le\beta_0|X(x)|$. Then
$$
{\rm proj}_1\hat{\phi}_{\beta_0}(z)=\int_0^{\beta_0}\hat{X}_1(\hat{\phi}_t(z)){\rm d}t\ge 0.999|X(x)|\beta_0.
$$
This means that $z$ and $\hat{\phi}_{\beta_0}(z)$ are on the different sides of the graph of $f_t$ for $t\in [\beta_0/3, 2\beta_0/3]$, from which one can define $\tau(t,y)$. Notice that $f$ is $C^1$ and $X$ is $C^1$, we have that $\tau(t,y)$ is $C^1$ w.r.t. $t$ and $y$.
\end{proof}

For any $T\ge \beta_0$, let $n=[3T/\beta_0]$ and a partition
$$
0=t_0<t_1<\cdots<t_{n-1}<t_n=T,
$$
such that $t_i=i\beta_0/3, i=0,1,\cdots, n-1$. Then we have $t_n-t_{n-1}\in[\beta_0/3,2\beta_0/3].$ Then we can define $\beta_T$ inductively.

\end{proof}

%
%
%
%

\begin{Lemma}\label{Lem:uniformcontinuous}
Let $X\in{\cal X}^1(M^d)$ and $T>0$. By reducing $\beta_T>0$  as in Lemma~\ref{Lem:domainofPoincare} if necessary, for any $x\in M^d\setminus{\rm Sing}(X)$, for the sectional Poincar\'e map
$${\cal P}_{x,\phi_T(x)}: {\cal N}_{x}(\beta_T|X(x)|)\to {\cal N}_{\phi_T(x)}(\beta_*),$$
$D{\cal P}_{x,\phi_T(x)}(y)$ is uniformly continuous in the following sense: for any $\epsilon>0$ there exists $\delta\in(0, \beta_T]$ such that for any $x\in M^d\setminus {\rm Sing}(X)$ and $y,y'\in {\cal N}_x (\beta_T|X(x)|)$, if $|y-y'|\le\delta|X(x)|$, then
    $$
    |D\mathcal{P}_{x,\phi_T(x)}(y)-D\mathcal{P}_{x,\phi_T(x)}(y')|<\epsilon.
    $$
    (Note that $D\mathcal{P}_{x,\phi_T(x)}(0)=\psi_T|_{{\mathcal N}_x}$.)
    And hence there exists $K_T>0$ (independent of $x$) such that
    $$
    |D\mathcal{P}_{x,\phi_T(x)}|\le K_T.
    $$
\end{Lemma}

\begin{proof}

We will still use the notations and terminologies as in the proof of Lemma~\ref{Lem:domainofPoincare}. We first assume that $T<\beta_0$.
Notice that we are in a local Euclidean coordinate, $N_z={\cal N}_z$ for each regular point $z$.

In the local coordinate, we assume that the vector field $\hat{X}$ has the following form $\hat{X}(x,y)=(f(x,y),g(x,y))$, where $x\in\RR^1$, $y\in\RR^{d-1}$,
$f:\RR^d\to\RR^1$ and $g:\RR^{d}\to\RR^{d-1}$ are continuously $C^1$ maps such that
\begin{itemize}

\item $\hat{X}(0,0)=(f(0,0),0)$, where $f(0,0)>0$.

\item for any $(x,y)$ in the local coordinate, one has $f(x,y)\in (0.999f(0,0), 1.001f(0,0))$ and $|g(x,y)|<f(0,0)/1000$.
\end{itemize}

The flow of $\hat{X}$ satisfies the following differential equations:
\[\begin{array}{ll}\
&\displaystyle\frac{\mathrm{d}x}{\mathrm{d}t}=  f(x,y), \\[0.2cm] &\displaystyle\frac{\mathrm{d}y}{\mathrm{d}t}=  g(x,y).
\end{array}\]

We assume that the solution of these differential equations is $(\widehat{\varphi}_t(x,y),\widehat{\psi}_t(x,y))$, where $\widehat{\varphi}_t(x,y):~\RR^1\times\RR^d\to\RR^1$ and $\widehat{\psi}_t(x,y):~\RR^1\times\RR^d\to\RR^{d-1}$.

Now we consider the expression of the sectional Poincar\'e map $\mathcal{P}_T:~\mathcal{N}_{x}(\beta_0|X(x)|)\to \mathcal{N}_{\phi_T(x)}(\beta_*)$ in this local coordinate. The local coordinate of $x$ is $(0,0)$. Thus $\mathcal{N}_{(0,0)}(\beta_*)\subset\{0\}\times \RR^{d-1}$. Now we consider $(\widehat{\varphi}_T(0,0),\widehat{\psi}_T(0,0))$, whose normal manifold is contained in a graph of an affine map $h:\RR^{d-1}\to \RR$ such that $|Dh|<0.001$.

For $y\in\RR^{d-1}$,
$$\mathcal{P}(0,y)=\mathcal{P}(y)=(\widehat{\varphi}_\tau(0,y),\widehat{\psi}_\tau(0,y))^T,$$
where the time function $\tau:~\RR^{d-1}\to\RR^1$ satisfies
$$\widehat{\varphi}_{\tau(y)}(0,y)=h(\widehat{\psi}_{\tau(y)}(0,y)).$$

By differentiating $y$ in the above equality, one has

$$\left.\frac{\partial \widehat{\varphi}}{\partial t}\right|_{t=\tau(y)} \frac{\partial \tau}{\partial y}+\frac{\partial \widehat{\varphi}}{\partial y}=\frac{\partial h}{\partial y}(\frac{\partial \widehat{\psi}}{\partial t}\frac{\partial \tau}{\partial y}+\frac{\partial \widehat{\psi}}{\partial y}).$$
Notice that in the above equality, $\partial \tau/\partial y$, $\partial \widehat{\varphi}/\partial y$ and $\partial h/\partial y$ are row vectors with $d-1$ elements, $\partial \widehat{\psi}/\partial t$ is a column vector with $(d-1)$ elements, $\partial \widehat{\psi}/\partial y$ is a $(d-1)\times (d-1)$ matrix.

By solving the above equality, one has

$$\frac{\partial \tau}{\partial y}=\frac{\frac{\partial h}{\partial y}\frac{\partial \widehat{\psi}}{\partial y}-\frac{\partial \widehat{\varphi}}{\partial y}}{\frac{\partial \widehat{\varphi}}{\partial t}-\frac{\partial h}{\partial y}\frac{\partial \widehat{\psi}}{\partial t}}.$$

Thus,

$$\frac{\partial \mathcal{P}}{\partial y}=\begin{pmatrix} \frac{\partial \widehat{\varphi}}{\partial t}\frac{\partial \tau}{\partial y}+\frac{\partial \widehat{\varphi}}{\partial y} \\ \frac{\partial \widehat{\psi}}{\partial t}\frac{\partial \tau}{\partial y}+\frac{\partial \widehat{\psi}}{\partial y}\end{pmatrix}=   \begin{pmatrix}   \frac{\partial \widehat{\varphi}}{\partial t}   \\  \frac{\partial \widehat{\psi}}{\partial t}   \end{pmatrix} \frac{\frac{\partial h}{\partial y}\frac{\partial \widehat{\psi}}{\partial y}-\frac{\partial \widehat{\varphi}}{\partial y}}{\frac{\partial \widehat{\varphi}}{\partial t}-\frac{\partial h}{\partial y}\frac{\partial \widehat{\psi}}{\partial t}}+\begin{pmatrix}   \frac{\partial \widehat{\varphi}}{\partial y}   \\  \frac{\partial \widehat{\psi}}{\partial y}   \end{pmatrix}. $$

By the expression of the differential equations, one has

$$\frac{\partial \mathcal{P}}{\partial y}=\widehat{X}(\mathcal{P}(y))\frac{\frac{\partial h}{\partial y}\frac{\partial \widehat{\psi}}{\partial y}-\frac{\partial \widehat{\varphi}}{\partial y}}{\frac{\partial \widehat{\varphi}}{\partial t}-\frac{\partial h}{\partial y}\frac{\partial \widehat{\psi}}{\partial t}}+\begin{pmatrix}   \frac{\partial \widehat{\varphi}}{\partial y}   \\  \frac{\partial \widehat{\psi}}{\partial y}   \end{pmatrix}$$
In another form,
$$\frac{\partial \mathcal{P}}{\partial y}=\frac{\widehat{X}\circ \mathcal{P}(y)}{\widehat{f}\circ \mathcal{P}(y)-\frac{\partial h}{\partial y}\widehat{g}\circ \mathcal{P} (y)}\left( \frac{\partial h}{\partial y}\frac{\partial \widehat{\psi}}{\partial y}-\frac{\partial \widehat{\varphi}}{\partial y}  \right)+\begin{pmatrix}   \frac{\partial \widehat{\varphi}}{\partial y}(0,y)   \\  \frac{\partial \widehat{\psi}}{\partial y}(0,y)   \end{pmatrix}.$$

Since $0.999<|f|/|\widehat{X}|<1.001$ and $|g|/|X|<0.001$, there is a uniform constant $\widehat{K}>0$ such that

$$\|\frac{\partial \mathcal{P}}{\partial y}\|\le \widehat{K}.$$

For any $\varepsilon>0$, since the tangent flow $\Phi_t$ is uniformly continuous, there is $\delta>0$ such that for any $y,y'\in N_x(\beta_T|X(x)|)$, if $|y-y'|\le\delta|X(x)|$, one has
$$
\|\frac{\partial \widehat{\varphi}}{\partial y}(0,y)-\frac{\partial \widehat{\varphi}}{\partial y}(0,y')\|<\varepsilon/4,~~~\|\frac{\partial \widehat{\psi}}{\partial y}(0,y)-\frac{\partial \widehat{\psi}}{\partial y}(0,y')\|<\varepsilon/4.
$$

Let $\alpha(y)=\widehat{X}\circ \mathcal{P}(y)$ and $\beta(y)=\widehat{f}\circ \mathcal{P}(y)-\partial h/\partial y (\widehat{g}\circ \mathcal{P}(y))$, then we have
$$\frac{\alpha(y)}{\beta(y)}-\frac{\alpha(y')}{\beta(y')}=\frac{\alpha(y)-\alpha(y')}{\beta(y)}+\frac{\beta(y')-\beta(y)}{\beta(y)\beta(y')}\alpha(y').$$

We have the following estimation by the mean value theorem:
$$\|\alpha(y)-\alpha(y')\|\le \|D\widehat{X}\|\|\mathcal{P}(y)-\mathcal{P}(y')\|\le \widehat{K}\|D\widehat{X}\||y-y'|.$$

By reducing $\delta$ if necessary, for any $y,y'\in N_x(\beta_0|X(x)|)$, if $|y-y'|\le\delta|X(x)|$, one has

$$\|\frac{\alpha(y)-\alpha(y')}{\beta(y)}\|\le 1.001 \frac{\widehat{K}\|D\widehat{X}\|\delta|X(x)|}{|X(x)|}\le 2\delta \widehat{K}.$$
We just need to choose $\delta<\varepsilon/8\widehat{K}$

$$\|\frac{\beta(y)-\beta(y')}{\beta(y)\beta(y')} \alpha(y')\|\le \frac{\|\alpha(y')\|}{|\beta(y')|}\frac{\widehat{K}\|D\widehat{f}+\widehat{K}D\widehat{g}\||y-y'|}{|\widehat{X}|}<\varepsilon/4,$$
if we reduce $\delta$ again.

Let
$$F(y)=\frac{\widehat{X}\circ \mathcal{P}(y)}{\widehat{f}\circ \mathcal{P}-\partial h/\partial y}\widehat{g}\circ \mathcal{P},~~~G(y)=\frac{\partial h}{\partial y}\frac{\partial \widehat{\psi}}{\partial y}-\frac{\partial \widehat{\varphi}}{\partial y}.$$

We know that $F(y)$ is uniformly continuous and $G(y)$ is uniformly continuous. Thus $F(y)G(y)$ is uniformly continuous.

Combining all above estimations, we can know that $\mathcal{P}_{x,\phi_T(x)}$ is uniformly continuous.

For any $T\ge \beta_0$, let $n=[3T/\beta_0]$ and a partition
$$
0=t_0<t_1<\cdots<t_{n-1}<t_n=T,
$$
such that $t_i=i\beta_0/3, i=0,1,\cdots, n-1$. Then we have $t_n-t_{n-1}\in[\beta_0/3,2\beta_0/3].$ Then by using the prolongation, we know the result is true.
\end{proof}

Sometimes one needs to consider the \emph{scaled sectional Poincar\'e map ${\cal P}^*$} which is defined
in the following way:
$${\cal  P}^*_{x,\phi_T(x)}(y)=\frac{{\cal P}_{x,\phi_T(x)}(y|X(x)|)}{|X(\phi_T(x))|}$$
for each $y\in{\cal N}_x(\beta_T)$. Thus ${\cal P}^*_{x,\phi_T(x)}$ is a map from ${{\cal
N}_x(\beta_T)}$ to ${\cal N}_{\phi_T(x)}$.

\begin{Lemma}\label{Lem:estimationscaled}

Given $X\in{\cal X}^1(M^d)$ and $T>0$, there are constants $\beta_T>0$ and $K_T>0$ such that for any $t\in(0,T)$ and any regular point $x\in M^d$,
\begin{enumerate}

\item ${\cal P}^*_{x,\phi_t(x)}$ can be defined on ${\cal N}_x(\beta_T)$.

\item $D{\cal P}^*_{x,\phi_t(x)}$ is uniformly continuous: for any $\varepsilon>0$, there is $\delta>0$, such that for any $y,~z\in {\cal N}_x(\beta_T)$, if $d(y,z)<\delta$, one has $\|D{\cal P}^*_{x,\phi_t(x)}(y)-D{\cal P}^*_{x,\phi_t(x)}(z)\|<\varepsilon$.

\item $D_y
{\cal P}^*_{x,\phi_T(x)}(y)|_{y=0}=\psi_T^*(x)$.

\item $\|D{\cal P}^*_{x,\phi_t(x)}(y)\|\le K_T$ for any $y\in {\cal N}_x(\beta_T)$.

\end{enumerate}

\end{Lemma}

\begin{proof}
Item 1 is true because of Lemma~\ref{Lem:domainofPoincare}.

For any $y,z\in {\cal N}_x(\beta_T)$, one has
$$D{\cal  P}^*_{x,\phi_T(x)}(y)-D{\cal  P}^*_{x,\phi_T(x)}(z)=\frac{|X(x)|}{|X(\phi_T(x))|}(D{\cal  P}_{x,\phi_T(x)}(y|X(x)|)-D{\cal  P}_{x,\phi_T(x)}(z|X(x)|)).$$
Since $|X(x)|/|X(\phi_T(x))|$ is uniformly bounded, item 2 follows from Lemma~\ref{Lem:uniformcontinuous}.

For item 3, we have

$$D_y{\cal  P}^*_{x,\phi_T(x)}(y)=D_y \left(\frac{{\cal P}_{x,\phi_T(x)}(y|X(x)|)}{|X(\phi_T(x))|}\right)=D_y{\cal P}_{x,\phi_T(x)}(y|X(x)|)\frac{|X(x)|}{|X(\phi_t(x))|}.$$

Thus, when $y=0$, one has
$$D_y{\cal  P}^*_{x,\phi_T(x)}(0)=D_y{\cal P}_{x,\phi_T(x)}\frac{|X(x)|}{|X(\phi_T(x))|}=\frac{\psi_T(x)}{\|\Phi_T|_{\left<X(x) \right>}\|}=\psi_T^*(x).$$

Item 4 holds because that $\psi_T^*$ is uniformly bounded and item 2.

\end{proof}

\subsection{Franks' Lemma and dominated splittings}

As in the diffeomorphism case, one needs Franks' lemma to get some information on the derivative along periodic orbits. We state a version of Franks' lemma for flows which is taken from \cite[Theorem A.1]{BGV06}. Liao also had a version by using his standard differential equations \cite[Proposition 3.4]{Lia79}.

\begin{Lemma}\label{Lem:Franks}
Given $X\in\cX^1(M^d)$ and a $C^1$ neighborhood $\cU\subset\cX^1(M)$ of $X$, there is a neighborhood ${\cal V}\subset {\cal U}$ of $X$ and $\varepsilon>0$ such that for any $Y\in{\cal U}$, for any periodic orbit $\orb(x)$ of $Y$ with period $T\ge 1$, any neighborhood $U$ of ${\orb(x)}$ and any partition of $[0,T]$:
$$
0=t_0<t_1<\cdots<t_l=T,\quad 1\le t_{i+1}-t_i\le 2, i=0,1,\cdots, l-1,
$$
and any linear isomorphisms $L_i: \cN_{\phi^Y_{t_i}(x)}\to\cN_{\phi^Y_{t_{i+1}}(x)}, i=0,1,\cdots,l-1$ with $\|L_i-\psi^Y_{t_{i+1}-t_i}|_{\cN_{\phi_{t_i}(x)}}\|\le\e$, there exists $Z\in\cU$ such that $\psi^Z_{t_{i+1}-t_i}|_{\cN_{\phi_{t_i}(x)}}=L_i$ and $Z=Y$ on $(M^d\setminus U)\cup \orb(x)$.
\end{Lemma}

\begin{Remark}\label{Rem:timelength}
For simplicity, the time length of the partition is restricted to $[1,2]$. We remark that this is not a serious restriction. Sometimes, the system may contain periodic orbits with period less than $1$. But in our consideration, usually singularities are all hyperbolic. So there is a lower bound for the periods of periodic orbits. And after scaling, we may assume that the lower bound is $1$.
\end{Remark}

In this paper, we will use two estimations obtained by Franks' lemma: dominated splittings for vector fields away
from homoclinic tangencies and uniform estimation along a sequence of periodic orbits restricted on the stable bundle.

\paragraph{Dichotomy for the stable bundle:}

For any hyperbolic periodic orbit $\gamma$, ${\cal N}_\gamma$ admits a natural splitting ${\cal N}_{\gamma}={\cal N}^s\oplus {\cal N}^u$ such that ${\cal N}^s$ is contracting and ${\cal N}^u$ is expanding w.r.t. the linear Poincar\'e flow $\psi_t$.

By using the methods of periodic linear systems as in \cite{BDP03,BGV06,Man82,Wen02}, one can have the following
dichotomy result for a hyperbolic periodic orbit with large period:

\begin{Lemma}\label{Lem:dichotomyalongstable}
Given $X\in{\cal X}^1(M^d)$, for any $C^1$ neighborhood ${\cal U}$ of $X$, there are $\eta_{\cal U}>0$ and
$\iota_{\cal U}>0$ and a neighborhood ${\cal V}\subset{\cal U}$ of $X$ such that for any hyperbolic periodic
orbit $\gamma$ of index $i$ of $Y\in\cal V$ with $\tau(\gamma)>\iota_{\cal U}$, then
\begin{itemize}

\item either, there is $Z\in{\cal U}$ such that $\gamma$ is a hyperbolic periodic orbit of index $i-1$;

\item or, for any $x\in\gamma$, for \emph{any} time partition
$$0=t_0<t_1<\cdots<t_n=\tau(\gamma),$$
verifying $t_{i+1}-t_i\ge \iota_{\cal U}$ for $0\le i\le n-1$, one has
$$\prod_{i=0}^{n-1}\|\psi_{t_{i+1}-t_{i}}|_{{\cal N}^s(\phi_{t_i}(x))}\|\le C \exp\{-\eta_{\cal U} \tau(\gamma)\}.$$
\end{itemize}
\end{Lemma}

\begin{Definition}
Let $C>0$, $\eta>0$ and $T>0$. For a hyperbolic periodic orbit $\gamma$, for a $\psi_t$-invariant bundle $E\subset {\cal N}_\gamma$, we say that $\gamma$ is called \emph{$(C,\eta,T,E)$-contracting at the period (w.r.t. $\psi_t$)} if there is $m\in\NN$, and for any $x\in\gamma$, there is \emph{some} time partition
$$0=t_0<t_1<\cdots<t_n=m\tau(\gamma),$$
with $t_{i+1}-t_i\le T$ for $0\le i\le n-1$, such that
$$\prod_{i=0}^{n-1}\|\psi_{t_{i+1}-t_{i}}|_{E(\phi_{t_i}(x))}\|\le C \exp\{-\eta m\tau(\gamma)\},$$

For an $\psi_t$-invariant subbundle $F\subset {\cal N}_\gamma$, we say that $\gamma$ is called \emph{$(C,\eta,T,F)$-expanding at the period} if it is $(C,\eta,T,F)$-contracting at the period for $-X$.

\end{Definition}

\begin{Remark}

Since for a periodic orbit $\gamma$, for any $x\in\gamma$, one has $\Phi_{\tau(\gamma)}X(x)=X(\phi_{\tau(\gamma)}(x))=X(x)$, one can give the above definition by using $\psi_t^*$.

\end{Remark}

\begin{Corollary}\label{Cor:sequenceofsinks}

Given $X\in{\cal X}^1(M^d)$, we assume that $\Lambda$ is a compact invariant set and not reduced to a
critical element. If there is a sequence of vector fields $\{X_n\}$ such that
\begin{itemize}

\item $\lim_{n\to\infty}X_n=X$,

\item Each $X_n$ has a sink $\gamma_n$ such that $\lim_{n\to\infty}\gamma_n=\Lambda$.

\end{itemize}

then one has the following dichotomy:

\begin{itemize}

\item either, there is a sequence of vector fields $\{Y_n\}$ such that $\lim_{n\to\infty}d_{C^1}(X_n,Y_n)=0$, and $\gamma_n$ is a hyperbolic periodic orbit
of $Y_n$ of index $d-1$.

\item or, there are $\eta>0$ and $T>0$ such that for $n$ large enough, $\gamma_n$ is a $(1,\eta,T,{\cal N})$-contracting at the period w.r.t. $\psi_t^{Y_n}$.

\end{itemize}

\end{Corollary}

\begin{proof}
Note that $\lim_{n\to\infty}\tau(\gamma_n)=\infty$. If the ``either'' case is not true, then by Lemma~\ref{Lem:dichotomyalongstable}, one can get constants $\eta_{\cal U}$ and $\iota_{\cal U}$. Then the ``or'' case is true by taking $\eta=\eta_{\cal U}$ and $T=2\iota_{\cal U}$.

\end{proof}

\paragraph{Vector fields away from homoclinic tangencies: }

For any invariant set $\Lambda$ (maybe not compact) without singularities, if there are $\iota>0$ and an invariant splitting ${\cal N}_\Lambda={\cal N}^{cs}\oplus {\cal N}^{cu}$  w.r.t. $\psi_t$ satisfying $\|\psi_{\iota}|_{{\cal N}^{cs}(x)}\| \|\psi_{-\iota}|_{{\cal N}^{cu}(\phi_\iota(x))}\|\le 1/2$ for any $x\in\Lambda$, then we say that $\Lambda$ admits a $\iota$-dominated splitting w.r.t. $\psi_t$. If $\dim {\cal N}^{cs}(x)$ is independent of $x$, then it is called the \emph{index} of this dominated splitting. Note that for any linear flow defined on some linear bundle, one can define the notion of dominated splitting for that linear flow. Recall that
$${\cal HT}=\{X\in{\cal X}^1(M^d):~X~\textrm{has a homoclinic tangency.}\}.$$

By the similar arguments as diffeomorphisms and by using Franks' lemma for flows, from \cite{Wen02,Wen04}, we have
\begin{Lemma}\label{Lem:dominatedsplitting}
For any $X\in {\cal X}^1(M^d)\setminus\overline{{\cal HT}}$, there is a $C^1$ neighborhood ${\cal U}$ and
constants $C>0$, $\lambda>0$, $\delta>0$ and $\iota>0$ such that for any periodic orbit $\gamma$ of $Y\in\cal U$
with period $\pi(\gamma)>\iota$.

\begin{itemize}
\item $\Phi_{\pi(\gamma)}$
has at most two exponents in $(-\delta, \delta)$, and $\Phi_{\pi(\gamma)}$ has at
least one zero exponent, and this exponent corresponds to the flow direction.

\item There is an invariant splitting
${\cal N}_{\gamma}=G^s\oplus G^c\oplus G^u$ with respect to the linear Poincar\'e flow $\psi_t^Y$, where $G^s$
is the invariant space corresponding to the exponents less than $-\delta$, $G^c$ is the
invariant space corresponding to the exponents in $(-\delta,\delta)$ and  the dimension
of $G^c$ is zero or one, $G^u$ is the invariant space corresponding to the exponents larger than
$\delta$; moreover  for any $x\in\gamma$, for \emph{any} time partition
$$0=t_0<t_1<\cdots<t_n=\tau(\gamma),$$
verifying $t_{i+1}-t_i\ge \iota$ for $0\le i\le n-1$, one has
$$\prod_{i=0}^{n-1}\|\psi_{t_{i+1}-t_{i}}|_{G^s(\phi_{t_i}(x))}\|\le C \exp\{-\lambda \tau(\gamma)\},$$
$$\prod_{i=0}^{n-1}\|\psi_{t_{i}-t_{i+1}}|_{G^u(\phi_{t_{i+1}}(x))}\|\le C \exp\{-\lambda \tau(\gamma)\}.$$


\item If $\gamma$ is hyperbolic, and ${\cal N}_{\gamma}={\cal N}^s\oplus {\cal N}^u$ is the hyperbolic splitting with respect to $\psi_t^Y$,
then for any $x\in\gamma$ and $T>\iota$, one has
$$\|\psi_T^Y|_{{\cal N}^s(x)}\|\|\psi_{-T}^Y|_{{\cal N}^u(\phi_{T}^Y(x))}\|\le \frac{1}{2}.$$

\end{itemize}

\end{Lemma}

\begin{Corollary}\label{Cor:dominatedonweaktransitive}

Let $X\in{\cal X}^1(M^d)\setminus\overline{\cal HT}$. We assume that
\begin{itemize}

\item There is a sequence of vector fields $\{X_n\}$ such that $\lim_{n\to\infty}X_n=X$.

\item Each $X_n$ has a hyperbolic periodic orbit $\gamma_n$ of index $i$ such that
$\Lambda=\lim_{n\to\infty}\gamma_n$ in the Hausdorff topology.

\end{itemize}

Then
\begin{itemize}

\item ${\cal N}_{\Lambda\setminus{\rm Sing}(X)}$ admits a dominated splitting of index $i$ with respect to the
linear Poincar\'e flow $\psi_t$.

\item ${\cal N}_{\Lambda\setminus{\rm Sing}(X)}$ admits a dominated splitting of index $i$ with respect to the
scaled linear Poincar\'e flow $\psi^*_t$.

\item If one considers $\widetilde\Lambda$, then $\widetilde{\cal N}_{\widetilde\Lambda}$
admits a dominated splitting of index $i$ with respect to the flow $\widetilde\psi_t$.

\end{itemize}

\end{Corollary}

\begin{proof}
Since $X\in{\cal X}^1(M^d)\setminus{\overline {\cal HT}}$, $X_n\to X$ and $\gamma_n$ is a hyperbolic periodic orbit of $X_n$ of index $i$,
$${\cal N}_{\gamma_n}={\cal N}^{s}(\gamma_n)\oplus {\cal N}^u(\gamma_n)$$
is an $\iota$-dominated splitting of index $i$ w.r.t. $\psi_t^{X_n}$ for some uniform $\iota>0$.

For each $x\in\Lambda\setminus{\rm Sing}(X)$, by taking a subsequence if necessary, one can assume that there is $x_n\in\gamma_n$ such that $\lim_{n\to\infty}x_n=x$. After taking another subsequence, one can assume that ${\cal N}^{cs}(x)=\lim_{n\to\infty}{\cal N}^s(x_n)$ and ${\cal N}^{cu}(x)=\lim_{n\to\infty}{\cal N}^u(x_n)$.

Thus ${\cal N}_{\Lambda\setminus{\rm Sing}(X)}={\cal N}^{cs}\oplus{\cal N}^{cu}$ is an $\iota$-dominated splitting of index $i$. One can see \cite{LGW05} for more details.

Since
$$\psi_t^*(x)=\frac{\psi_t(x)}{\|\Phi_t|_{\left<X(x) \right>}\|},$$
any dominated splitting of $\psi_t$ is also a dominated splitting of $\psi_t^*$.

The dominated splitting of the linear Poincar\'e flow can be extended to the closure of its representation in the sphere bundle. See \cite{LGW05} for more details.
\end{proof}

\begin{Lemma}\label{Lem:unifromordominated}

For every $X\in{\cal X}^1(M^d)\setminus \overline{\cal HT}$, there are $\iota>0$,  $C>0$, $\eta>0$ and a $C^1$
neighborhood ${\cal U}$ of $X$ such that for any $Y\in\cal U$, if $\gamma$ is a periodic sink of $Y$ with period
$\tau(\gamma)>\iota$, then
\begin{itemize}

\item either, ${\cal N}_{\gamma}$ admits an $\iota$-dominated splitting of index $d-2$ with respect to $\psi_t^Y$.

\item or $\gamma$ is $(C,\eta,2\iota,{\cal N})$-contracting at the period w.r.t. $\psi_t^Y$.

\end{itemize}

\begin{proof}
Let $C$ and $\iota$ be as in Lemma~\ref{Lem:dominatedsplitting}. If the conclusion is not true, there exist $\eta_n\to 0$, $X_n\to X$ and a periodic sink $\gamma_n$ of $X_n$ with $\tau(\gamma_n)>\iota$, neither item 1 nor item 2 is satisfied. Then according to Franks' Lemma, after a small perturbation of $X_n$ of size $\eta_n$, we get a $Y_n$ such that $\gamma_n$ is a periodic orbit of index $d-2$. Since $Y_n\to X$, for $n$ large enough, $\psi_t^{Y_n}$ has an $\iota$-dominated splitting over $\gamma_n$ of index $d-2$, and then we get a dominated splitting for the extended linear Poincar\'e flow over the limit. But the limits of $X_n$  and $Y_n$ are the same since $X_n|_{\gamma_n}=Y_n|_{\gamma_n}$. By the continuity of dominated splitting of the extended linear Poincar\'e flow, we know that for $n$ large enough, $X_n$ has also an $\iota$-dominated splitting over $\gamma_n$ of index $d-2$, which gives a contradiction.
\end{proof}
\end{Lemma}

\subsection{Mixed dominated splittings: from linear Poincar\'e flow to tangent flow}

For two linear normed spaces $E$ and $F$, and a linear operator $A:~E\to F$, the mini-norm $m(A)$ is defined by
$$m(A)=\inf_{v\in E,~|v=1|}|Av|.$$
We use $L(E,F)$ to denote the space of bounded linear maps from $E$ to $F$.

The following lemma concerns how we can get the dominated splitting of the tangent flow from the dominated splitting of the linear Poincar\'e flow. \cite[Lemma 5.5, Lemma 5.6]{LGW05} used this kind of ideas. Here we give a general version.

\begin{Lemma}\label{Lem:mixed}

Let $\widehat\Lambda\subset SM^d$ be a compact invariant set of $\Phi_t^I$. Suppose
\begin{itemize}

\item $\widetilde{\cal N}_{\widehat{\Lambda}}=\Delta^{cs}\oplus\Delta^{cu}$ is a dominated splitting w.r.t. $\widetilde{\psi}_t$.

\item There are $C>0$ and $\lambda>0$ such that for any $u\in\widehat{\Lambda}$, for any $t>0$, one has
$$\frac{\|\widetilde{\psi}_t|_{\Delta^{cs}(u)}\|}{\|\Phi_t(u)\|}\le C{\rm e}^{-\lambda t}.$$

\end{itemize}
Then the projection $\pi({\widehat \Lambda})$ admits a dominated splitting $T_{\pi({\widehat\Lambda})} M^d=E\oplus F$ w.r.t the tangent flow $\Phi_t$, where $\dim
E=\dim \Delta^{cs}$.

\end{Lemma}

\begin{proof}
For each point $u\in\widetilde{\Lambda}\subset SM^d$, the direct-sum splitting $T_{\pi(u)} M^d=\Delta^{cs}\oplus<u>\oplus
\Delta^{cu}$ is continuous w.r.t. $u$. With respect to this decomposition, the tangent flow $\Phi_T$ has the following form:

$$
\left(
\begin{array}{ccc}
 \widetilde{\psi}_{T}|_{\Delta^{cs}(u)} & 0 & 0 \\
 B(u) & \Phi_{T}|_{<u>} & C(u) \\
 0 & 0 & \widetilde{\psi}_{T}|_{\Delta^{cu}(u)}
\end{array}
\right)
$$

By the definitions of $\chi_t$ and $\widetilde{\psi}_t$ (see Subsection~\ref{Sub:flows}), one has that
$F(u)=<u>\oplus\Delta^{cu}(u)$ is an invariant sub-bundle of $\Phi_t$. Let's find another invariant
sub-bundle of $\Phi_t$.

\begin{Claim}
There is $C_1>0$ and $\lambda_1>0$ such that for any $u\in\widehat\Lambda$ and any $t\ge 0$, one has
$$\frac{\|\widetilde{\psi}_t|_{\Delta^{cs}(u)}\|}{m(\Phi_t|_{F(u)})}\le C_1{\rm e}^{-\lambda_1 t}.$$

\end{Claim}

\begin{proof}[Proof of the claim]

By enlarging $T$ if necessary, one can assume that for any $u\in{\widehat\Lambda}\subset SM^d$, one has

$$\frac{\|\widetilde{\psi}_T|_{\Delta^{cs}(u)}\|}{\|\Phi_T|_{\left<u\right>}\|}\le \frac{1}{2},~~~\frac{\|\widetilde{\psi}_T|_{\Delta^{cs}(u)}\|}{m(\widetilde{\psi}_T|_{\Delta^{cu}(u)})}\le \frac{1}{2}.$$

Since $\Phi_T$ is bounded, by the continuity of the splitting, there is $K>0$ such that $\|\Phi_T\|\le K$ and $m(\Phi_T)\ge 1/K$. Denote by

$$ D(u)=\begin{pmatrix}   \Phi_{T}|_{<u>} & C(u) \\
 0 & \widetilde{\psi}_{T}|_{\Delta^{cu}(u)}   \end{pmatrix}.$$

For any $n\in\NN$, one has

$$\Phi_{-nT}|_{F(\Phi_{nT}^I(u))}=\prod_{i=0}^{n-1}D^{-1}(\Phi_{iT}^I(\Phi_{nT}^I(u)))=D^{-n}(\Phi_{nT}^I(u)).$$

Since
$$D^n(u)= \begin{pmatrix}  \Phi_{nT}|_{\left< u \right>} & \sum_{i=0}^{n-1}\Phi_{(n-1-i)T}|_{\left<\Phi^I_{(i+1)T}(u)\right>}C(\Phi_{iT}^I(u))\widetilde{\psi}_{iT}|_{\Delta^{cu}(u)} \\ 0 & \widetilde{\psi}_{nT}|_{\Delta^{cu}(u)}   \end{pmatrix},$$
we have
$$D^{-n}((\Phi_{nT}^I(u)))= \begin{pmatrix}  \Phi_{-nT}|_{\left< \Phi_{nT}^I(u) \right>} & \sum_{i=0}^{n-1}\Phi_{(-1-i)T}|_{\left<\Phi^I_{(i+1)T}(u)\right>}C(\Phi_{iT}^I(u))\widetilde{\psi}_{(i-n)T}|_{\Delta^{cu}(\Phi_{nT}^I(u))} \\ 0 & \widetilde{\psi}_{-nT}|_{\Delta^{cu}(\Phi_{nT}^I(u))}   \end{pmatrix}.$$

This implies

\begin{eqnarray*}
\|D^{-n}|_{\Phi_{nT}^I(u)}\|&\le&\|\Phi_{-nT}|_{\left< \Phi_{nT}^I(u) \right>}\|+\|\widetilde{\psi}_{-nT}|_{\Delta^{cu}(\Phi_{nT}^I(u))}\|\\
&+&\|\sum_{i=0}^{n-1}\Phi_{(-1-i)T}|_{\left<\Phi^I_{(i+1)T}(u)\right>}C(\Phi_{iT}^I(u))\widetilde{\psi}_{(i-n)T}|_{\Delta^{cu}(\Phi_{nT}^I(u))}\|\\
&\le& 2\cdot2^{-n}\frac{1}{\|\widetilde{\psi}_{nT}|_{\Delta^{cs}(u)}\|}+K\sum_{i=0}^{n-1}\|\Phi_{(-1-i)T}|_{\left<\Phi^I_{(i+1)T}(u)\right>}\|\|\widetilde{\psi}_{(i-n)T}|_{\Delta^{cu}(\Phi_{nT}^I(u))}\|\\
&\le&\frac{2^{-n+1}}{\|\widetilde{\psi}_{nT}|_{\Delta^{cs}(u)}\|}+K\sum_{i=0}^{n-1}\frac{2^{-(i+1)}}{\|\widetilde{\psi}_{(i+1)T}|_{\Delta^{cs}(u)}\|}\frac{2^{i-n}}{\|\widetilde{\psi}_{(n-i)T}|_{\Delta^{cs}(\Phi^I_{iT}(u))}\|}\\
&\le&\frac{2^{-n+1}}{\|\widetilde{\psi}_{nT}|_{\Delta^{cs}(u)}\|}+K^2\sum_{i=0}^{n-1}\frac{2^{-n-1}}{\|\widetilde{\psi}_{iT}|_{\Delta^{cs}(u)}\|\|\widetilde{\psi}_{(n-i)T}|_{\Delta^{cs}(\Phi^I_{iT}(u))}\|}\\
&\le&\frac{2^{-n+1}}{\|\widetilde{\psi}_{nT}|_{\Delta^{cs}(u)}\|}+K^2\frac{n 2^{-n-1}}{\|\widetilde{\psi}_{nT}|_{\Delta^{cs}(u)}\|}.
\end{eqnarray*}

Thus, when $n$ large enough, one has
$$\frac{\|\widetilde{\psi}_{nT}|_{\Delta^{cs}(u)}\|}{m(\Phi_{nT}|_{F(u)})}\le\frac{1}{2}.$$

This inequality implies the claim.
\end{proof}

Now we will start to find a $\Phi_t$-invariant bundle $E(u)$ and $T'>0$ such that for any $u\in\widehat\Lambda$,
$$\frac{\|\Phi_{T'}|_{E(u)}\|}{m(\Phi_{T'}|_{F(u)})}\le \frac{1}{2}.$$

By the claim above, there is $T_0>0$ such that
$$\frac{\|\widetilde{\psi}_{T_0}|_{\Delta^{cs}(u)}\|}{m(\Phi_{T_0}|_{F(u)})}\le \frac{1}{2}.$$

Let $L(\widehat{\Lambda})=\prod_{u\in\widehat\Lambda}L(\Delta^{cs}(u),F(u))$. For each $\Pi\in L(\widehat{\Lambda})$, one can define the norm $\|\Pi\|=\sup_{u\in\widehat\Lambda}\|\Pi(u)\|$. Under this norm, one knows that $L(\widehat\Lambda)$ is a Banach space. For any $\Pi\in L(\widehat\Lambda)$, one has

$$\begin{pmatrix}  \widetilde{\psi}_{T_0}|_{\Delta^{cs}(u)} & 0 \\ B(u) & \Phi_{T_0}|_{F(u)}\end{pmatrix} \begin{pmatrix} \Delta^{cs}(u) \\ \Pi(u)\Delta^{cs}(u)     \end{pmatrix}=\begin{pmatrix}   \widetilde{\psi}_{T_0}|_{\Delta^{cs}(u)}\Delta^{cs}(u) \\ B(u)\Delta^{cs}(u)+\Phi_{T_0}|_{F(u)}\Pi(u)\Delta^{cs}(u)\end{pmatrix}.$$

Thus, if we want to find an invariant bundle w.r.t. $\Phi_t$, we need to require that
$$B(u)\Delta^{cs}(u)+\Phi_{T_0}|_{F(u)}\Pi(u)\Delta^{cs}(u)=\Pi(\Phi_{T_0}^I(u))\widetilde{\psi}_{T_0}|_{\Delta^{cs}(u)}\Delta^{cs}(u).$$

In the spirit of the above equality, one can define a map ${\cal F}:~L(\widehat\Lambda) \to L(\widehat\Lambda)$ by the following form:

$${\cal F}\Pi(u)=(\Phi_{-T_0}|_{F(\Phi_{T_0}^{I}u)})(\Pi(\Phi_{T_0}^I(u))\widetilde{\psi}_{T_0}|_{\Delta^{cs}(u)}-B(u)).$$

Given $\Pi_1,\Pi_2\in L(\widehat\Lambda)$, one has
$${\cal F}\Pi_1-{\cal F}\Pi_2=(\Phi_{-T_0}|_{F(\Phi_{T_0}^{I}u)})(\Pi_1-\Pi_2)(\widetilde{\psi}_{T_0}|_{\Delta^{cs}(u)}).$$
Thus,
$$\|{\cal F}\Pi_1-{\cal F}\Pi_2\|\le \|\Phi_{-T_0}|_{F(\Phi_{T_0}^{I}u)}\|\|\Pi_1-\Pi_2\|\|\widetilde{\psi}_{T_0}|_{\Delta^{cs}(u)}\|\le\frac{1}{2}\|\Pi_1-\Pi_2\|.$$
So, ${\cal F}$ is a contracting map. By contraction mapping principle, $\mathcal{F}$ has a unique fixed point $\Pi\in L(\widehat\Lambda)$, i.e.,
$$
B(u)\Delta^{cs}(u)+\Phi_{T_0}|_{F(u)}\Pi(u)\Delta^{cs}(u)=
\Pi(\Phi_{T_0}^I(u))\widetilde{\psi}_{T_0}|_{\Delta^{cs}(u)}\Delta^{cs}(u).
$$
As a corollary, $E=({\rm id},\Pi)\Delta^{cs}$ is an invariant bundle of $\Phi_{T_0}$.

Since $E$, $F$ and $\Delta^{cs}$ are continuous bundles w.r.t. $u\in\widehat\Lambda$, there is $L>0$, which depends on the angles between each two bundles, such that for any non-zero vector $v_E\in E$, $v^{cs}\in\Delta^{cs}$, if $v_E=({\rm id},\Pi)v^{cs}$, then
$$ |v_E|\le \frac{1}{L}|v^{cs}|. $$
Thus, for each $n$, one has
$$\|\Phi_{nT_0}|_{E(u)}\|\le \frac{1}{L}\|\widetilde{\psi}_{nT_0}|_{\Delta^{cs}(u)}\|\le \frac{1}{L}\frac{1}{2^n}m(\Phi_{nT_0}|_{F(u)}).$$

From these, we get a $\Phi_t$-invariant splitting $E\oplus F$ over $\Lambda$, which satisfies the condition of the dominated splitting.
\end{proof}

\subsection{The existence of invariant manifolds}

We assume that $\Lambda$ is a compact invariant set and ${\cal N}_{\Lambda\setminus{\rm Sing}(X)}$ admits a
dominated splitting with respect to the linear Poincar\'e flow. If $\Lambda\cap{\rm Sing}(X)=\emptyset$, then
$\Lambda$ will have plaque family (\cite{HPS77}) as in the case of diffeomorphisms. If $\Lambda\cap{\rm
Sing}(X)\neq\emptyset$, then $\Lambda$ won't have uniform size of plaque family: the plaque family is defined on
a non-compact set and the size is scaled by the norm of the vector field.

The scaled Poincar\'e sectional map ${\cal P}^*$ could be defined in a uniform neighborhood of the zero section of ${\cal
N}$. Moreover, we have uniform estimations on $D{\cal P}^*_{x,\phi_T(x)}(y)$ by Lemma~\ref{Lem:estimationscaled}. For getting plaque
families of dominated splittings, one needs the following abstract lemma.

\begin{Lemma}\label{Lem:abstractplaque}
For any $d\in{\mathbb N}, L>0, r>0$, and $\alpha>0$, there is
$\gamma_0>0, \varepsilon_0>0$ such that: for any $\gamma\le\gamma_0$, there exists $\delta>0$, if a sequence of $C^1$ diffeomorphisms
$$f_i: {\mathbb R}^d(r)\to {\mathbb R}^d,\ i\in {\mathbb Z}$$
satisfy the following properties:

\begin{enumerate}

\item $f_i(0)=0$,

\item $\sup_{i\in{\mathbb Z}}\max\{|Df_i(0)|,|Df_i^{-1}(0)|\}\le L$

\item There is a sequence of invariant decompositions ${\mathbb R}^d=E_i\oplus F_i$ with the following properties:
\begin{itemize}
\item $Df_i(0)(E_i)=E_{i+1}$, $Df_i(0)(F_i)=F_{i+1}$,

\item $\angle(E_i,F_i)>\alpha$,

\item $$\frac{\|Df_i(0)|_{E_i}\|}{m(Df_i(0)|_{F_i})}\le \frac{1}{2}.$$

\end{itemize}

\item ${\rm Lip}(f_i-Df_i(0))<\varepsilon_0$.

\end{enumerate}

Then there are two sequences of embedding maps $\phi^{cs}_i\in {\rm Emb}(E^{cs}_i(\gamma),F^{cu}_i)$ and $\phi^{cu}_i\in {\rm Emb}(E^{cu}_i(\gamma),F^{cs}_i)$ such that
\begin{itemize}

\item $\phi^{cs/cu}_i(0)=0$, $D\phi^{cs/cu}_i(0)=0$,

\item $f_iW^{cs/cu}_i(\delta)\subset W^{cs/cu}_{i+1}(\gamma)$, where $W^{cs}_i(\gamma)$ is the graph of $\phi^{cs}_i$ restricted to $E^{cs}_i(\gamma)$ and $W^{cu}_i(\gamma)$ is the graph of $\phi^{cu}_i$ restricted to $E^{cu}_i(\gamma)$.

\end{itemize}

Moreover, the invariant manifolds are continuous with respect to the sequence of $f=(f_i)$. Precisely, for two sequences $f=(f_i), g=(g_i)$, define their metric as
$$
|f-g|_{C^1}=\sum_{i=-\infty}^\infty \frac{|f_i-g_i|_{C^1}}{2^{|i|}}.
$$
Then both $W^{cs}_i(\gamma,f)$ and $W^{cu}_i(\gamma,f)$ are continuous with respect to $f$, i.e., for every $i\in{\mathbb Z}$, if $f^{(n)}\to f$, $x_n\in W^{cs}_i(\gamma,f^{(n)})$, $x_n\to x$, then $x\in W^{cs}_i(\gamma,f)$, and $T_{x_n}W^{cs}_i(\gamma,f^{(n)})\to T_x W^{cs}_i(\gamma,f)$.
\end{Lemma}

The proof of Lemma~\ref{Lem:abstractplaque} needs to adapt the argument of \cite[Theorem 5.5]{HPS77}. We omit the proof here.

For diffeomorphisms, we know that plaque family of compact invariant set with dominated splittings exists.
For vector fields, if a compact invariant singular set has a dominated splitting w.r.t. the linear
Poincar\'e flow, we also have some similar results, but the form is changed: one should modify the size of the manifolds. Recall that $P_{x,\phi_t(x)}$ is the sectional Poincar\'e map between $N_x$ and $N_{\phi_t(x)}$.

\begin{Lemma}\label{Lem:plaquefamily}
Let $\Lambda$ be a compact invariant set of $X\in{\cal X}^1(M^d)$. We assume that $\Lambda\setminus{\rm
Sing}(X)$ admits a dominated splitting ${\cal N}_{\Lambda\setminus{\rm Sing}(X)}=\Delta^{cs}\oplus\Delta^{cu}$
of index $i$ with respect to the linear Poincar\'e flow $\psi_t$. Let $T>0$. There are continuous
maps:~$\eta^{cs}:\Lambda\setminus{\rm Sing}(X)\to {\rm Emb}^1(B^i,TM^d)$ and $\eta^{cu}:\Lambda\setminus{\rm
Sing}(X)\to {\rm Emb}^1(B^{d-1-i},TM^d)$ verifying the following properties:
\begin{enumerate}

\item $\eta^{cs}(x)(B^i(1))\subset {\cal N}_x$ and $\eta^{cu}(x)(B^{d-2-i}(1))\subset {\cal N}_x$ for any $x\in\Lambda\setminus{\rm
Sing}(X)$.

\item For any $\xi>0$, we define two sub-manifolds by $W_\xi^{cs}(x)=\exp_x(\eta^{cs}(x)(B^i(\xi)))$ and
$W_\xi^{cu}(x)=\exp_x(\eta^{cu}(x)(B^{d-2-i}(\xi)))$, then one has
\begin{itemize}
\item $T_x W^{cs}_1(x)=\Delta^{cs}(x)$ and $T_x W^{cu}_1(x)=\Delta^{cu}(x)$.

\item For any $\varepsilon>0$, there is $\delta>0$ such that for any regular point $x\in\Lambda$ and any
$t\in[0,T]$,one has $P_{x,\phi_t(x)}(W^{cs}_{\delta|X(x)|}))\subset
W^{cs}_{\varepsilon|X(\phi_t(x))|}(\phi_t(x))$ and $P_{x,\phi_t(x)}(W^{cu}_{\delta|X(x)|})\subset
W^{cu}_{\varepsilon|X(\phi_t(x))|}(\phi_t(x))$.

\end{itemize}

\end{enumerate}

\end{Lemma}

\begin{proof}

We will mainly use Lemma~\ref{Lem:abstractplaque} to prove this lemma. For each point $x$, ${\cal N}_x$ is isomorphic to ${\mathbb R}^{d-1}$. Since $\Lambda\setminus{\rm Sing}(X)$ admits a dominated splitting of index $i$ w.r.t. the linear Poincar\'e flow, there is $T>0$ such that
$$\frac {\|\psi^*_T|_{\Delta^{cs}(x)}\|}{m(\psi^*_T|_{\Delta^{cu}(x)})}\le \frac{1}{2},~~~\forall x\in\Lambda\setminus{\rm Sing}(X).$$
For each $i\in{\mathbb Z}$, one take $f_i={\cal P}_{\phi_{T i}(x),\phi_{(T+1)i}(x)}^*$ on ${\cal N}_{\phi_{T i}(x)}$. By Lemma~\ref{Lem:estimationscaled}, all assumptions of Lemma~\ref{Lem:abstractplaque} are satisfied. Then by Lemma~\ref{Lem:abstractplaque}, we get the existence of  plaque family.
\end{proof}

$W^{cs}(x)$ and $W^{cu}(x)$ are called \emph{central stable plaques} and \emph{central unstable plaques}
respectively.

\begin{Corollary}\label{Cor:intersection}
Let $\Lambda$ be a compact invariant set of $X\in{\cal X}^1(M^d)$. We assume that $\Lambda\setminus{\rm
Sing}(X)$ admits a dominated splitting ${\cal N}_{\Lambda\setminus{\rm Sing}(X)}=\Delta^{cs}\oplus\Delta^{cu}$
of index $i$ with respect to the linear Poincar\'e flow $\psi_t$. Let $T>0$. Then for any $\varepsilon>0$,
there is $\delta>0$ such that for any $x,y\in\Lambda$, if $d(x,y)<\delta$, $d(x,{\rm Sing}(X))>\varepsilon$,
$d(y,{\rm Sing}(X))>\varepsilon$, then $W^{cs}(x)\cap \phi_{[-T,T]}(W^{cu}(y))\neq\emptyset$.

\end{Corollary}
The proof of this corollary bases on the uniform continuity of plaque families when points are far away from singularities.

\subsection{Estimations on the size of stable/unstable manifolds}
\begin{Definition}

Let $\Lambda$ be an invariant set and $E\subset {\cal N}_{\Lambda\setminus{\rm Sing}(X)}$ an invariant
subbundle of the linear Poincar\'e flow $\psi_t$. For $C>0$, $\eta>0$ and $T>0$, $x\in\Lambda\setminus{\rm
Sing}(X)$ is called \emph{$(C,\eta,T,E)$-$\psi^*_t$-contracting} if for any partition of times
$0=t_0<t_1<\cdots<t_n<\cdots$ verifying
\begin{itemize}

\item $t_{n+1}-t_{n}\ge T$ for any $n\in\NN$ and $t_n\to\infty$ as $n\to\infty$,

\item For any $n\in\NN$,
$$\prod_{i=0}^{n-1}\|\psi_{t_{i+1-t_i}}^*|_{E(\phi_{t_i}(x))}\|\le C \mathrm{e}^{-\eta t_n}.$$
\end{itemize}

$x\in\Lambda\setminus{\rm Sing}(X)$ is called \emph{$(C,\eta,T,E)$-$\psi_t^*$-expanding} if it's
$(C,\eta,T,E)$-$\psi_t^*$-contracting for $-X$.

\end{Definition}

An increasing homeomorphism $\theta:{\mathbb R}\to{\mathbb R}$ is called a \emph{reparametrization} if $\theta(0)=0$. we meet the reparametrization problem for flows.

For any orbit ${\rm Orb}(x)$, one defines

$$
W^s({\rm Orb}(x))=\{y\in M^d, \exists \textrm{ a reparametrization }\theta~\textrm{s.t.,}~\lim_{t\to\infty}d(\phi_{\theta(t)}(y),\phi_t(x))=0\},
$$
$$
W^s({\rm Orb}(x))=\{y\in M^d, \exists \textrm{ a reparametrization }\theta~\textrm{s.t.,}~\lim_{t\to-\infty}d(\phi_{\theta(t)}(y),\phi_t(x))=0\}.
$$
%
%

\begin{Lemma}\label{Lem:intersect}
Let $\Lambda$ be a compact invariant set of $X\in{\cal X}^1(M^d)$.
Assume that $\Lambda\setminus{\rm Sing}(X)$ admits a dominated splitting ${\cal N}_{\Lambda\setminus{\rm Sing}(X)}=\Delta^{cs}\oplus\Delta^{cu}$
of index $i$ with respect to the linear Poincar\'e flow $\psi_t$. For $C>0$, $\eta>0$ and $T>0$, there is
$\delta>0$ such that
\begin{itemize}

\item For any regular point $x\in\Lambda$, if $x$ is $(C,\eta,T,\Delta^{cs})$-$\psi_t^*$-contracting, then $W^{cs}_{\delta|X(x)|}\subset W^{s}({\rm
Orb}(x))$;

\item For any regular point $x\in\Lambda$, if $x$ is $(C,\eta,T,\Delta^{cu})$-$\psi_t^*$-expanding, then $W^{cu}_{\delta|X(x)|}\subset W^{u}({\rm
Orb}(x))$.

\end{itemize}

\end{Lemma}

\begin{proof}
We need to prove that there is $\delta>0$ such that
$$\lim_{t\to+\infty} {\rm diam }\left({\cal P}^*_{x,\phi_t(x)}(\eta^{cs}(x)(B^i(\delta)))\right)=0.$$
By the uniform continuity of $D{\cal P}^*_{x,\phi_t(x)}$ in Lemma~\ref{Lem:estimationscaled}, we have uniform linearized neighborhood of $0$ in ${\cal N}_z$ for each regular point $z$. Then the proof parallels to \cite[Corollary 3.3]{PuS00}.
\end{proof}

\begin{Corollary}
Under the assumption of Lemma~\ref{Lem:intersect} for any compact set $\Lambda_0\subset\Lambda\setminus{\rm Sing}(X)$, there is $\varepsilon>0$
such that for any $x,y\in\Lambda_0$, if
\begin{itemize}

\item $d(x,y)<\varepsilon$;

\item $x$ is $(C,\eta,T,\Delta^{cs})$-$\psi_t^*$-contracting and $y$ is
$(C,\eta,T,\Delta^{cu})$-$\psi_t^*$-expanding;

\end{itemize}

then $W^{s}({\rm Orb}(x))\cap W^{u}({\rm Orb}(y))\neq\emptyset$.

\end{Corollary}

Similar to the proof of Lemma~\ref{Lem:intersect}, we have

\begin{Lemma}\label{Lem:liaoestimation}
Let $X\in{\cal X}^1(M^d)$. For any $C>0$, $\eta>0$ and $T>0$, there is $\delta=\delta(X,C,\eta,T)>0$ such that
if a regular point $x\in M^d$ is $(C,\eta,T,{\cal N})$-$\psi^*$-contracting, then
$$\lim_{t\to+\infty} {\rm diam }\left({P}_{x,\phi_t(x)}({N}_{x}(\delta|X(x|))\right)=0.$$
In other words,
$N_{x,\delta|X(x)|}\subset W^s({\rm Orb}(x))$.
\end{Lemma}
%
%
%
Notice that not only $x$ may be close to a singularity, but also $\omega(x)$ may contain singularities.

\begin{Theorem}\cite[Theorem 4.1, Proposition 6.2]{Lia89}\label{Thm:etad}
Given $X\in\cX^1(M^d)$ and a hyperbolic singularity $\sigma$ of $X$, for any $C>0, \eta>0$ and $T>1$, there exists a
neighborhood $U$ of $\sigma$ such that there is no $(C,\eta,T,{\cal N})$-$\psi^*$-contracting periodic point in $U$.
\end{Theorem}

The proof of Theorem~\ref{Thm:etad} is not short and it is contained in \cite{Lia89}. Now we give some idea of Theorem~\ref{Thm:etad}: if Theorem~\ref{Thm:etad} is not true, then there is a sequence of $(C,\eta,T,{\cal N})$-$\psi^*$-contracting periodic points $\{p_n\}$ such that $\lim_{n\to\infty}p_n=\sigma$. This holds only if $\sigma$ is a saddle. By the property of $\{p_n\}$, we have
\begin{itemize}

\item $T_\sigma M^d$ admits a dominated splitting $T_\sigma M^d=E^{cs}\oplus E^{uu}$ w.r.t. the tangent flow, where $\dim E^{uu}=1$ and $E^{uu}$ is strong unstable.

\item There is $\delta>0$, $N_{p_n}(\delta|X(p_n)|)\subset W^s(p_n)$.

\end{itemize}

Then by a careful estimation (which is not obvious), one has for $n$ large enough, $W^{uu}(\sigma)\cap W^s(p_n)\neq\emptyset$, where $W^{uu}(\sigma)$ is the strong unstable manifold corresponding to $E^{uu}$. But $W^{uu}(\sigma)\setminus\{\sigma\}$ contains only two orbits, and $p_n$ are distinct periodic orbits. This gives us a contradiction.

\subsection{Pliss Lemma}

We use the following lemma of Pliss type to get the points which can have uniform estimations to infinity.

\begin{Lemma}\label{Lem:pliss}
Given $X\in\cX^1(M^d)$, $C>0$, $T>0$ and $\eta>0$, for any $\eta'\in(0,\eta)$, there is $N=N(C,T,\eta,\eta')>0$, such that if $\gamma$ is a periodic orbit with period $\tau(\gamma)>N$, $E\subset{\cal N}_\gamma$ is an invariant bundle w.r.t. $\psi_t$, and if $\gamma$ is $(C,\eta,T,E)$-contracting at the period w.r.t. $\psi_t$, then there is $x\in\gamma$ such that $x$ is $(1,\eta',T,E)$-$\psi_t^*$-contracting.

\end{Lemma}

\begin{proof}

Since $\gamma$ is a $(C,\eta,T,E)$-contracting at the period w.r.t. $\psi_t$, there is $m\in\NN$ and a time partition
$$0=t_0<t_1<\cdots<t_n=m\tau(\gamma),$$
with $t_{i+1}-t_i\le T$ for $0\le i\le n-1$, such that
$$\prod_{i=0}^{n-1}\|\psi_{t_{i+1}-t_{i}}|_{E(\phi_{t_i}(x))}\|\le C \exp\{-\eta m\tau(\gamma)\}.$$

Since $\Phi_{\tau(\gamma)}(X(x))=X(\phi_{\tau(\gamma)}(x))=X(x)$, the above estimation is also true for $\psi_t^*$:
$$\prod_{i=0}^{n-1}\|\psi^*_{t_{i+1}-t_{i}}|_{E(\phi_{t_i}(x))}\|\le C \exp\{-\eta m\tau(\gamma)\}.$$

When $\eta'<\eta,$ if $\tau(\gamma)$ is large enough, one can cancel the constant $C$. Following \cite[Lemma 2.14]{CSY11}, one can get the $(1,\eta',T,E)$-$\psi_t^*$-contracting point $x$.
\end{proof}

\begin{Lemma}\label{Lem:finiteunifsink}
Let $X\in{\cal X}^1(M^d)$ be Kupka-Smale. For any $C>0$, $T>0$ and $\eta>0$, $X$ can only have finitely many $(C,\eta,T,{\cal N})$-contracting periodic orbits.
\end{Lemma}

\begin{proof}
If the conclusion is not true, then $X$ has infinitely many distinct periodic orbits $\{\gamma_n\}$ such that
\begin{itemize}

\item $\lim_{n\to\infty}\tau(\gamma_n)=\infty$.

\item Each $\gamma_n$ is $(C,\eta,T,{\cal N})$-contracting at the period w.r.t. $\psi_t$.
\end{itemize}

By Lemma~\ref{Lem:pliss}, for each $n$ large enough, there is $x_n\in\gamma_n$ such that $x_n$ is a $(1,\eta/2,T,{\cal N})$-$\psi_t^*$-contracting point. Thus, there is $\delta=\delta(X,\eta,T)>0$ such that $N_{x_n}(\delta|X(x_n)|)$ is contained in the stable manifold of $x_n$. If $x_n$ accumulates on singularities, then one can get a contradiction by Theorem~\ref{Thm:etad}. If $x_n$ dose not accumulate on singularities, then the basin of $\gamma_n$ covers an open set with uniform size. This also gives a contradiction because the volume of $M^d$ is finite and $\{\gamma_n\}$ are distinct periodic orbits.
\end{proof}

\section{Chain recurrence and genericity}\label{Sec:generic}

\subsection{Conley theory}

A chain recurrent class is called \emph{non-trivial} if it is not reduced to a critical element;
otherwise, it is called trivial. For each hyperbolic critical element $p$ of $X$, since $\orb(p)$ has a
well-defined continuation $\orb(p_Y)$ for $Y$ close to $X$, $C(p)$ also has a well-defined continuation
$C(p_Y,Y)$.

A compact invariant set $\Lambda$ of $X$ (if it has a continuation) is called \emph{lower semi-continuous} if for any sequence of vector fields $\{X_n\}$ verifying $\lim_{n\to\infty}X_n=X$, one has
$\liminf_{n\to\infty}\Lambda_{X_n}\supset\Lambda$. A compact invariant set $\Lambda$ of $X$ (if it has a continuation)
is called \emph{upper semi-continuous} if for any sequence of vector fields $\{X_n\}$ verifying
$\lim_{n\to\infty}X_n=X$, one has $\limsup_{n\to\infty}\Lambda_{X_n}\subset\Lambda$. It is well known that the closure
of hyperbolic periodic orbits is lower semi-continuous and the chain recurrent set is upper semi-continuous.
There is a classical result saying that lower semi-continuous sets and upper semi-continuous sets are continuous for
generic vector fields.

\begin{Lemma}\label{Lem:u.s.c}
For a hyperbolic critical element $p$, $C(p)$ is upper semi-continuous. As a corollary, if $p_1$ and $p_2$ are two critical elements of
$X$ with the property $C(p_1)\cap C(p_2)=\emptyset$, then there is a neighborhood $\cal U$ of $X$ such that for any
$Y\in\cal U$, one has $C(p_{1,Y})\cap C(p_{2,Y})=\emptyset$.
\end{Lemma}

\begin{proof}
The fact that $C(p)$ is upper semi-continuous because of the continuity of the flows with respect to the vector fields.

By \cite{Con78}, if we have two chain recurrent classes $C(p_1)$ and $C(p_2)$ satisfying $C(p_1)\cap C(p_2)=\emptyset$, then there is an open set $U$
such that
\begin{itemize}
\item $\phi_t(\overline{U})\subset U$ for $t>0$;

\item $C(p_1)\subset U$ and $C(p_2)\subset {\rm Int}(M\setminus\overline{U})$ or $C(p_2)\subset U$ and $C(p_1)\subset {\rm Int}(M\setminus\overline{U})$

\end{itemize}

Then by the continuity of vector fields, there is a $C^1$ neighborhood ${\cal U}$ of $X$ such that for any
$Y\in\cal U$, one has $C(p_{1,Y})\subset U$ and $C(p_{2,Y})\subset {\rm Int}(M\setminus\overline{U})$ or
$C(p_{2,Y})\subset U$ and $C(p_{1,Y})\subset {\rm Int}(M\setminus\overline{U})$. As a corollary, one has
$C(p_{1,Y})\cap C(p_{2,Y})=\emptyset$.

\end{proof}

 For each
point $x\in M^d$, one can define the strong stable manifold $W^{ss}(x)$ and the strong unstable manifold
$W^{uu}(x)$ as
$$W^{ss}(x)\triangleq\{y\in M^d:~\lim_{t\to\infty}d(\phi_t(x),\phi_t(y))=0\},$$
$$W^{uu}(x)\triangleq\{y\in M^d:~\lim_{t\to-\infty}d(\phi_t(x),\phi_t(y))=0\}.$$

But for flows, this definition is not enough in many cases. By the difference with diffeomorphisms, sometimes we
need to reparameterize the time variable. This leads us to give the definition of $W^s({\rm Orb}(x))$ and $W^u({\rm Orb}(x))$ as in Section~\ref{Sec:differentflows}.

%
%
%

By the definitions, one has that $W^s({\rm Orb}(x))$ and $W^u({\rm Orb}(x))$ are invariant sets. The proof of the following lemma is fundamental, hence omitted.

\begin{Lemma}\label{Lem:invariantinside}
For any hyperbolic critical point $p$, one has
\begin{itemize}
\item For any critical element $p$, one has $W^s({\rm Orb}(p))=\cup_{t\ge 0}\phi_t (W^{ss}(p))$ and $W^u({\rm Orb}(p))=\cup_{t\ge 0}\phi_t (W^{uu}(p))$.

\item If $C(p)$ is non-trivial, then $C(p)\cap W^s({\rm Orb}(p))\setminus{\rm
Orb}(p)\neq\emptyset$ and $C(p)\cap W^u({\rm Orb}(p))\setminus{\rm Orb}(p)\neq\emptyset$.

\end{itemize}
\end{Lemma}

For a compact invariant set $\Lambda$, one says that $\Lambda$ is \emph{Lyapunov stable} for $X$ if for any
neighborhood $U$ of $\Lambda$, there is a neighborhood $V$ of $\Lambda$ such that $\phi_t(V)\subset U$ for any $t>0$.

\begin{Lemma}\label{Lem:unstableinLyapunov}

If $\Lambda$ is Lyapunov stable, then $W^u({\rm Orb}(x))\subset\Lambda$ for each $x\in\Lambda$.
\end{Lemma}

\begin{proof}
Given any $y\in W^u({\rm Orb}(x))$, for any neighborhood $U$ of $\Lambda$, since $\Lambda$ is Lyapunov stable,
there is a neighborhood $V$ of $\Lambda$ such that $\phi_t(V)\subset U$ for any $t\ge 0$. For any $y\in W^u({\rm Orb}(x))$, there is an
increasing homeomorphism $\theta:\RR\to\RR$ such that $d(\phi_t(x),\phi_{\theta(t)}(y))\to 0$ as $t\to-\infty$.
Thus there is $t_V>0$ such that $\theta(-t_V)<0$ and $\phi_{\theta(-t_V)}\in V$. By the Lyapunov stability one
has $y=\phi_{-\theta(-t_V)}(\phi_{\theta(-t_V)}(y))\in U$. By the arbitrary property of $U$, one has
$y\in\Lambda$.
\end{proof}

By the definition, if $\Lambda$ is not Lyapunov stable, then there
is a neighborhood $U_0$ of $\Lambda$ such that there is a sequence
of neighborhoods $\{V_n\}$ such that
\begin{itemize}

\item $\lim_{n\to\infty}\overline{V_n}=\Lambda$.

\item $\phi_{t_n}(V_n)\nsubseteq U$ for some $t_n>0$.

\end{itemize}

As a corollary, $\lim_{n\to\infty}t_n=\infty$ since $\Lambda$ is an
invariant set. Hence, if $\Lambda$ is not Lyapunov stable, then
there are $\{t_n\}\subset\RR$ and a sequence of points $\{x_n\}$ such
that
\begin{itemize}

\item $\lim_{n\to\infty}x_n\in\Lambda$.

\item $\lim_{n\to\infty}t_n=\infty$.

\item $\lim_{n\to\infty}\phi_{t_n}(x_n)$ exists and
$\lim_{n\to\infty}\phi_{t_n}(x_n)\notin\Lambda$.

\end{itemize}
For $x\in M^d$, one can define the \emph{chain unstable set} $W^{ch,u}(x)$ of $x$ and the \emph{chain stable
set} $W^{ch,s}(x)$ of $x$ in the following way:
\begin{eqnarray*}
W^{ch,u}(x)& = & \{y\in M^d:~\forall\varepsilon>0,~\exists ~\textrm{an}~\varepsilon-\textrm{pseudo orbit}\{x_i\}_{i=0}^n ~\textrm{s.t.}~x_0=x,~x_n=y\},\\
W^{ch,s}(x)& = & \{y\in M^d:~\forall\varepsilon>0,~\exists ~\textrm{an}~\varepsilon-\textrm{pseudo
orbit}\{x_i\}_{i=0}^n ~\textrm{s.t.}~x_0=y,~x_n=x\}.
\end{eqnarray*}

$y$ is \emph{chain attainable} from $x$ iff $y\in W^{ch,u}(x)$. For $x,y\in M^d$, one says that
they are \emph{chain bi-attainable} if $y\in W^{ch,u}(x)\cap W^{ch,s}(x)$.

\subsection{The $C^1$ connecting lemmas and ergodic closing lemma for
flows}\label{Sec:connecting}

\cite{Arn01,WeX00} gave the following extension of Hayashi's $C^1$ connecting lemma \cite{Hay97}. We will use it in Section~\ref{Sec:birthhomo}.

\begin{Lemma}\label{Lem:wenxia}
For any vector field $X\in{\cal X}^1(M^d)$, for any point $z\notin{\rm Per}(X)\cup{\rm Sing}(X)$, for any
$\varepsilon>0$, there are $L>0$ and two neighborhoods $\widetilde{W_z}\subset\ W_z$ of $z$ such that
\begin{itemize}
\item one can choose $W_z$ and $\widetilde{W_z}$ to be arbitrarily small neighborhoods of $z$,

\item for any $p$ and $q$ in $M^d$, if the positive orbit of $p$ and the negative orbit of $q$ enter
$\widetilde{W_z}$, but the orbit segments $\{\phi_t(p):0\le t\le L\}$ and $\{\phi_t(q):-L\le t\le 0\}$ don't
intersect $W_z$,
\end{itemize}
then there is $Y$ $\varepsilon$-$C^1$-close to $X$ such that
\begin{itemize}
\item $q$ is in the positive orbit of $p$ with respect to the flow $\phi_t^Y$ generated by $Y$.

\item $Y(x)=X(x)$ for any $x\in M^d\setminus W_{L,z}$, where $W_{L,z}=\bigcup_{0\le t\le L}\phi_t^X(W_z)$.
\end{itemize}

\end{Lemma}

Ma\~n\'e's ergodic closing lemma \cite{Man82} is also useful in this paper. First we state a flow version of Ma\~n\'e's ergodic closing lemma taken from \cite{Wen96}.

\begin{Definition}
Let $X\in{\cal X}^1(M^d)$. A regular point $x\in M^d$ is called \emph{strongly closable} if for any $C^1$ neighborhood $\cal U$ of $X$, and any $\varepsilon>0$, there are $Y\in{\cal U}$ and a periodic point $y\in M^d$ of $Y$ with period $\tau(y)$ such that
\begin{itemize}

\item $X(z)=Y(z)$ for any $x\in M^d\setminus\bigcup_{t\in{\mathbb R}}\phi_t(B(x,\varepsilon))$,

\item $d(\phi_t^X(x),\phi_t^Y(y))<\varepsilon$ for each $t\in[0,\tau(y)]$.

\end{itemize}

Denote by $\Sigma(X)$ the set of strongly closable points of $X$.
\end{Definition}

The ergodic closing lemma \cite{Man82,Wen96} states:

\begin{Lemma}\label{Lem:ergodicclosing}
$\mu(\Sigma(X)\cup{\rm Sing}(X))=1$ for every $T>0$ and every $\phi_T^X$-invariant probability Borel measure $\mu$.
\end{Lemma}

One needs the following corollary which asserts that one can get a periodic orbit with some additional properties and preserve a compact
invariant subset in a transitive set by small perturbations simultaneously. For the applications in this paper,
one takes the compact invariant subset is the union of finitely homoclinic orbits of singularities. See Section
\ref{Sec:birthhomo}.

\begin{Corollary}\label{Cor:subsetisrobust}
Let $f:~M^d\to \mathbb R$ be a continuous function. Let $\mu$ be an ergodic measure of a flow $\phi_t$ generated by $X\in{\cal X}^1(M^d)$, which is not supported on a singularity. Assume that $\Lambda$ is a compact invariant set such that $\mu(\Lambda)=0$. Then for any $\varepsilon>0$, for any $C^1$ neighborhood ${\cal U}$ of $X$, and any neighborhood $U$ of ${\rm supp}(\mu)$, there is a periodic orbit $\gamma\subset U$ of $Y\in{\cal U}$ such that

\begin{itemize}

\item $\Lambda$ is a compact invariant set of $Y$.

\item $|\int f {\rm d}\delta_\gamma-\int f{\rm d}\mu|<\varepsilon$, where $\delta_\gamma$ is the uniform distribution measure on $\gamma$, i.e., for any continuous function $g: M^d\to \mathbb{R}$,
    $$
    \int g(x)\mathrm{d} \delta_\gamma(x) =\frac 1{\tau(\gamma)}\int_0^{\tau(\gamma)}g(\phi_t(p))\mathrm{d} t,
    $$
    where $p\in\gamma$.
\end{itemize}
\end{Corollary}

\begin{proof}
Let $C=\max\{\sup_{x\in M^d}|f(x)|,1\}$.
Without loss of generality, one can assume that $\mu$ is not supported on a periodic orbit. Since $\mu$ is not supported on a singularity, one has $\mu(\Sigma(X))=0$. We choose $x\in\Sigma$ satisfying
$$\lim_{t\to\infty}\delta_{x,T}=\mu,$$
where $\delta_{x,T}$ is the uniform distribution measure supported on $\phi_{[0,T]}(x)$, i.e.,
for any continuous function $g: M^d\to \mathbb{R}$,
    $$
    \int g(y)\mathrm{d} \delta_{x,T}(y) =\frac 1{T}\int_0^{T}g(\phi_t(x))\mathrm{d} t.
    $$

For any $\varepsilon>0$, there is $T_0>0$ such that for any $T>T_0$, one has
$$\left|\int f{\rm d}\delta_{x,T}-\int f{\rm d}\mu\right|<\varepsilon/2.$$

By reducing $\varepsilon$ if necessary, one can assume that $B(x,\varepsilon)\cap \Lambda=\emptyset$. Since $\Lambda$ is invariant, one has that $\phi_t(B(x,\varepsilon))\cap\Lambda=\emptyset$ for any $t\in\mathbb R$. By Lemma~\ref{Lem:ergodicclosing}, there is $Y$ which has a periodic point $p$ with period $\tau(p)$ such that
\begin{itemize}

\item $d_{C^1}(X,Y)<\varepsilon$;

\item $\Lambda$ is a compact invariant set of $Y$;

\item $d(\phi_t^X(x),\phi_t^Y(p))<\varepsilon/2C$ for each $t\in[0,\tau(p)]$.

\end{itemize}

$x$ is not periodic because we assume that $\mu$ is not supported on a periodic orbit. Thus, one can assume that $\tau(p)>T_0$. Let $\gamma={\rm Orb}(p)$. Then

\begin{eqnarray*}
\left|\int f {\rm d}\delta_\gamma-\int f{\rm d}\mu\right|& \le & \left|\int f{\rm d}\delta_\gamma-\int f{\rm d}_{\delta_{x,\tau(\gamma)}}\right|+\left|\int f{\rm d}_{\delta_{x,\tau(\gamma)}}-\int f{\rm d}\mu\right|\\
& < & C\frac{\varepsilon}{2C}+\frac{\varepsilon}{2}=\varepsilon.
\end{eqnarray*}

\end{proof}

%
%
%
%
%
%
A vector field $X\in{\cal X}^r(M^d)$ is called \emph{Kupka-Smale} if every critical element is hyperbolic, and the stable manifold of any critical element intersects the unstable manifold of any other critical element transversely.  A classical generic result is: Kupka-Smale vector fields form a residual set in ${\cal X}^r(M^d)$. We need the following weak terminology:

\begin{Definition}
A vector field $X\in {\cal X}^r(M^d)$ is called \emph{weak Kupka-Smale} if every critical element of $X$ is hyperbolic.

\end{Definition}

Since Kupka-Smale is $C^r$ generic in ${\cal X}^r(M^d)$, we have that weak Kupka-Smale is also a $C^r$ generic property in ${\cal X}^r(M^d)$.

We will state a connecting lemma for pseudo orbits, which was studied in \cite{BoC04}. \cite{BoC04} studied the connecting lemma for pseudo-orbits for weak Kupka-Smale diffeomorphisms. The assumption of weak Kupka-Smale is used since
\begin{itemize}

\item By using $\lambda$-lemma, for every non-periodic point $x$ which is not in the stable/unstable manifold of a periodic point, the positive/negative iteration of $x$ will be in a topological tower.

\end{itemize}

For flows, $\lambda$-lemma is true for both hyperbolic singularities and hyperbolic periodic orbits. The orbit structure is clear near hyperbolic critical elements. So the connecting lemma for pseudo-orbits is true for vector fields.

\begin{Lemma}(\cite[Theorem 1.2]{BoC04} )\label{Lem:connecting}

Let $X\in\mathcal{X}^1(M^d)$ be a weak Kupka-Smale vector field. Given any finite set ${\hat F}$ of periodic orbits of $X$, for any neighborhood $U$ of ${\hat F}$, for any $C^1$ neighborhood $\cal U$ of $X$, and for any $x,y\in M^d\setminus U$, if $y$ is chain attainable from $x$, then there is $Y\in\cal U$ and $t>0$ such that $\phi_t^Y(x)=y$ and $Y(z)=X(z)$ for any $z\in U$.

Moreover, if $X$ is $C^r$, we can require that $Y$ is also $C^r$.
\end{Lemma}

The statement here is a little bit different from \cite{BoC04}, but it is essentially contained there.
See \cite[Remarque 1.1]{BoC04}.

\begin{Remark}
Due to Lemma~\ref{Lem:connecting}, we can consider a weak Kupka-Smale vector field close to a generic one. See Lemma~\ref{Lem:stablepropertyofquasiattractor} and Section~\ref{Sec:birthhomo} for more details.
\end{Remark}

Sometimes, one needs to perturb a vector field to be a weak Kupka-Smale vector filed by preserving some non-transverse connection. This was used by Palis \cite{Pal88}. One can see R. Xi's master thesis \cite{Xir05} for a proof.

\begin{Theorem}\label{Thm:xiruibin}\cite{Pal88,Xir05}
For any vector field $X\in{\cal X}^r(M^d)$, any $n\in\NN$ and any hyperbolic critical elements
$\{P_1,Q_1,\cdots,P_n,Q_n\}$, if ${\rm Orb}(x_i)\subset W^s(P_i)\cap W^u(Q_i)$ is a non-transverse orbit for
$1\le i\le n$, then for any $C^r$ neighborhood $\cal U$ of $X$ there eixists $Y\in\cal U$,
\begin{itemize}
\item ${\rm Orb}(x_i)\subset W^s(P_i)\cap W^u(Q_i)$ is still a non-transverse orbit of $Y$.

\item Any other critical element of $Y$ is hyperbolic, in other words, $Y$ is weak Kupka-Smale.

\end{itemize}

\end{Theorem}

\subsection{Generic results}\label{Sub:generic}

${\cal R}\subset{\cal X}^1(M^d)$ is called \emph{residual} if it contains a dense $G_\delta$ subsets of ${\cal X}^1(M^d)$. A property of vector fields is called \emph{$C^1$
generic} if it holds on a residual set in ${\cal X}^1(M^d)$. Sometimes we use the terminology ``for $C^1$
generic $X$'' which is equivalent to say  that ``there is a residual set ${\cal R}\subset{\cal X}^1(M^d)$ and
$X\in{\cal R}$". Being a Banach space, every residual set of ${\cal X}^1(M^d)$ is dense. Usually we can get a dense open property via a generic way.

One knows that lower semi-continuous maps and upper semi-continuous maps defined on a complete metric space are continuous on a residual set.

\begin{Lemma}\label{Lem:continuitychainclass}

For $C^1$ generic $X\in{\cal X}^1(M^d)$, for every critical element $p$, $C(p,X)$ is continuous at $X$. This means, if $\{X_n\}$ is a sequence of vector fields and $\lim_{n\to\infty}X_n=X$ in the $C^1$ topology, then $\lim_{n\to\infty}C(p_{X_n},X_n)=C(p,X)$ in the Hausdorff topology.
\end{Lemma}

\begin{proof}
The proof of this lemma just uses the upper semi-continuity of chain recurrent class.
\end{proof}

\begin{Lemma}\label{Lem:twochainrecurrentclass}

For $C^1$ generic $X\in{\cal X}^1(M^d)$, if $p_1$ and $p_2$ are two different critical elements, and if $C(p_1)=C(p_2)$, then there is a $C^1$ neighborhood $\cal U$ of $X$ such that for any $Y\in\cal U$, one has $C(p_{1,Y},Y)=C(p_{2,Y},Y)$.

\end{Lemma}

\begin{proof}
Let $\cal C$ be the metric space of all compact subsets of $M^d$, endowed with the Hausdorff metric. $\cal C$ is
a compact metric space. Let ${\cal B}_1,~{\cal B}_2,~\cdots,{\cal B}_n,\cdots,$ be a countable basis of $\cal
C$. Let ${\cal O}_1,~{\cal O}_2,~\cdots,{\cal O}_n,\cdots,$ be the list of finite unions of elements of the countable basis. For each $n$ and $m$, we define the sets ${\cal H}_{n,m}$ and ${\cal N}_{n,m}$ of vector fields as following:

\begin{itemize}

\item $X\in {\cal H}_{n,m}$ iff there is a neighborhood $\cal U$ of $X$ such that for any
$Y\in\cal U$, for any hyperbolic critical element $p_n\in {\cal O}_n$ and any hyperbolic critical element $p_m\in {\cal O}_m$, one has
$C(p_n)=C(p_m)$ for $Y$.

\item $X\in {\cal N}_{n,m}$
iff there is a hyperbolic critical element $p_n\in {\cal O}_n$ and a hyperbolic critical element $p_m\in {\cal O}_m$ such
that $C(p_n)\cap C(p_m)=\emptyset$. Since every chain recurrent class is upper semi continuous, one knows that
${\cal N}_{n,m}$ is open.

\end{itemize}

From the definitions, ${\cal H}_{n,m}\cup {\cal N}_{n,m}$ is open and dense in ${\cal X}^1(M^d)$. Let
$${\cal R}=\bigcap_{n,m\in\NN}({\cal H}_{n,m}\cup{\cal N}_{n,m}).$$

$\cal R$ is a residual subset. We will verify that
every $X\in\cal R$ satisfies the the conclusion of the lemma. Let
$X\in\cal R$. Thus, for each $n$ and $m$ one has $X\in {\cal
H}_{n,m}\cup{\cal N}_{n,m}$. For any two hyperbolic critical
elements $p_1$ and $p_2$, there are $l\in\NN$ and $k\in\NN$ and a
$C^1$ neighborhood $\cal U$ of $X$ such that
\begin{itemize}
\item For any $Y\in\cal U$, ${\rm Orb}_Y(p_{1,Y})$ and ${\rm Orb}_Y(p_{2,Y})$ are the maximal
compact invariant sets in ${\cal O}_l$ and ${\cal O}_k$
respectively.

\end{itemize}

If $C(p_1)=C(p_2)$ for $X$, then $X\notin {\cal N}_{k,l}$. This
implies that $X\in {\cal H}_{k,l}$. Let ${\cal U}_0={\cal U}\cap
{\cal H}_{k,l}$. For each $Y\in{\cal U}_0$, one has
\begin{itemize}
\item Since $Y\in {\cal H}_{k,l}$, there is a critical point $p_1'\in {\cal O}_l$ and a
critical point $p_2'\in {\cal O}_k$ of $Y$, one has $C(p_1')=C(p_2')$.

\item Since $Y\in\cal U$, the unique critical point in ${\cal O}_l$
is $p_{1,Y}$ and the unique critical point in ${\cal O}_k$ is $p_{2,Y}$. As a corollary, $p_{1,Y}=p_1'$ and
$p_{2,Y}=p_2'$.

\end{itemize}
Thus, one has $C(p_{1,Y})=C(p_{2,Y})$ for any $Y\in{\cal U}_0$.
\end{proof}

\begin{Lemma}\label{Lem:unstablecontained}

For $C^1$ generic $X\in{\cal X}^1(M^d)$, and any hyperbolic critical element $p$ of $X$, if
$\overline{W^u(p)}\subset C(p)$, then there is a $C^1$ neighborhood $\cal U$ of $X$ such that
$\overline{W^u(p_Y,Y)}\subset C(p_Y,Y)$.

\end{Lemma}

\begin{proof}
Let $\cal C$ be the metric space of all compact subsets of $M^d$, endowed with the Hausdorff metric. $\cal C$ is
a compact metric space. Let ${\cal B}_1,~{\cal B}_2,~\cdots,{\cal B}_n,\cdots,$ be a countable basis of $\cal
C$. Let ${\cal O}_1,~{\cal O}_2,~\cdots,{\cal O}_n,\cdots,$ be the list of finite unions of elements of the countable basis. For each $n$, one can define sets ${\cal H}_n$ and ${\cal N}_n$ as following:

\begin{itemize}

\item $X\in {\cal H}_n$ iff there is a $C^1$ neighborhood
$\cal U$ of $X$ such that for any $Y\in\cal U$, every hyperbolic critical orbit ${\rm Orb}_Y(p_Y)\in {\cal
O}_n$ of $Y$ satisfies $\overline{W^u(p_Y,Y)}\subset C(p_Y,Y)$. By definition, ${\cal H}_n$ is open.

\item $X\in{\cal N}_n$ iff $X$ has a hyperbolic critical
element $p\in{\cal O}_n$ such  that $\overline{W^u(p,X)}\nsubseteqq C(p,X)$. $\overline{W^u(p,X)}$ varies lower
semi-continuously with respect to $X$ and $C(p,X)$ varies upper semi-continuously with respect to $X$. So if
$\overline{W^u(p,X)}\nsubseteqq C(p,X)$, there is a neighborhood $\cal U$ of $X$ such that
$\overline{W^u(p_Y,Y)}\nsubseteqq C(p_Y,Y)$ for any $Y\in\cal U$. This would imply that ${\cal N}_n$ is an open set in ${\cal
X}^1(M^d)$.

\end{itemize}

It's clear that ${\cal H}_n\cup{\cal N}_n$ is open and dense in ${\cal X}^1(M^d)$. Let
$${\cal R}=\bigcap_{n\in\NN}({\cal H}_n\cup{\cal N}_n).$$

${\cal R}$ is a residual subset of ${\cal X}^1(M^d)$. Take $X\in{\cal R}$. If $p$ is a hyperbolic critical
element of $X$, then there are $n$ and a neighborhood $\cal U$ of $X$ such that for each $Y\in\cal U$, ${\rm Orb}_Y(p_Y)$ is
the unique hyperbolic critical orbit in ${\cal O}_n$. Since $\overline{W^u(p)}\subset C(p)$, one has $X\notin
{\cal N}_n$. As a corollary, $X\in{\cal H}_n$. Let ${\cal U}_0={\cal U}\cap {\cal H}_n$. For any $Y\in{\cal U}_0$,
\begin{itemize}

\item Since $Y\in{\cal H}_n$, every hyperbolic critical element $q\in{\cal
O}_n$ of $Y$, one has $\overline{W^u(q,Y)}\subset C(q,Y)$.

\item Since $Y\in\cal U$, the unique critical element in ${\cal O}_n$ is $p_Y$.

\end{itemize}

Thus, $\overline{W^u(p_Y,Y)}\subset C(p_Y,Y)$ for any $Y\in{\cal U}_0$.
\end{proof}

The following lemma claims that for $C^1$ generic vector field, Lyapunov stability is a robust property under perturbations.

\begin{Lemma}\label{Lem:stablepropertyofquasiattractor}
For $C^1$ generic $X\in{\cal X}^1(M^d)$, let $p$ be a hyperbolic critical element of $X$. If $C(p,X)$ is a
Lyapunov stable chain recurrent class, then there is a $C^1$ neighborhood $\cal U$ of $X$ such that for any weak
Kupka-Smale vector field $Y\in\cal U$, $C(p_Y,Y)$ is a Lyapunov stable chain recurrent class of $Y$.
\end{Lemma}

\begin{proof}
Let ${\cal R}\subset {\cal X}^1(M^d)$ be the residual subset as in Lemma \ref{Lem:unstablecontained}: $X\in\cal R$ iff for any hyperbolic critical element $p$ of $X$, if $\overline{W^u(p,X)}\subset C(p,X)$, then there is a $C^1$ neighborhood ${\cal U}={\cal U}_{X,p}$ such that $\overline{W^u(p_Y,Y)}\subset C(p_Y,Y)$ for any $Y\in\cal U$. We will prove that $C(p_Y,Y)$ is Lyapunov stable chain recurrent class for each weak Kupka-Smale $Y\in\cal U$. If not, there is a weak Kupka-Smale vector field $X_0\in\cal U$ such that $C(p_{X_0},X_0)$ is not Lyapunov stable. Thus we have
\begin{itemize}

\item $\overline{W^u(p_{X_0},X_0)}\subset C(p_{X_0},X_0)$.

\item There is $y\notin C(p_{X_0},X_0)$ such that $y\in W^{ch,u}(C(p_{X_0},X_0))$.

\end{itemize}

Then, there is a $C^1$ neighborhood ${\cal U}_0\subset{\cal U}$ of $X_0$ such that for any $Y\in{\cal U}_0$,
$y\notin C(p_Y,Y)$ by the upper semi-continuity of chain recurrent classes. Choose $z\in W^u(p_{X_0},X_0)\setminus\{p_{X_0}\}$. Since $z\in C(p_{X_0},X_0)$, $y$ is chain attainable from
$z$. Choose a small neighborhood $U$ of $p_{X_0}$.

%
%
%
%
%
%
%
%

One can assume that the negative orbit of $z$ is contained in $U$. By Lemma \ref{Lem:connecting}, there is a vector field
$Y\in {\cal U}_0$ such that
\begin{itemize}
\item $Y(x)=X_0(x)$ for any $x\in U$.

\item $y$ is in the positive orbit of $z$ with respect to $\phi_t^Y$.

\end{itemize}

As a corollary, $y$ is in the unstable manifold of $p_Y$ with respect to $Y$. This fact gives a contradiction.
\end{proof}

The following lemma asserts that for a generic vector field, if the perturbed system has a periodic orbit which is $(C,\eta,d,{\cal N})$-contracting at the period, then the original generic system already have itself by relaxing the constants.

\begin{Lemma}\label{Lem:uniformgeneric}
There is a dense $G_\delta$ set ${\cal G}\subset{\cal X}^1(M^d)$ such that for any $X\in{\cal G}$, if for any two open sets $U,V\subset M^d$ wit $\overline U\subset V$, for any $C^1$ neighborhood ${\cal U}$ of $X$, and any three number $K\in\NN$, $\eta>0$ and $T>0$, if some $Y\in {\cal U}$ has a hyperbolic periodic orbit $\gamma_Y$ which is $(K,\eta,T,{\cal N})$-contracting at the period w.r.t. $\psi_t^Y$ such that $\gamma_Y\cap U\neq\emptyset$, then $X$ has a periodic orbit $\gamma$ which is $(K,\eta/2,2T,{\cal N})$-contracting at the period w.r.t. $\psi_t$ such that $\gamma\cap V\neq\emptyset$.

\end{Lemma}
\begin{proof}
The proof of this lemma is an application of fundamental tricks for generic properties. Thus we just give a sketch.  Let $O_1,O_2,\cdots,O_n,\cdots$ be a topological base of $M^d$. Let $\{\eta_m\}$ and $\{d_{\ell}\}$ be the sets of positive rational numbers. For each $n\in\NN$, $K\in\NN$, $m\in\NN$ and $\ell\in\NN$, one defines:
\begin{itemize}

\item $X\in {\cal H}_{n,K,m,\ell}$ iff there is a hyperbolic periodic orbit $\gamma$ such that

\begin{itemize}

\item $\gamma\cap O_n\neq\emptyset$.

\item There is $x\in\gamma$ and a time partition $0=t_0<t_1<\cdots<t_q=\alpha\pi(\gamma)$ for some positive integer $\alpha$ satisfying $t_{i+1}-t_i\le d_\ell$ for $0\le i\le q-1$ such that
$$\prod_{i=0}^{q-1}\|\psi_{t_{i+1}-t_i}|_{{\cal N}_{\phi_{t_i}(x)}}\|< K{\rm e}^{-\eta_m \alpha\pi(\gamma)}.$$

\end{itemize}

\item $X\in {\cal N}_{n,K,m,\ell}$ iff there is a neighborhood ${\cal U}$ of $X$ such that for any $Y\in{\cal U}$, one has

\begin{itemize}

\item either $Y$ has no hyperbolic periodic orbit intersecting $O_n$;
\item Or, $Y$ has a hyperbolic periodic orbit $\gamma$ such that $\gamma\cap O_n\neq\emptyset$ and for any $x\in \gamma$ and any time partition $0=t_0<t_1<\cdots<t_q=\alpha\pi(\gamma)$ for some positive integer $\alpha$ satisfying $t_{i+1}-t_i\le d_\ell$ for $0\le i\le q-1$, one has
$$\prod_{i=0}^{q-1}\|\psi^{Y}_{t_{i+1}-t_i}|_{{\cal N}_{\phi^Y_{t_i}(x)}}\|\ge K{\rm e}^{-\eta_m \alpha\pi(\gamma)}.$$

\end{itemize}

\end{itemize}

By the definitions, ${\cal H}_{n,K,m,\ell}\cup{\cal N}_{n,K,m,\ell}$ is a dense open set in ${\cal X}^1(M^d)$. Thus,

$${\mathcal G}=\bigcap_{n,K,m,\ell\in\NN}({\cal H}_{n,K,m,\ell}\cup{\cal N}_{n,K,m,\ell})$$
is a residual subset of ${\cal X}^1(M^d)$. Now we check that every $X\in{\mathcal G}$ satisfies the properties of the lemma.

For any $\eta>0$, $T>0$, one can take a rational number $\eta_{m_0}\in (\eta/2,\eta)$ and $T_{\ell_0}\in (T, 2T)$. If there is a sequence of vector fields $\{X_n\}$ such that
\begin{itemize}

\item $\lim_{n\to\infty}X_n=X$.

\item Each $X_n$ has a hyperbolic periodic orbit $\gamma_n$ which is $(K,\eta,T,{\cal N})$-contracting at the period such that $\gamma_n\cap U\neq\emptyset$.

\end{itemize}

There is $x\in {\overline U}$ such that $x$ is an accumulating point of $\gamma_n$. Thus, there is $n_0$ such that $x\in O_{n_0}\subset V$. Since $X\in {\cal G}\subset {\cal H}_{n_0,K,m_0,\ell_0}\cup {\cal N}_{n_0,K,m_0,\ell_0}$, one has $X\in {\cal H}_{n_0,K,m_0,\ell_0}$ by the definitions. Thus $X$ has a periodic orbit in $O_{n_0}$ itself which is $(K,\eta_{m_0},T_{\ell_0},{\cal N})$-contracting at the period w.r.t. $\psi_t$ . It's $(K,\eta/2,2T,{\cal N})$-contracting at the period w.r.t. $\psi_t$ in $V$.
\end{proof}

\begin{Lemma}\label{Lem:uniformsinks}
There is a dense $G_\delta$ set ${\cal G}\subset{\cal X}^1(M^d)$ such that for any $X\in{\cal G}$ and $x\in M^d$, for any
$K\in\NN$, $\eta>0$ and $d>0$, one has
\begin{itemize}
\item either, $x$ is contained in a periodic sink which is $(K,\eta/2,2d,{\cal N})$-contracting at the period w.r.t. $\psi_t$;

\item or, there is a $C^1$ neighborhood ${\cal U}$ of $X$ and a neighborhood $U$ of $x$ such that for any $Y\in\cal U$, $Y$ has no periodic sink which is $(K,\eta,d,{\cal N})$-contracting at the period w.r.t. $\psi_t^Y$, and intersects $U$.
\end{itemize}

\end{Lemma}

\begin{proof}
Let $\cal G$ be as in Lemma~\ref{Lem:uniformgeneric}. If the conclusion of this lemma is not true, one has that
\begin{itemize}

\item $x$ is not contained in a periodic sink which is $(K,\eta/2,2d,{\cal N})$-contracting at the period w.r.t. $\psi_t$.

\item For any  $C^1$ neighborhood ${\cal U}$ of $X$ and any neighborhood $U$ of $x$, some $Y\in\cal U$ has a periodic sink, which is $(K,\eta,d,{\cal N})$-contracting at the period w.r.t. $\psi_t^{Y}$, and intersects $U$ for.
\end{itemize}

Thus, by Lemma~\ref{Lem:uniformgeneric}, for any neighborhood of $x$, $X$ itself has a periodic sink, which is $(K,\eta/2,2d,{\cal N})$-contracting at the period w.r.t. $\psi_t$. In other words, there is a sequence of periodic points $x_n$ such that
\begin{itemize}

\item $\lim_{n\to\infty}x_n=x$.

\item $x_n$ is contained in a period sink $\gamma_n$, which is $(K,\eta,d,{\cal N})$-contracting at the period for each $n$. Moreover, $\{\gamma_n\}$ are distinct.

\end{itemize}

We assert that $\tau(\gamma_n)\to\infty$ as $n\to\infty$. Otherwise, by taking a limit, $x$ would be in a periodic sink which is $(K,\eta,d,{\cal N})$-contracting at the period. Thus one can get a contradiction by Lemma~\ref{Lem:finiteunifsink}.
\end{proof}

\begin{Corollary}\label{Cor:dominatedinneighborhood}
Assume that $\dim M^3=3$. There is a residual set ${\cal G}\subset{\cal X}^1(M^3)\setminus\overline{{\cal HT}}$
and $\iota>0$ such that for any $X\in\cal G$, any $\sigma\in{\rm Sing}(X)$, there is a $C^1$ neighborhood ${\cal
U}$ of $X$ and a neighborhood $U$ of $\sigma$ such that for any periodic orbit $\gamma$ of $Y$, if $\gamma\cap
U\neq\emptyset$, then ${\cal N}_\gamma$ admits an $\iota$-dominated splitting of index 1 w.r.t. the linear
Poincar\'e flow.

\end{Corollary}

\begin{proof}
If ${\rm ind}(\gamma)=1$, then it is done by Lemma~\ref{Lem:dominatedsplitting}. Thus, one can assume that $\gamma$ is a sink or source. Without loss of generality, we assume that it is a sink.
By Lemma~\ref{Lem:unifromordominated}, if $\gamma$ dose not admits an $\iota$-dominated splitting for some $\iota$, then there are $C>0$, $\eta>0$ and $T>0$ such that $\gamma$ is $(C,\eta,T,{\cal N})$-contracting at the period w.r.t. $\psi_t$. Since $\sigma$ is a singularity, not a periodic point, by Lemma~\ref{Lem:uniformsinks}, one can get a contradiction.
\end{proof}

We would like to list some other generic results we need in this paper.
\begin{Lemma}\label{Lem:generic}
There is a dense $G_\delta$ set ${\cal G}\subset{\cal X}^1(M^d)$ such that for each $X\in{\cal G}$, one has
\begin{enumerate}


\item For any non-trivial chain recurrent class $C(\sigma)$, where $\sigma$
is a hyperbolic singularity of index $d-1$, then every separatrix of $W^u(\sigma)$ is dense in $C(\sigma)$. Especially, $C(\sigma)$ is transitive.

\item Let $i\in[0,\dim M-1]$. If there is a sequence of vector fields $\{X_n\}$ such that
\begin{itemize}

\item $\lim_{n\to\infty}X_n=X$,

\item each $X_n$ has a hyperbolic periodic orbits $\gamma_{X_n}$ of index $i$ such that $\lim_{n\to\infty}\gamma_{X_n}=\Lambda$,

\end{itemize}

Then there is a sequence of  hyperbolic periodic orbits $\gamma_n$ of index $i$ \emph{of $X$} such that
$\lim_{n\to\infty}\gamma_n=\Lambda$.

\item There exists a neighborhood $\cal U$ of $X$ such that for any $Y\in \cal U$, $Y$ has only finitely many singularities. Moreover, for every singularity $\sigma$ of $Y$, the eigenvalues $\lambda_1, \lambda_2, \cdots, \lambda_d$ of $DY(\sigma)$ satisfy:
    $$
    {\rm Re}(\lambda_i)+{\rm Re}(\lambda_j)\not=0,
    $$
    for any $1\le i, j\le d$.

\item For any hyperbolic periodic orbit $P$ of $X$, then $C(P)=H(P)$, where $H(P)$ is the homoclinic class of $P$.

\item Every non-trivial chain transitive set is the limit of a sequence of periodic orbits in the Hausdorff topology.

\item $X$ is Kupka-Smale.
\end{enumerate}

\end{Lemma}

\begin{Remark}
Item 1 is a corollary of the connecting lemma for pseudo-orbits \cite{BoC04}. There is no explicit version like this. \cite[Section 4]{MoP03} gave some ideas about the proof of Item 1 without using of the terminology of chain recurrence. Item 2 is fundamental, one can see \cite{Wen04} for instance. Item 3 is fundamental. It is true because generic $X$ can only have finitely many singularities. Moreover, the eigenvalues of the singularities have some continuous property. Item 4 is also a result in \cite{BoC04}. Item 5 is the main result in \cite{Cro06}. Item 6 is the classical Kupka-Smale theorem.
\end{Remark}

{\bf Let ${\cal G}_0$ be a dense $G_\delta$ set of ${\cal X}^1(M^d)$ such that $X\in{\cal G}_0$ iff $X$ satisfies all generic properties in this subsection.}

\section{Reduce the proof into technical results}\label{Sec:reduction}

Now we proof Theorem~\ref{Thm:main} under the condition that Theorem~\ref{Thm:homoclinicclass} is true.

\begin{proof}[Proof of Theorem~\ref{Thm:main}] Now we give the proof of Theorem A by assuming the result of Theorem
~\ref{Thm:homoclinicclass}. Suppose on the contrary that ${\cal
X}^1(M^3)\setminus\overline{{\cal MS}\cup {\cal HS}}$ is not empty. Choose $C^1$ generic $X\in{\cal
X}^1(M^3)\setminus\overline{{\cal MS}\cup {\cal HS}}$. Since every homoclinic tangency of a hyperbolic periodic orbit can be perturbed to be a transverse homoclinic intersection by an arbitrarily small perturbation, we have that $X$ is far away from ones with a homoclinic tangency.
Thus, by Theorem~\ref{Thm:homoclinicclass}, every chain recurrent class is a homoclinic class. Since we assume that Theorem~\ref{Thm:main} is not true, one has that every chain recurrent class is trivial: it is reduced to be a critical element. If there are finitely many chain recurrent classes, then we have that $X$ is Morse-Smale. Thus, one can assume that $X$ has infinitely many chain recurrent classes, and each chain recurrent class is a hyperbolic critical element. It is known that $C^1$ generic vector field can only have finitely many singularities since it is Kupka-Smale. Thus $X$ has countably many distinct hyperbolic periodic orbits $\{\gamma_n\}$. By taking a subsequence if necessary, we can assume that $\{\gamma_n\}$ converges in the Hausdorff topology. Let $\Lambda$ be the limit. Then, $\Lambda$ is chain transitive, which implies that $\Lambda$ is contained in a chain recurrent class. Because we know that every chain recurrent class of $X$ is a hyperbolic critical element, $\Lambda$ is a hyperbolic critical element. This cannot happen because hyperbolic critical elements are locally maximal.
\end{proof}

To prove Theorem~\ref{Thm:homoclinicclass}, we need to prove the following Theorem~\ref{Thm:dichotomysingularclass} and \ref{Thm:dominationimplyperiodic}.

\begin{Theorem}\label{Thm:dichotomysingularclass}
For $C^1$ generic $X\in {\cal X}^1(M^3)\setminus\overline{\cal HT}$, for the non-trivial chain recurrent class $C(\sigma)$ of some singularity $\sigma$, if $C(\sigma)$ is not a homoclinic class, then $C(\sigma)$ admits a dominated splitting $T_{C(\sigma)}M^3=E\oplus F$ w.r.t. the tangent flow.
\end{Theorem}

\begin{Theorem}\label{Thm:dominationimplyperiodic}
For $C^1$ generic $X\in {\cal X}^1(M^3)$, for any non-trivial chain recurrent class $C(\sigma)$ of some singularity $\sigma$, if $C(\sigma)$ admits a dominated splitting $T_{C(\sigma)}M^3=E\oplus F$ w.r.t. the tangent flow, then $C(\sigma)$ is a homoclinic class.
\end{Theorem}

In general, we give the definition of singular hyperbolic sets as the following:\footnote{We postpone to give the definition because we want the introduction to be easier to read.}
\begin{Definition}
Let $\Lambda$ be a compact invariant set of $X\in{\cal X}^1(M^3)$, $E\subset T_\Lambda M^3$ a two dimensional invariant sub-bundle, we say that $E$ is \emph{area-contracting}, if there are $C>0$, $\lambda>0$ such that for any $x\in\Lambda$ and any $t>0$, $|{\rm det} \Phi_t|_{E(x)}|\le C\mathrm{e}^{-\lambda t}$; we say that $E$ is \emph{area-expanding} if it is area-contracting for $-X$.

A compact invariant set $\Lambda$ of $X\in{\cal X}^1(M^3)$ is called \emph{singular hyperbolic}, if
\begin{itemize}

\item either, $\Lambda$ admits a partially hyperbolic splitting $T_\Lambda M^3=E^{ss}\oplus E^{cu}$, where $\dim E^{ss}=1$, $E^{ss}$ is contracting and $E^{cu}$ is area-expanding;

\item or, $\Lambda$ admits a partially hyperbolic splitting $T_\Lambda M^3=E^{cs}\oplus E^{uu}$, where $\dim E^{uu}=1$, $E^{uu}$ is expanding and $E^{cs}$ is area-contracting.
\end{itemize}

\end{Definition}

Note that in the above definition, we don't require $\Lambda$ contains a singularity or not. So every non-singular hyperbolic set is singular hyperbolic.

We also needs the following two results from \cite{BGY11}.
\begin{Lemma}\label{Lem:dominatedwithoutsingularity}
For $C^1$ generic $X\in{\cal X}^1(M^3)$, if $\Lambda$ is a non-singular chain transitive set and admits a dominated splitting ${\cal N}_\Lambda=\Delta^s\oplus\Delta^u$ with respect to $\psi_t$, then $\Lambda$ is hyperbolic.
\end{Lemma}

\begin{Theorem}\label{Thm:partialtosingular}
For $C^1$ generic $X\in {\cal X}^1(M^3)$, if the chain recurrent class $C(\sigma)$ of a singularity $\sigma$ contains {\bf a periodic orbit} and
admits a dominated splitting $T_{C(\sigma)}M=E\oplus F$ with respect to $\Phi_t$,
then $C(\sigma)$ is singular hyperbolic. Consequently, $C(\sigma)$ is an attractor or repeller according to
the index of $\sigma$ equals to $2$ or $1$.
\end{Theorem}

\begin{proof}[Proof of Theorem~\ref{Thm:homoclinicclass}]

If Theorem~\ref{Thm:homoclinicclass} is not true, then there is a $C^1$ generic $X\in{\cal X}^1(M^3)\setminus\overline{\cal HT}$ and a chain recurrent class $\cal C$ of $X$ such that $\cal C$ is not a homoclinic class. Now we have two cases:

\paragraph{$C$ contains no singularities.} Since $\cal C$ is chain transitive, there is a sequence of periodic orbits $\{\gamma_n\}$ such that $\gamma_n\to \cal C$ as $n\to\infty$ in the Hausdorff metric. Since $C$ is not reduced to a periodic orbit, one can assume that $\{\gamma_n\}$ are distinct periodic orbits and $\tau(\gamma_n)\to\infty$.
By Corollary~\ref{Cor:sequenceofsinks}, if we cannot perturb $\gamma_n$ of hyperbolic periodic orbit of index 1 for $n$ large enough, then there are constants $C>0$, $T>0$, $\eta>0$ such that
$\gamma_n$ are $(C,\eta,T,{\cal N})$-contracting at the period w.r.t. $\psi_t$ for $X$ or $-X$. Then by Lemma~\ref{Lem:finiteunifsink}, one can get a contradiction.


Thus, one can assume that the index of every $\gamma_n$ is $1$. From this, we have that $C$ admits a dominated splitting ${\cal N}_{C}=\Delta^{cs}\oplus \Delta^{cu}$ of index $1$ w.r.t. $\psi_t$. By Lemma~\ref{Lem:dominatedwithoutsingularity}, we have that $C$ is hyperbolic. This fact shows that $C$ is a homoclinic class, which gives a contradiction.

\paragraph{$C$ contains a singularity $\sigma$.} We assume that $C$ is non-trivial. Since $C$ is not a homoclinic class, by Theorem~\ref{Thm:dichotomysingularclass}, $C$ admits a dominated splitting $T_C M^3=E\oplus F$ w.r.t. the tangent flow $\Phi_t$.  Then by Theorem~\ref{Thm:dominationimplyperiodic}, $C$ is a homoclnic class, which gives a contradiction.

\end{proof}

\begin{proof}[Proof of Theorem~\ref{Thm:dominationtosingular}]
For $C^1$ generic $X\in{\cal X}(M^3)$, if $C(\sigma)$ admits a dominated splitting w.r.t. the tangent flow, then by by Theorem~\ref{Thm:dominationimplyperiodic}, $C$ is a homoclnic class. By using Theorem~\ref{Thm:partialtosingular}, $C(\sigma)$ is singular hyperbolic.
\end{proof}

\subsection{Proof of Theorem~\ref{Thm:dichotomysingularclass}}

The proof of Theorem~\ref{Thm:dichotomysingularclass} can be divided into the following two propositions:
\begin{Proposition}\label{Pro:farfromtangency}
There is a dense ${G}_\delta$ set ${\cal G}\in{\cal X}^1(M^3)\setminus\overline{\cal HT}$ such that for any
$X\in{\cal G}$, if $\sigma$ is a hyperbolic saddle of $X$ and $C(\sigma)$ is Lyapunov stable, then  every
singularity $\rho\in C(\sigma)$ is Lorenz-like, i.e., the eigenvalues $\lambda_1, \lambda_2, \lambda_3$ of
$DX(\rho)$ satisfy:
    $$
    \lambda_1<\lambda_2<0<-\lambda_2<\lambda_3.
    $$
Moreover, there is a $C^1$ neighborhood ${\cal U}_X$ of $X$ such that for any $Y\in{\cal U}_X$ and any $\rho\in
C(\sigma)\cap{\rm Sing}(X)$, one has $\rho_Y\in C(\sigma_Y,Y)$ and $W^{ss}(\rho_Y)\cap C(\sigma_Y,Y)=\{\rho_Y\}$.
\end{Proposition}

\begin{Proposition}\label{Pro:mixingdominated}
Let $\Lambda$ be a compact invariant set of $X\in{\cal X}^1(M^d)$
verifying the following properties:

\begin{itemize}
\item $\Lambda\setminus{\rm Sing}(X)$ admits an index $i$ dominated splitting
${\cal N}_{\Lambda\setminus{\rm Sing}(X)}=\Delta^{cs}\oplus\Delta^{cu}$ in the normal bundle w.r.t. the linear Poincar\'e flow $\psi_t$.

\item Every singularity $\sigma\in\Lambda$ is hyperbolic and ${\rm ind}(\sigma)>i$. Moreover, $T_\sigma M^d$
admits a partially hyperbolic splitting $T_{\sigma}M^d=E^{ss}\oplus E^{cu}$ with respect to the tangent flow,
where $\dim E^{ss}=i$ and for the corresponding strong stable manifolds $W^{ss}(\sigma)$, one has
$W^{ss}(\sigma)\cap \Lambda=\{\sigma\}$.

\item For every $x\in\Lambda$, one has $\omega(x)\cap {\rm Sing}(X)\neq\emptyset$.

\end{itemize}

Then one has
\begin{itemize}
\item either $\Lambda$ admits a partially hyperbolic splitting $T_\Lambda M^d=E^{ss}\oplus F$ with respect to the
tangent flow $\Phi_t$, where $\dim E^{ss}=i$,

\item or, there is a sequence of hyperbolic periodic orbits $\{\gamma_n\}_{n\in\NN}$ of index $i$ such that
\begin{itemize}

\item $\tau(\gamma_n)\to\infty$,

\item $H(\gamma_n)\cap\Lambda\neq\emptyset$ for each $n\in\NN$,

\item There is $T>0$ such that for $x_n\in\gamma_n$,
$$\lim_{n\to\infty}\frac{1}{[\tau(\gamma_n)/T]}\sum_{i=0}^{[\tau(\gamma_n)/T]-1}\log\|\psi_T|_{{\cal N}^s(x_n)}\|=0.$$
\end{itemize}
\end{itemize}

\end{Proposition}

Now one can give a proof of Theorem~\ref{Thm:dichotomysingularclass} by Proposition~\ref{Pro:farfromtangency} and
Proposition~\ref{Pro:mixingdominated}.

\begin{proof}[Proof of Theorem~\ref{Thm:dichotomysingularclass}]
Since $\dim M^3=3$, without loss generality, one can assume that
${\rm ind}(\sigma)=2$. Otherwise, one considers $-X$. By Lemma~\ref{Lem:generic}, $C(\sigma)$ is Lyapunov
stable. By Proposition~\ref{Pro:farfromtangency}, every singularity $\rho$ in $C(\sigma)$ is Lorenz-like and
$W^{ss}(\rho)\cap C(\sigma)=\{\rho\}$. Since $X\in {\cal X}^1(M^3)\setminus\overline{{\cal HT}}$,
one has
\begin{itemize}

\item The normal bundle of $C(\sigma)\setminus{\rm Sing}(X)$ admits a dominated splitting with respect to the
linear Poincar\'e flow (See more details from Corollary~\ref{Cor:dominatedinneighborhood}).

\item $C(\sigma)$ contains no periodic points. As a corollary, for every regular point $x\in C(\sigma)$,
$\omega(x)$ contains a singularity. Otherwise, if $\omega(x)$ contains no singularities, then by Lemma~\ref{Lem:dominatedwithoutsingularity}, $\omega(x)$ is hyperbolic. Then one can get a periodic orbit in $C(\sigma)$ by the shadowing lemma, which is a contradiction.
\end{itemize}

By Proposition~\ref{Pro:mixingdominated}, either $C(\sigma)$ admits a partially hyperbolic splitting, or $C(\sigma)\cap H(\gamma)\neq\emptyset$ for some hyperbolic saddle $\gamma$. But the fact that $C(\sigma)\cap H(\gamma)\neq\emptyset$ gives a contradiction.
\end{proof}

\subsection{The proof of Theorem~\ref{Thm:dominationimplyperiodic}}

The proof of Theorem~\ref{Thm:dominationimplyperiodic} involves more notations and definitions. It's not easy to state the precise results here. But we can give some rough idea here.

\begin{Proposition}\label{Pro:roughcrosssection}
For $C^1$ generic $X\in{\cal X}^1(M^3)$, for a Lyapunov stable chain recurrent class $C(\sigma)$, if $C(\sigma)$ contains no periodic orbits and $C(\sigma)$ admits a partially hyperbolic splitting $T_{C(\sigma)}M^3=E^{ss}\oplus E^{cu}$ with $\dim E^{ss}=1$, then $C(\sigma)$ admits a cross section system $(\Sigma,F)$. Moreover, this cross section system $(\Sigma,F)$ has some continuous properties w.r.t. $X$.

\end{Proposition}

The precise definition of cross section system could be seen in Section~\ref{Sec:birthhomo}.

\begin{Proposition}\label{Pro:roughsinkbasin}

For $C^1$ generic $X\in{\cal X}^1(M^3)$ and $\sigma\in\mathrm{Sing}(X)$, if the chain recurrent class $C(\sigma)$ is Lyapunov stable, contains no periodic orbits and admits a partially hyperbolic splitting $T_{C(\sigma)}M^3=E^{ss}\oplus E^{cu}$ with $\dim E^{ss}=1$, then for any $C^1$ neighborhood $\cal U$ of $X$, there is a weak Kupka-Smale vector field $Y\in\cal U$ such that $Y$ has a periodic sink $\gamma$, and $C(\sigma_Y)\cap \overline{B(\gamma)}\neq\emptyset$, where $B(\gamma)$ is the basin of $\gamma$ and $\overline{B(\gamma)}$ is the closure of $B(\gamma)$.

\end{Proposition}

The proof of Proposition~\ref{Pro:roughsinkbasin} needs to use the cross section system constructed in Proposition~\ref{Pro:roughcrosssection}

Proposition~\ref{Pro:roughsinkbasin} could imply Theorem~\ref{Thm:dominationimplyperiodic}: Lemma~\ref{Lem:stablepropertyofquasiattractor} implies for every weak Kupka-Smale vector field $Y$ close to $X$, $C(\sigma_Y)$ is Lyapunov stable for $Y$; By the definitions, Lyapunov stable set cannot intersect the closure of the basin of some sink when it is not a sink itself.

%
%
%
%
%
%
\section{Partial hyperbolicity: the proof of Theorem \ref{Thm:dichotomysingularclass}}\label{Sec:lorenz-like}

\subsection{Lorenz-like singularities}

\begin{Lemma}\label{Lem:strongstablesingularity}
Let $\Lambda$ be a non-trivial chain transitive set of $X\in{\cal X}^1(M^d)$. We assume
\begin{itemize}

\item Every singularity in $\Lambda$ is hyperbolic.

\item ${\cal N}_{\Lambda\setminus{\rm Sing}(X)}$ admits a dominated splitting of index $i$ w.r.t the
linear Poincar\'e flow $\psi_t$.
\end{itemize}

Then for every hyperbolic singularity $\sigma\in\Lambda$ with ${\rm ind}(\sigma)>i$, $T_\sigma M^d$ admits a dominated splitting $T_\sigma M^d=E^{ss}\oplus E^{cu}$ with respect to the tangent flow $\Phi_t$, where $\dim E^{ss}=i$.
\end{Lemma}

\begin{proof}

For a hyperbolic singularity $\sigma\in\Lambda$, since $\Lambda$ is a non-trivial chain transitive set, one has
$W^{s}(\sigma)\cap\Lambda\setminus\{\sigma\}\neq\emptyset$ and
$W^{u}(\sigma)\cap\Lambda\setminus\{\sigma\}\neq\emptyset$. One can see \cite[Lemma 2.6]{BGY11} for more details about the proof. Recall the definition of $\widetilde{\Lambda}$: the
lift of $\Lambda$ in the sphere bundle as in Subsection~\ref{Sub:flows}. One has that there is $v\in E^u(\sigma)\cap\widetilde{\Lambda}$.

Since ${\cal N}_{\Lambda\setminus{\rm Sing}(X)}$ admits a dominated splitting of index $i$ with respect to the
linear Poincar\'e flow, $\widetilde{\cal N}_{\widetilde{\Lambda}}$ admits a dominated splitting of index $i$ with respect
to the extended
linear Poincar\'e flow $\widetilde{\psi}_t$. We will consider the negative limit set $\alpha(v)$ with
respect to the flow $\Phi^I_t$.

By changing the Remainnain metric in a small neighborhood of $\sigma$, without loss of generality, one can assume
$E^s(\sigma)\perp E^{u}(\sigma)$. Thus,
\begin{itemize}

\item ${\widetilde {\cal N}}_{\alpha(v)}$ admits a dominated splitting $\Delta^{cs}\oplus \Delta^{cu}$ of index $i$ w.r.t $\widetilde{\psi}_t$ since ${\cal N}_{\Lambda\setminus{\rm Sing}(X)}$ admits a dominated splitting of index $i$ with respect to $\psi_t$.

\item The hyperbolic splitting on $T_\sigma M^d$ implies that: ${\widetilde {\cal N}}_{\alpha(v)}$ admits a dominated splitting of index $E^s\oplus F$ of index $\dim
E^s$ for some $F$ w.r.t. $\widetilde{\psi_t}$, since $\sigma$ is hyperbolic and $E^s(\sigma)\perp E^u(\sigma)$.
\end{itemize}

Thus by the properties of dominated splittings, one has $\Delta^{cs}(\alpha(v))\subset E^s(\sigma)$. Thus,
there are $C>0$ and $\lambda>0$ such that for any $u\in E^u\cap S_\sigma M^d$ and a splitting
$\Delta^{cs}\oplus\Delta^{cu}={\widetilde {\cal N}}_{(\sigma,u)}$ with the following property:
\begin{itemize}
\item $$\frac{\|\widetilde{\psi}_t|_{\Delta^{cs}(u)}\|} {\|\Phi_t|_{<u>}\|}\le C\mathrm{e}^{-\lambda t},~~~\forall t\ge 0.$$

\item $$\frac{\|\widetilde{\psi}_t|_{\Delta^{cs}(u)}\|}{m(\widetilde{\psi}_t|_{\Delta^{cu}(u)})}\le C\mathrm{e}^{-\lambda t},~~~\forall t\ge 0.$$

\end{itemize}

Then by Lemma~\ref{Lem:mixed}, we know that $\sigma$ admits a dominated splitting $T_{\sigma}M^d=E\oplus F$ w.r.t. the tangent flow $\Phi_t$, where $\dim E=i$. Since ${\rm ind}\sigma>i$, one can get the splitting as required.
\end{proof}

\begin{Lemma}\label{Lem:strongstableoutside}
There is a residual subset ${\cal G}\subset {\cal X}^1(M^3)\setminus\overline{\cal HT}$ such that for any
hyperbolic singularity $\sigma$ of index $2$ of $X\in{\cal G}$, if $C(\sigma)$ is non-trivial, then $T_\sigma M^3$ admits a dominated splitting $T_\sigma
M^3=E^{ss}\oplus F$ w.r.t. the tangent flow $\Phi_t$, where $\dim E^{ss}=1$,
$E^{ss}$ is contracting, and $W^{ss}(\sigma)\cap C(\sigma)=\{\sigma\}$.

Similarly, if ${\rm ind}(\sigma)=1$ and $C(\sigma)$ is non-trivial, then $W^{uu}(\sigma)\cap
C(\sigma)=\{\sigma\}$.
\end{Lemma}

\begin{proof}
Assume that $X$ satisfies all generic properties in \S 3.
We focus on the case of ${\rm ind}(\sigma)=2$. By Lemma~\ref{Lem:generic}, there is a sequence of periodic orbit $\gamma_n$ such that $\lim_{n\to\infty}\gamma_n=C(\sigma)$ in the Hausdorff topology. By Corollary~\ref{Cor:dominatedinneighborhood}, there is $\iota>0$ such that each $\gamma_n$ admits an $\iota$-dominated splitting in ${\cal N}_{\gamma_n}$ w.r.t. the linear Poincar\'e flow. Then by Corollary~\ref{Cor:dominatedonweaktransitive}, ${\cal
N}_{C(\sigma)\setminus{\rm Sing}(X)}$ admits a dominated splitting of index $1$ with respect to the linear
Poincar\'e flow $\psi_t$. As a corollary of Lemma~\ref{Lem:strongstablesingularity}, $T_\sigma M$ has dominated splitting $T_{\sigma}M=E^{ss}\oplus F$ w.r.t. the tangent flow $\Phi_t$. Thus what we need to prove now is $W^{ss}(\sigma)\cap
C(\sigma)=\{\sigma\}$. We will prove this by absurd. If this is not true, there is $x_0\in W^{ss}(\sigma)\cap
C(\sigma)\setminus\{\sigma\}$. One also notice that $C(\sigma)$ is Lyapunov stable since ${\rm ind}(\sigma)=2$ by Lemma~\ref{Lem:generic}.
Hence $W^u(\sigma)\subset C(\sigma)$ and $W^u(\sigma)$ is dense in $C(\sigma)$. By changing the Riemainning metric in a small neighborhood, we may assume that $E^{ss}(\sigma)$,
$E^{cs}(\sigma)$ and $E^u(\sigma)$ are mutually orthogonal, where $E^{cs}(\sigma)=E^s(\sigma)\cap F$.

Thus, by using the $C^1$ connecting lemma (Lemma~\ref{Lem:wenxia}), there is an arbitrarily small perturbation
$Y$ of $X$ such that

\begin{itemize}
\item $Y$ has a strong connection with respect to $\sigma$: there is $y\in M^3$ such that $y\in
W^{ss}(\sigma)\cap W^u(\sigma)\setminus\{\sigma\}$.

\item $Y(x)=X(x)$ if $x$ is in a small neighborhood of $\sigma$.

\end{itemize}

By an extra small perturbation, one can assume that $Y$ has the following extra properties:

\begin{itemize}

\item $Y$ is linear in a small neighborhood of $\sigma$ in some local chart;

\item $E^{ss}(\sigma,Y)$, $E^{cs}(\sigma,Y)$ and $E^u(\sigma,Y)$ are still mutually orthogonal.
\end{itemize}

Let $P$ be the plane spanned by $E^{ss}$ and $E^u$ in the local chart. One has that $P$ is locally
invariant: there is a neighborhood $O_1$ of $\sigma$ such that for any $x\in P\cap O_1$, if
$\phi_{[0,t]}^Y(x)\in O_1$, then $\phi_t^Y(x)\in P$. Now for $Y$, by an extra perturbation, there is a sequence
of vector fields $Y_n$ and  a smaller neighborhood $O_2$ of $\sigma$ such that
\begin{itemize}
\item $\lim_{n\to\infty}Y_n=Y$;

\item For each $n$, $Y_n=Y$ in $O_2$;

\item $Y_n$ has a periodic orbit $\gamma_n$ such that $\gamma_n\cap O_2\subset P$
 and $\lim_{n\to\infty}\gamma_n={\rm Orb}(y,Y)$.
\end{itemize}

By Corollary~\ref{Cor:dominatedinneighborhood},  one has

\begin{itemize}

\item There are $\iota=\iota(X)$ such that ${\cal N}_{\gamma_n}$ has an $\iota$-dominated splitting
$\Delta^{cs}\oplus \Delta^{cu}$ of index 1 with respect to the linear Poincar\'e flow $\psi_t^{Y_n}$.

\end{itemize}

Thus, ${\cal N}_{{\rm Orb}(y,Y)}$ admits an $\iota$-dominated splitting w.r.t. $\psi_t^Y$. Let $\Gamma_Y={\rm
Orb}(y,Y)\cup\sigma_Y$. It is a compact invariant set. Over the lift $\widetilde{\Gamma_Y}$ (see
Subsection~\ref{Sub:flows}), there exists an dominated splitting ${\widetilde {\cal N}}_{{\widetilde \Gamma_Y}}=\Delta^{cs}\oplus \Delta^{cu}$ w.r.t. $\widetilde{\psi}_t$ such that
\begin{itemize}
\item For $v^{u}\in E^u(\sigma)\cap {\widetilde \Gamma_Y}$, one has $\Delta^{cs}(v^u)=E^{ss}$
and $\Delta^{cu}(v^u)=E^{cs}$.
\item For $v^{ss}\in E^{ss}(\sigma)\cap {\widetilde \Gamma_Y}$, one has $\Delta^{cs}(v^{ss})=E^{cs}$
and $\Delta^{cu}(v^{ss})=E^{u}$.
\end{itemize}

Then by following the analysis in the proof of Lemma 4.3 in \cite[page 255-256]{LGW05}, we will get a contradiction.
\end{proof}

\begin{Corollary}\label{Cor:uniqueindexofsingularity}

There is a residual subset ${\cal G}\subset {\cal X}^1(M^3)\setminus\overline{\cal HT}$ such that if a chain
recurrent class contains singularities, then all the singularities in the chain recurrent class have the same
index.

\end{Corollary}

\begin{proof}
Let ${\cal G}={\cal G}_0\setminus\overline{\cal HT}$, where ${\cal G}_0$ is as in the end of Subsection~\ref{Sub:generic}. We will
prove this corollary by absurd. If it's not true, then there is $X\in{\cal G}$ and a chain recurrent class $C$
of $X$ such that $C$ contains singularities of different indices. Thus, one can assume that $C$ contains two
singularities $\sigma_1$ and $\sigma_2$ satisfying ${\rm ind}(\sigma_1)=1$ and ${\rm ind}(\sigma_2)=2$. Since by
Lemma~\ref{Lem:strongstableoutside}, $T_{\sigma_1}M^3=E^{cs}\oplus E^{uu}$ is a dominated splitting w.r.t.
$\Phi_t$, where $\dim E^{uu}=1$, and for the corresponding strong unstable manifold $W^{uu}(\sigma_1)\cap
C=\{\sigma_1\}$. By Lemma~\ref{Lem:generic}, $C$ is Lyapunov stable since $C$ contains $\sigma_2$. As a
corollary, one has $W^{uu}(\sigma_1)\subset C$. This gives us a contradiction.
\end{proof}

\begin{Corollary}\label{Cor:robuststongstable}

There is a residual subset ${\cal G}\subset {\cal X}^1(M^3)\setminus\overline{\cal HT}$ such that for any
hyperbolic singularity $\sigma$ of index $2$ of $X\in{\cal G}$, there is a $C^1$ neighborhood $\cal U$ of
$X$ such that for any $Y\in{\cal U}$ and for any singularity $\rho\in C(\sigma)$, one has
\begin{itemize}

\item ${\rm ind}(\rho)=2$ and $\rho_Y\in C(\sigma_Y)$.

\item $T_\rho M^3=E^{ss}\oplus E^{cu}$ is a dominated splitting w.r.t. $\Phi_t$, where $\dim E^{ss}=1$.

\item For the corresponding stable manifolds of $E^{ss}$, one has $W^{ss}(\rho_Y,Y)\cap C(\sigma_Y)=\{\rho_Y\}$

\end{itemize}

\end{Corollary}

\begin{proof}

This is true just because that $X$ is $C^1$ generic and the continuous property of the local strong stable manifolds.
\end{proof}

Furthermore, we have

\begin{Theorem}\label{Thm:Lorenz-like}
For a generic $X\in \cX^1(M^3)\setminus\overline{\cal HT}$, and a hyperbolic singularity $\sigma$ of index 2, if
$C(\sigma)$ is non trivial, then $\sigma$ is Lorenz-like for $X$, i.e., the eigenvalues $\lambda_1, \lambda_2,
\lambda_3$ of $DX(\sigma)$ are all real and satisfy
    $$
    \lambda_1<\lambda_2<0<-\lambda_2<\lambda_3. \hspace{5cm} (*)
    $$
\end{Theorem}
\begin{proof}
First by Lemma~\ref{Lem:strongstableoutside}, for the three eigenvalues $\lambda_1,
\lambda_2, \lambda_3$ of $DX(\sigma)$, they are all real and
$$ \lambda_1<\lambda_2<0<\lambda_3.$$
What's left is to prove that $\lambda_2+\lambda_3>0$. The three corresponding eigenspaces are denoted by
$E^{ss}(\sigma)$, $E^{cs}(\sigma)$ and $E^{u}(\sigma)$. By changing the Riemannian metric in a small
neighborhood of $\sigma$, we can assume that they are mutually orthogonal. We will prove this by absurd, i.e.,
we assume that $\lambda_2+\lambda_3\le 0$. Since $X$ is $C^1$ generic, by Lemma~\ref{Lem:generic} one has
$\lambda_2+\lambda_3<0$. Moreover,

\begin{itemize}
\item $W^u(\sigma)\subset C(\sigma)$ and $W^u(\sigma)$ is dense in $C(\sigma)$ by Lemma~\ref{Lem:generic}.

\item $W^{s}(\sigma)\cap C(\sigma)\setminus\{\sigma\}\neq\emptyset$ since $C(\sigma)$ is non-trivial.

\end{itemize}

By using Lemma~\ref{Lem:wenxia} (the $C^1$ connecting lemma), for any $C^1$ neighborhood $\cal U$ of
$X$, there is $Y\in\cal U$ such that

\begin{itemize}
\item $Y$ has a homoclinic orbit $\Gamma$ associated to $\sigma_Y$.

\item For the three eigenvalues $\lambda_{1}^Y<\lambda_2^Y<0<\lambda_3^Y$ of $DY(\sigma_Y)$,
one still has $\lambda_2^Y+\lambda_3^Y<0$.

\item $W^{ss}(\sigma_Y)\cap C(\sigma_Y)=\{\sigma_Y\}$ by Corollary~\ref{Cor:robuststongstable}.

\end{itemize}

By simple perturbations, there is a sequence of vector fields $Y_n$ such that
\begin{itemize}

\item $\lim_{n\to\infty}Y_n=Y$.

\item Each $Y_n$ has a hyperbolic periodic orbit $\gamma_n$ such that $\lim_{n\to\infty}\gamma_n=\Gamma\cup\sigma$.

\end{itemize}

Under our assumption, $\{\gamma_n\}$ could be sinks or saddles. Let $C>0$, $\eta>0$, $\iota>0$ be as in Lemma~\ref{Lem:unifromordominated}. By Corollary~\ref{Cor:dominatedinneighborhood}
${\cal N}_{\gamma_n}$ admits an $\iota$-dominated splitting of index $1$ with respect to $\psi_t^{Y_n}$ for $n$ large enough.


 As a corollary, ${\cal
N}_{\Gamma}$ admits an $\iota$-dominated splitting w.r.t. $\psi_t^Y$. Thus, $\widetilde{\cal N}_{{\widetilde {\Gamma\cup\sigma}}}$
admits an $\iota$-dominated splitting ${\widetilde{\cal N}_{{\widetilde {\Gamma\cup\sigma}}}}=\Delta^{cs}\oplus\Delta^{cu}$ w.r.t.
$\widetilde{\psi}_t^Y$.

\begin{Claim}
For every $v\in T_\sigma M^3\cap {\widetilde {(\Gamma\cup\sigma)}}$, one has
$\Delta^{cs}(v)=E^{ss}(\sigma)$.

\end{Claim}

\begin{proof}
For each $v\in T_\sigma M^3$, if $v\in{\widetilde {\Gamma\cup\sigma}}$, then $v\in E^{cs}(\sigma)$ or $v\in
E^{u}(\sigma)$. Since $E^{ss}(\sigma)$, $E^{cs}(\sigma)$ and $E^u(\sigma)$ are mutually orthogonal, one has
\begin{itemize}
\item if $v\in E^{cs}(\sigma)$, since $\widetilde{\cal N}_{v}=E^{ss}(\sigma)\oplus E^u(\sigma)$ is a dominated
splitting w.r.t. $\widetilde{\psi}_t$, one has $\Delta^{cs}(v)=E^{ss}(\sigma)$;

\item if $v\in E^{u}(\sigma)$, since $\widetilde{\cal N}_{v}=E^{ss}(\sigma)\oplus E^{cs}(\sigma)$ is a dominated
splitting w.r.t. $\widetilde{\psi}_t$, one has $\Delta^{cs}(v)=E^{ss}(\sigma)$.
\end{itemize}

\end{proof}

Since the unique ergodic measure is supported on $\{\sigma\}$ for $\phi_t|_{\Gamma\cup\sigma}$, one has that
there are $C>0$ and $\lambda>0$ such that for any $t>0$ and any $(x,v)\in \widetilde{\Gamma\cup\sigma}$,
$$\frac{\|\widetilde{\psi}_t|_{\Delta^{cs}(x,v)}\|}{\|\Phi_t|_{<v>}\|}\le Ce^{-\lambda t}.$$
By Lemma~\ref{Lem:mixed}, $\Gamma\cup\sigma$ admits a dominated splitting $T_{\Gamma\cup\sigma}M^3=E\oplus F$
w.r.t. the tangent flow $\Phi_t$, where $\dim E=1$. Thus $E(\sigma)=E^{ss}(\sigma)$ by the uniqueness of
dominated splittings. Since the unique ergodic measure is supported on $\{\sigma\}$, one has that $E$ is
uniformly contracting. Thus $\gamma_n$ is also $(C,\eta,d,{\cal N})$-contracting at the period w.r.t. $\psi_t^{Y_n}$ for some
$C>0$, $d>0$ and $\eta\in(0,-(\lambda_2+\lambda_3))$ which depends on $X$ since $\gamma_n$ stays in a small neighborhood of the singularity for most time. This also gives a contradiction by Lemma~\ref{Lem:uniformsinks}.

\end{proof}

\subsection{Liao's shadowing lemma for $\psi_t^*$ and Liao's sifting lemma}

\begin{Definition}
Let $\eta>0$ and $T>0$. For any $x\in M^d\setminus\Sing(X)$ and $T_0>T$, the orbit arc $\phi_{[0,T_0]}(x)$
is called $(\eta,T)$-$\psi_t^*$-quasi hyperbolic with respect to a direct sum splitting $\cN_x=E(x)\oplus
F(x)$ if there is a partition
$$
0=t_0<t_1<\cdots<t_l=T,\quad t_{i+1}-t_i\le T,
$$
such that for $k=0, 1, \cdots, l-1$, we have
\begin{eqnarray*}
\prod_{i=0}^{k-1}\|\psi^*_{t_{i+1}-t_i}|_{\psi_{t_i}(E(x))}\|&\le& {\rm e}^{-\eta t_k},\\
\prod_{i=k}^{l-1}m(\psi^*_{t_{i+1}-t_i}|_{\psi_{t_i}(F(x))})&\ge& {\rm e}^{\eta (t_l-t_k)},\\
\frac{\|\psi^*_{t_{k+1}-t_k}|_{\psi_{t_k}(E(x))}\|}{m(\psi^*_{t_{k+1}-t_k}|_{ \psi_{t_k}(F(x))})}&\le& {\rm
e}^{-\eta(t_{k+1}-t_k)}.
\end{eqnarray*}

\end{Definition}

\begin{Remark}
The third inequality is usually satisfied in an invariant set with a $T^*$-dominated splitting in the normal
bundle with respect to the linear Poincar\'e flow.

\end{Remark}

Note that this definition is similar to the usual quasi hyperbolic orbit arc for linear Poincar\'e flow, while
the only difference is that now we use the scaled linear Poincar\'e flow $\psi_t^*$ instead of linear
Poincar\'e flow $\psi_t$. Let $d_T$ be the metric in $TM^d$ induced by the Riemanian metric. For $x,y\in M^d$
and two linear subspaces $E(x)$ and $F(y)$, one defines
$$\tilde{d}(E(x),F(y))=\sup_{\{u\in E(x),|u|=1\}}\inf_{v\in F(y),|v|=1}\{d_T(u,v)\}.$$

\begin{Theorem}\label{Thm:Liaoshadowing}
Let $X\in\cX^1(M^d)$, $\Lambda\subset M^d\setminus\Sing(X)$ be a compact set, and $\eta>0, T>1$. For any
$\varepsilon>0$ there exists $\delta>0$, such that for any $(\eta,T)$-$\psi_t^*$-quasi hyperbolic orbit arc
$\phi_{[0,T]}(x)$ with respect to some direct sum splitting $\cN_x=E(x)\oplus F(x)$ satisfying $x,
\phi_T(x)\in\Lambda$ and $\tilde d(E(x), \psi_T(E(x)))\le\delta$, there exists a point $p\in M^d$ and a $C^1$
strictly increasing function $\theta:[0,T]\to\RR$ such that

\begin{itemize}

\item $\theta(0)=0$ and $1-\varepsilon<\theta'(t)<1+\varepsilon$;

\item $p$ is periodic: $\phi_{\theta(T)}(p)=p$;

\item
$d(\phi_t(x),\phi_{\theta(t)}(p))\le\e|X(\phi_t(x))|,\quad t\in[0,T].$

\item There is a direct-sum splitting $\cN_p=E(p)\oplus F(p)$ such that $\psi^*_{\theta(T)}(E(p))=E(p)$,
$\psi^*_{\theta(T)}(F(p))=F(p)$, and for any $t\in[0,T]$,
$$\tilde{d}(\psi^*_t(E(x)),\psi^*_{\theta(t)}(E(p)))\le\varepsilon,$$
$$\tilde{d}(\psi^*_t(F(x)),\psi^*_{\theta(t)}(F(p)))\le\varepsilon.$$

\end{itemize}
\end{Theorem}

\begin{Remark}
If we consider an invariant set with a dominated splitting in the normal bundle w.r.t. the linear Poincar\'e
flow, we can replace $\tilde d(E(x), \psi_T(E(x)))\le\delta$ by $d(x,\phi_T(x))<\delta$.

\end{Remark}
Note that in this version of shadowing lemma, we only need that the head and tail of orbit arc are far from singularities, while other part of the orbit arc can approximate singularities. This enables us to deal with some problems where regular orbits approximate singularities.

We also need Liao's sifting lemma \cite{Lia80,Lia81}, whose aim is to find quasi-hyperbolic orbit segments. One can see \cite{Wen08} for a proof.
\begin{Lemma}\label{Lem:Liaosifting}

Let $\phi_t:\Lambda\to\Lambda$ be a continuous flow on a compact metric space $\Lambda$ and  $f:~\Lambda\to\RR$  a
continuous function. Let $\eta_2>\eta_1>0$ and $T>0$. Assume that
\begin{itemize}
\item there is $b\in\Lambda$ such that for any $n\in\NN$,
$$\sum_{i=0}^{n-1}f(\phi_{iT}(b))\ge 0;$$

\item for any $c\in\Lambda$ verifying for any $n\in\NN$,
$$\sum_{i=0}^{n-1}f(\phi_{iT}(c))\ge -n\eta_1,$$
there is $g\in\omega(c)$ such that for any $n\in\NN$,
$$\sum_{i=0}^{n-1}f(\phi_{iT}(g))\le -n\eta_2.$$
\end{itemize}

Then for any $\eta_3,\eta_4$ verifying $\eta_2>\eta_3>\eta_4>\eta_1$, for any $k\in\NN$, there is $y$ in the
positive orbit of $x$ and integers $0=n_0<n_1<\cdots<n_k$ such that for each integer $i\in[0,k-1]$, for any
integer $m\in[1,n_{i+1}-n_i]$ one has
$$\sum_{j=0}^{m-1}f(\phi_{jT}(\phi_{n_i T}(y)))\le -m\eta_4,$$

$$\sum_{j=m-1}^{n_{i+1}-n_i-1}f(\phi_{jT}(\phi_{(n_i+m-1)T}))\ge -(n_{i+1}-n_i-m+1)\eta_3.$$

\end{Lemma}

We need the following fundamental lemma to prove hyperbolicity for compact sets.

\begin{Lemma}\label{Lem:compact}
Let $\phi_t:\Lambda\to\Lambda$ be a continuous flow on a compact metric space $\Lambda$ and  $f:~\Lambda\to\RR$  a
continuous function.  Given $T>0$, if for any $x\in\Lambda$, there is $n_x\in\NN$ such that
$$\sum_{i=0}^{n_x-1}f(\phi_{i T}(x))<0,$$
then there are $C\ge 0$ and $\lambda<0$ such that for any $x\in\Lambda$ and any $n\in\NN$, one has
$$\sum_{i=0}^{n-1}f(\phi_{i T}(x))\le C+n\lambda.$$
\end{Lemma}

\subsection{Proof of Proposition~\ref{Pro:mixingdominated}}

Now we can give the proof of Proposition~\ref{Pro:mixingdominated}.
\begin{proof}[Proof of Proposition~\ref{Pro:mixingdominated}]
For proving this proposition, one assumes that the second case of the conclusion cannot occur.

By changing the Riemannian metric in a small neighborhood of singularities, one can assume that
$E^{ss}(\sigma)\perp E^{cu}(\sigma)$ for any singularity $\sigma\in\Lambda$. There is $T^*>0$ such that
\begin{itemize}

\item For any $\sigma\in\Lambda$ and any unit vector $v\in E^{cu}(\sigma)$, one has
$$\frac{\|\Phi_{T^*}|_{E^{ss}}\|}{|\Phi_{T^*}(v)|}\le \frac{1}{4}.$$

\item ${\cal N}_{\Lambda\setminus{\rm Sing}(X)}=\Delta^{cs}\oplus\Delta^{cu}$ is a $T^*$-dominated splitting w.r.t. $\psi_t$.

\end{itemize}

We consider $\widetilde{\Lambda}$: the lift of $\Lambda$ in the sphere bundle as in Subsection~\ref{Sub:flows}.
By considering the dynamics of $\chi_t$, one has ${\widetilde {\cal N}}_{\widetilde{\Lambda}}=\widetilde{\Delta^{cs}}\oplus
\widetilde{\Delta^{cu}}$ is a $T^*$-dominated splitting of index $i$ w.r.t. $\widetilde{\psi}_t={\rm
proj}_2(\chi_t)$ verifying the following property: $\Delta^{cs}(x)=\widetilde{\Delta^{cs}}(X(x)/|X(x)|)$ and
$\Delta^{cu}(x)=\widetilde{\Delta^{cu}}(X(x)/|X(x)|)$ for any regular point $x\in\Lambda$.

Since $W^{ss}(\sigma)\cap \Lambda=\{\sigma\}$, one has if $v\in \widetilde{\Lambda}\cap T_\sigma M^d$, then $v\in
E^{cu}(\sigma)$. On $\widetilde{\Lambda}$, one defines the function $\widetilde{\xi}$ by
\begin{eqnarray*}
\widetilde{\xi}:~\widetilde{\Lambda}& \to & {\mathbb R},\\
                   v&\mapsto& \log\|\widetilde{\psi}_{T^*}|_{\widetilde{\Delta^{cs}}(v)}\|-\log\|\Phi_{T^*}(v)\|.
\end{eqnarray*}

Since $\widetilde{\Delta}^{cs}$ is a continuous bundle, $\widetilde{\xi}$ is a continuous function.

On $\Lambda\setminus{\rm Sing}(X)$, one can define the function $\xi$ by
\begin{eqnarray*}
\xi:~\Lambda\setminus{\rm Sing}(X) & \to & {\mathbb R},\\
          {x} &  \mapsto & \log\|\psi_{T^*}|_{\Delta^{cs}(x)}\|-\log\|\Phi_{T^*}|_{<X(x)>}\|.
\end{eqnarray*}

By the definitions, for every regular point $x\in\Lambda$, $\xi(x)=\tilde{\xi}(X(x)/|X(x)|)$.
$\widetilde{\xi}$ is defined on a \emph{compact} set and $\xi$ is defined on a \emph{non-compact} set.

\begin{Claim}
There is $C>1$ and $0<\lambda<1$ such that for any $v\in\widetilde{\Lambda}$ and $n\in\NN$, one has
$$\frac{\|\widetilde{\psi}_{nT^*}|_{\widetilde{\Delta^{cs}}(v)}\|}{|\Phi_{nT^*}(v)|}\le C\lambda^n.$$
\end{Claim}

\begin{proof}[Proof of the Claim]
The claim is equivalent to the following statement: There are $C>1$ and
$0<\lambda<1$ such that for any $v\in\widetilde{\Lambda}$ and $n\in\NN$,
$$\sum_{i=0}^{n-1}\widetilde{\xi}(\Phi^I_{i T^*}(v))\le \log C+ n\log\lambda.$$

If the claim is not true, by Lemma~\ref{Lem:compact}, for any $n\in\NN$, there is $v_n\in\widetilde{\Lambda}$ such
that for any integer $\ell\in[1,n]$, one has
$$\sum_{j=0}^{\ell-1}\widetilde{\xi}(\Phi^I_{i T^*}(v_n))\ge 0.$$
Let $b\in\widetilde{\Lambda}$ be an accumulation point of $\{v_n\}$. Then for any $n\in\NN$,
$$\sum_{i=0}^{n-1}\tilde{\xi}(\Phi^I_{i T^*}(b))\ge 0.$$
Since we assume that $E^{ss}(\sigma)\perp E^{cu}(\sigma)$, one has for any $v\in T_\sigma M^d\cap \widetilde{\Lambda}$, $\widetilde{\Delta^{cs}}(v)=E^{ss}(\sigma)$ and $v\in E^{cu}$. As
a corollary, one has for any $n\in\NN$,
$$\sum_{i=0}^{n-1}\tilde{\xi}(\Phi^I_{i T^*}(v))\le -n\log4.$$
For every point $x\in\Lambda$, by assumptions, $\omega(x)$ contains a hyperbolic
singularity. Thus for its lift $\widetilde{\Lambda}$, for any point $c\in\widetilde{\Lambda}$, there is a
singularity $\sigma\in\Lambda$ and a unit vector $v\in E^{cu}(\sigma)$ such that $v\in\omega(c)$. Thus
for the function $\tilde{\xi}$ and the flow $\phi_t^I$, the conditions of Lemma~\ref{Lem:Liaosifting} (Liao's
sifting lemma) are satisfied.

For any four numbers $\lambda_1$, $\lambda_2$, $\lambda_3$ and $\lambda_4$ satisfying
$-\log4/2<\lambda_1<\lambda_2<\lambda_3<\lambda_4<0$. Let
$$\widetilde{\Lambda}_{\lambda_2}=\{v\in\widetilde{\Lambda}:~\tilde{\xi}(\Phi_{-T^*}^I(v))\ge \lambda_2\}.$$
Since there are only finite many singularities, there is $\varepsilon_0=\varepsilon_0(\lambda_2)>0$, such that for any singularity $\sigma\in\Lambda$, and any unit vector $v\in
T_\sigma M^d\cap\widetilde{\Lambda}$, $d_T(v,\widetilde{\Lambda}_{\lambda_2})\ge \varepsilon_0$. By Lemma~\ref{Lem:Liaosifting}, for any $k\in\NN$, there is
$u$ in the positive orbit of $b$ and integers $0=n_0<n_1<\cdots<n_k$ such that for each integer
$\ell\in[0,k-1]$, for any integer $m\in[1,n_{\ell+1}-n_\ell]$ one has
$$\sum_{j=0}^{m-1}\tilde{\xi}(\Phi_{jT^*}^I(\Phi_{n_i T^*}^I(u)))\le m\lambda_3,$$

$$\sum_{j=m-1}^{n_{\ell+1}-n_\ell-1}\tilde{\xi}(\Phi^I_{jT^*}(\Phi^I_{(n_\ell+m-1)T^*}(u)))\ge (n_{\ell+1}-n_\ell-m+1)\lambda_2.$$

Thus $\Phi^I_{n_\ell T^*}(y)\in \widetilde{\Lambda}_{\lambda_2}$ for $1\le \ell\le k$. Assume that
$b=X(x)/|X(x)|$ and $y=\Phi_{t_0}^I(b)$. Thus,
$$d(\phi_{n_i T^*}(\phi_{t_0}(x)),{\rm Sing}(X))\ge\varepsilon_0.$$
For any $\varepsilon>0$ small enough such that: for any regular point $\beta\in\Lambda$, for any point $\theta\in
B(\beta,\varepsilon|X(\beta)|)$, for any $T'\in[(1-\varepsilon)T^*,(1+\varepsilon)T^*]$, for any two subspaces
$G(\beta)\subset {\cal N}_\beta$ and $G(\theta)\subset{\cal N}_\theta$ satisfying $\tilde{d}(G(\beta),G(\theta))<\varepsilon$, one has
$$\lambda_1-\lambda_2\le\log\|\psi_{T^*}^*|_{G(\beta)}\|-\log\|\psi_{T'}^*|_{G(\theta)}\|\le\lambda_4-\lambda_3.$$
$$\max\{\frac{\|\psi_{T'}|_{G(\theta)}\|}{\|\psi_{T^*}|_{G(\beta)}\|},\frac{m(\psi_{T^*}|_{G(\beta)})}{m(\psi_{T'}|_{G(\theta)})}\}\le \frac{\sqrt 2}{2}.$$

For this $\varepsilon>0$, there is $\delta=\delta(\varepsilon)$ as in Theorem~\ref{Thm:Liaoshadowing} (Liao's
shadowing). There is $k_\delta\in\NN$ such that for any $k_\delta$ points
$\{x_1,x_2,\cdots,x_{k_{\delta}}\}\subset\widetilde{\Lambda}_{\lambda_2}$, there are $1\le i_1<i_2\le k_\delta$ such that
$d(x_{i_1},x_{i_2})<\delta$. For this $k_\delta$, there are $n_1<n_2<\cdots<n_{k_\delta}$ and a point $y'\in {\rm
Orb}^+(x)$ such that for the function $\xi$ and $0\le \ell\le k_\delta-1$, one has

$$\sum_{j=0}^{m-1}{\xi}(\phi_{jT^*}(\phi_{n_\ell T^*}(y')))\le m\lambda_3,$$

$$\sum_{j=m-1}^{n_{\ell+1}-n_\ell-1}{\xi}(\phi_{jT^*}(\phi_{(n_\ell+m-1)T^*}(y')))\ge (n_{\ell+1}-n_\ell-m+1)\lambda_2.$$

Let $y_\ell=\phi_{n_\ell T^*}(y')$. By the dominated properties, one has for each $y_\ell$, for any integer
$m\in[1,n_{i+1}-n_i]$, one has
$$\prod_{j=0}^{m-1}\|\psi_{T^*}^*|_{\Delta^{cs}}(\phi_{jT^*}(y_\ell))\|\le e^{m\lambda_3},$$

$$\prod_{j=m-1}^{n_{\ell+1}-n_\ell-1}\|\psi_{-T^*}^*|_{\Delta^{cu}(\phi_{j+1}(y_\ell))}\|\le (\frac{1}{4})^{n_{\ell+1}-n_\ell-m}e^{-(n_{\ell+1}-n_\ell-m)\lambda_2}.$$

Let $\eta=\min\{-\lambda_3,\log4+\lambda_2\}$. One has that
\begin{itemize}
\item $\phi_{[0,(n_{\ell+1}-n_\ell)T^*]}(y_\ell)$ is an $(\eta,T^*)$-$\psi_t^*$-quasi hyperbolic strings.

\item $d(y_i,{\rm Sing}(X))>\varepsilon$.

\end{itemize}

By the choice of $k_\delta$, there are $y_\alpha$ and $y_\beta$ such that $d(y_\alpha,y_\beta)<\delta$. Thus by
Theorem~\ref{Thm:Liaoshadowing}, the orbit segment from $y_\alpha$ to $y_\beta$ can be shadowed: there is a
periodic orbit $P_\varepsilon$ with period $\tau(P_\varepsilon)$ and $p_\varepsilon\in P_\varepsilon$ and a
monotone increasing function $\theta_\varepsilon(t)$ such that
\begin{itemize}

\item $d(\phi_{\theta_\varepsilon(t)}(p_\varepsilon),\phi_t(y_\alpha))<\varepsilon|X(\phi_t(y_\alpha))|$ for any $0\le t\le
(n_\beta-n_\alpha)T^*$.

\item $1-\varepsilon\le \theta_\varepsilon'(t)\le 1+\varepsilon$ and $\theta_\varepsilon((n_\beta-n_\alpha)T^*)=\tau(P_\varepsilon)$.

\item There is a direct-sum splitting $\cN_{p_\varepsilon}=E(p_\e)\oplus F(p_\e)$ such that $\psi^*_{\theta((n_\beta-n_\alpha)T^*)}(E(p_\e))=E(p_\e)$,
$\psi^*_{\theta((n_\beta-n_\alpha)T^*)}(F(p_\e))=F(p_\e)$, and for any $t\in[0,(n_\beta-n_\alpha)T^*]$,
$$\tilde{d}(\psi^*_t(E(y_\alpha)),\psi^*_{\theta(t)}(E(p_\e)))\le\varepsilon,$$
$$\tilde{d}(\psi^*_t(F(y_\alpha)),\psi^*_{\theta(t)}(F(p_\e)))\le\varepsilon.$$

\end{itemize}

Let $\eta'=\min\{-\lambda_4,\log2+\lambda_1\}$ and $T'=(1+\varepsilon)T$. By the choosing of $\varepsilon$, one has that $P_\varepsilon$
is an $(\eta',T')$-$\psi_t^*$-quasi hyperbolic string. Take a sequence $\{1/q\}_{q\in\NN}$. For each $q$,
there is a periodic orbit $P_{1/q}$ with period $\tau(P_{1/q})$ such that
\begin{itemize}

\item $\limsup_{q\to\infty}P_{1/q}\subset\Lambda$.

\item There is $p_{1/q}\in P_{1/q}$ such that $d(p_{1/q},{\rm Sing}(X))\ge\varepsilon_0$, $p_{1/q}$ is $(1,\eta',T',E)$-$\psi_t^*$-contracting
and $(1,\eta',T',F)$-$\psi_t^*$-expanding.

\item $P_{1/q}$ admits a $T^*$-dominated splitting w.r.t. $\psi_t^*$.
\end{itemize}

Without loss of generality, one can assume that $\{p_{1/q}\}_{q\in\NN}$ converges. By Lemma~\ref{Lem:intersect}\footnote{This is the part that we need the existence of central plaques as in Section 2.},
for any two $w,m\in\NN$ large enough, one has
$$W^{s}(P_{1/w})\pitchfork W^u(P_{1/m})\neq\emptyset,~~~W^{u}(P_{1/w})\pitchfork W^s(P_{1/m})\neq\emptyset.$$
In other words, they are homoclinically related. Thus for $w$ large enough, one has $\Lambda\cap
H(P_{1/w})\neq\emptyset$. Moreover one has
$$\frac{1}{[\tau(P_{1/w})/T']}\sum_{i=0}^{[\tau(P_{1/w})/T']-1}\log\|\psi_{T'}|_{{\cal N}^s(p_{1/w})}\|\ge\lambda_1.$$

By the arbitrary property of $\lambda_1$, we have the second case.
\end{proof}

By Lemma~\ref{Lem:mixed}, and the claim above, one has $\Lambda$ admits a dominated splitting $T_\Lambda
M=E\oplus F$ w.r.t. the tangent flow, where $\dim E=i$ and $X(x)\subset F(x)$ for any regular point
$x\in\Lambda$. We assume that it is a $T$-dominated splitting. We define the function $f:~\Lambda\to \RR$ by
$f(x)=\log\|\Phi_T|_{E(x)}\|$.

\begin{Claim}

$$\liminf_{n\to\infty}\frac{1}{n}\sum_{\ell=0}^{n-1}f(\phi_{\ell T}(x))<0,~~~\forall x\in\Lambda.$$
\end{Claim}

\begin{proof}
For every point $x\in\Lambda$, there are two cases: either $\omega(x)\subset{\rm Sing}(X)$, or $\omega(x)$
contains a regular point $a\in\Lambda$. In the first case, since every singularity in $\Lambda$ admits a
partially hyperbolic splitting, one has that the claim is true. Now we consider the second case: $\omega(x)$
contains a regular point $a$. We fix a neighborhood $U_a$ of $a$ such that for any $z,y\in U_a$, one has
$$\frac{1}{2}\le\frac{|X(z)|}{|X(y)|}\le 2.$$
Let $f_0(z)=\log\|\Phi_T|_{E(z)}\|-\log\|\Phi_T|_{<X(z)>}\|$ for every regular point $z\in\Lambda$. Since
$T_\Lambda M^d=E\oplus F$ is a dominated splitting and $X(z)\subset F(z)$ for every regular point $z\in\Lambda$,
one has $f_0(z)\le-\log2$ for any regular point $z$. There is $x_0\in{\rm Orb}^+(x)$ such that $x_0\in U_a$.
Since $a\in\omega(x)$, there is a sequence of times $\{t_n\}_{n\in\NN}$ such that
$\lim_{n\to\infty}\phi_{t_n}(x_0)=a$ and $\lim_{n\to\infty}t_n=\infty$.

For $n$ large enough, we assume that $t_n=kT+t$, where $t\in[0,T]$. Thus we have
$$\frac{1}{k}\sum_{\ell=0}^{k-1}f(\phi_{\ell T}(x_0))\le -\log2+\frac{1}{k}(\log\|\Phi_T|_{<X(\phi_{kT}(x_0))>}\|-\log\|\Phi_T|_{X(x_0)}\|).$$
Since $k\to\infty$ as $n\to\infty$, one has
$$\liminf_{n\to\infty}\frac{1}{n}\sum_{\ell=0}^{n-1}f(\phi_{\ell T}(x_0))<0.$$
Thus the same inequality holds for $x$.
\end{proof}

By the above claim, one has $\sum_{\ell=0}^{n-1}f(\phi_{\ell T}(x))<0$ for any $x\in\Lambda.$ By
Lemma~\ref{Lem:compact}, we have that $E$ is uniformly contracting. This ends the proof of this proposition.
\end{proof}

\section{Perturbations in partially hyperbolic quasi attractor}\label{Sec:birthhomo}
In this section, we will manage to prove Theorem~\ref{Thm:dominationimplyperiodic}. To prove Theorem~~\ref{Thm:dominationimplyperiodic}, we notice that \cite[Theorem B]{BGY11} proved that\footnote{The statement is a little bit stronger than \cite{BGY11}. But the proof is contained there.}
\begin{Theorem}\label{Thm:dominatedtopartial}
For $C^1$ generic $X\in {\cal X}^1(M^3)$, for the non-trivial chain recurrent class $C(\sigma)$ of some singularity $\sigma$, if $C(\sigma)$ admits a dominated splitting $T_{C(\sigma)}M^3=E\oplus F$ w.r.t. the tangent flow, then $C(\sigma)$ admits a partially hyperbolic splitting; more precisely, ${\rm ind}(\sigma)=2$ iff $\dim E=1$ and $E$ is contracting.

\end{Theorem}

Thus, what we need to prove is

\begin{Theorem}\label{Thm:partialtoperiodic}

For $C^1$ generic $X\in {\cal X}^1(M^3)$, for the non-trivial chain recurrent class $C(\sigma)$ of some singularity $\sigma$, if $C(\sigma)$ admits a dominated splitting $T_{C(\sigma)}M^3=E^s\oplus F$ w.r.t. the tangent flow, where $E^s$ is one-dimensional and contracting, then $C(\sigma)$ is a homoclinic class.

\end{Theorem}

\subsection{Cross sections of partially hyperbolic Lyapunov stable chain recurrent classes}

Our aim in this subsection is to prove Proposition~\ref{Pro:roughcrosssection}. We also need some other lemmas to get an open set of vector fields, which guarantee that we can do a perturbation in that open set.

For a singularity $\sigma$ of $X\in{\cal X}^1(M^3)$, recall that $\sigma$ is \emph{Lorenz-like} if the eigenvalues $\lambda_1,\lambda_2,\lambda_3$ of
$DX(\sigma)$ satisfy
$$\lambda_1<\lambda_2<0<-\lambda_2<\lambda_3.$$

Since we not only need to consider $C(\sigma)$, but also its continuation, we give the notion of \emph{pseudo-Lorenz} to characterize its properties.

\begin{Definition}
A compact invariant set $\Lambda$ of $X\in{\cal X}^1(M^3)$ is
\emph{pseudo-Lorenz} if
\begin{itemize}

\item $\Lambda$ is Lyapunov stable.

\item $\Lambda$ has a partially hyperbolic splitting $T_\Lambda M=E^{ss}\oplus E^{cu}$ with respect to $\Phi_t$, where $\dim E^{ss}=1$.

\item Every singularity in $\Lambda$ is Lorenz-like.

\item For every singularity $\sigma\in\Lambda$,
$W^{ss}(\sigma)\cap \Lambda=\{\sigma\}$.

\end{itemize}
\end{Definition}

This notion is a generalization of singular hyperbolic transitive attractor \cite{MPP04}.

\begin{Proposition}\label{Pro:continuitypseudoLorenz}
For $C^1$ generic $X\in{\cal X}^1(M^3)$ and $\sigma\in {\rm Sing}(X)$, if $C(\sigma)$ is pseudo-Lorenz, then there is a $C^1$ neighborhood
$\cal U$ of $X$ such that for any $Y\in\cal U$,
\begin{itemize}

\item $C(\sigma_Y,Y)$ admits a partially hyperbolic splitting $T_{C(\sigma_Y,Y)}M^3=E^{ss}\oplus E^{cu}$, where $\dim
E^{ss}=1$.

\item Every singularity in $C(\sigma_Y,Y)$ is the continuation of some singularity in $C(\sigma)$, and Lorenz-like.

\item For every singularity $\rho\in C(\sigma_Y,Y)$, one has $W^{ss}(\rho)\cap C(\sigma_Y,Y)=\{\rho\}$.

\end{itemize}

Moreover, if $Y$ is weak Kupka-Smale, then $C(\sigma_Y,Y)$ is Lyapunov stable and hence it is pseudo-Lorenz.

\end{Proposition}

\begin{proof}
Assume that $X$ satisfies the generic properties in Subsection~\ref{Sub:generic}. We will
prove that such an $X$ satisfies the proposition.

Since $X$ is Kupka-Smale, $X$ has only finitely many hyperbolic singularities $\{\sigma_i\}_{i=1}^m$. Moreover, there is a neighborhood ${\cal U}_0$ such that for any $Y\in{\cal
U}_0$, ${\rm Sing}(Y)=\{\sigma_{i,Y}\}_{i=1}^m$, where $\sigma_{i,Y}$ is the continuations of $\sigma_i$. By Lemma \ref{Lem:twochainrecurrentclass}, there is a $C^1$ neighborhood ${\cal U}_1\subset{\cal U}_0$ such that for any $Y\in{\cal U}_1$, such that for any $1\le i,j\le m$, we have
$C(\sigma_{i,Y},Y)=C(\sigma_{j,Y},Y)$ iff $C(\sigma_i,X)=C(\sigma_j,X)$.

Given a pseudo-Lorenz set $C(\sigma)$ of $X$, we assume that $C(\sigma)\cap{\rm Sing}(X)=\{\sigma_1,\sigma_2,\cdots,\sigma_k\}$. Since the eigenvalues of $\sigma_{i,Y}$ vary continuously with $Y$ (\cite[2.18 Proposition]{PdM82}) and $C(\sigma)$ is pseudo-Lorenz, there exists a neighborhood ${\cal U}_2\subset {\cal U}_1$ such that for any $Y\in{\cal U}_2$, the index of every singularity in $C(\sigma_Y,Y)$ is 2, $C(\sigma_Y,Y)\cap{\rm Sing}(Y)=\{\sigma_{1,Y},\sigma_{2,Y},\cdots,\sigma_{k,Y}\}$ and every singularity in $C(\sigma_Y,Y)$ is Lorenz-like. Since $C(\sigma_Y, Y)$ is upper-continuous with $Y$, by shrinking ${\cal U}_2$, we have that for every $Y\in{\cal U}_2$, $C(\sigma_Y,Y)$ has a partially hyperbolic splitting with respect to the tangent flow $\Phi_{t}^Y$:
$$
T_{C(\sigma_Y,Y)}M=E^{ss}_Y\oplus E^{cu}_Y,
$$
such that $\dim E^{ss}_Y=1$ and $E^{ss}_Y$ is uniformly contracting.

We have that, by shrinking ${\cal U}_2$ if necessary, for any $Y\in{\cal U}_2$, $W^{ss}_{loc}(\sigma_{i,Y})\cap C(\sigma_Y,Y)=\{\sigma_{i,Y}\}$. In fact, suppose on the contrary that there exists a sequence $Y_n\to X$ such that $W^{ss}_{loc}(\sigma_{i,Y_n})\cap C(\sigma_{Y_n},Y_n)\not=\{\sigma_{i,Y_n}\}$. Then one branch of $W^{ss}_{loc}(\sigma_{i,Y_n})$ is contained in
$\{\sigma_{i,Y_n}\}$. By the continuity of local strong stable manifolds of singularities, letting $n\to \infty$, we get that one branch of $W^{ss}_{loc}(\sigma_{i})$ is also contained in $C(\sigma)$, which is a contradiction.

Finally, according to Lemma
\ref{Lem:stablepropertyofquasiattractor}, $C(\sigma_Y,Y)$ is Lyapunov stable if $Y$ is weak Kupka-Smale by
shrinking ${\cal U}_2$ again if necessary.
\end{proof}

For $C^1$ generic $X\in{\cal X}^1(M^3)$, it is proved that every Lyapunov stable chain recurrent class $C$ is a quasi attractor by \cite{BoC04}, i.e., there exists a sequence of neighborhoods $U_n$ of $C$ such that $\cap_{n\ge 1} U_n=C$ and $\phi_1({\overline U}_n)\subset {\rm Int}(U_n)$. Especially there is an arbitrarily small neighborhood $U$ of $C(\sigma)$ such that $\phi_1(\overline{U})\subset {\rm Int}(U)$. As a consequence, there are $\varepsilon_0>0$, a $C^1$ neighborhood $\cal U$ of $X$ such that for any $Y\in\cal U$ and any $x\in U$, the strong stable manifold $W^{ss}_{\varepsilon_0}(x, Y)$ exists (e.g., see \cite[page 289]{BDV05}).

\begin{Definition} For $Z\in{\cal X}^1(M^3)$, $S$ is called a \emph{cross-section} of $Z$ if
\begin{itemize}

\item $S$ is a $C^1$ surface which is homeomorphic to $(-1,1)^2$.

\item $\measuredangle(T_x S, <Z(x)>)>\pi/4$, $\forall x\in S$.

\end{itemize}
\end{Definition}
\begin{Definition}
Let $\sigma$ be a Lorenz-like singularity. $S$ is called a
\emph{singular cross-section associated to $\sigma$} if $S$ is a cross section and the following conditions are satisfied:
\begin{enumerate}

\item There is a homeomorphism $h=h(x,y):[-1,1]^2\to\overline{S}$ such that
$h((-1,1)^2)=S$.

\item $\displaystyle\frac{\partial h}{\partial y}$ exists and is continuous on $[-1,1]^2$.

\item $S\cap W^s_{loc}(\sigma)=h(\{0\}\times (-1,1))$.

\end{enumerate}

\end{Definition}
In fact, $(h, [-1,1]^2)$ (or $h$ for short) is a coordinate system of the surface $S$.

Denote by $\ell=S\cap W^s_{loc}(\sigma)$ and $S\setminus \ell=S^l\cup S^r$, where $S^l=h((-1,0)\times (-1,1))$ and $S^r=h((0,1)\times
(-1,1))$ (see Figure \ref{cross-section}). $S^l$ is called \emph{the left side} of $S$ and $S^r$
is called \emph{the right side} of $S$.

\begin{figure}[h]
\centering
\includegraphics[width=0.60\textwidth]{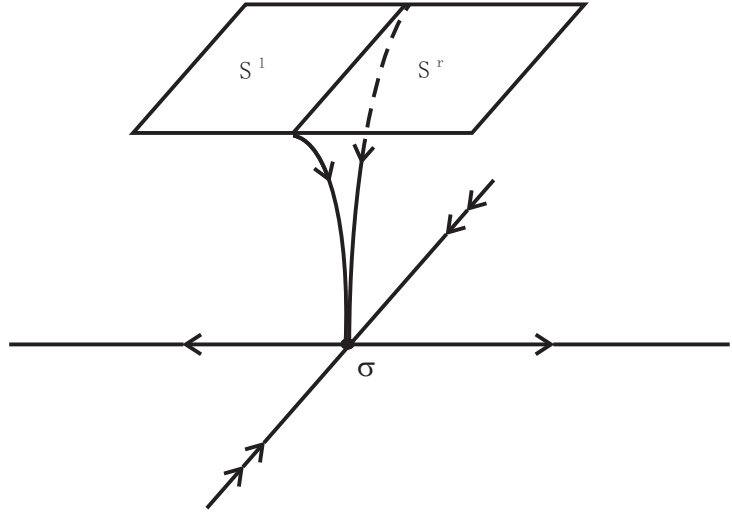}
\caption{Cross-section}\label{cross-section}
\end{figure}

For any $p,q\in S$, one can define the \emph{horizontal distance} $L^h_S(p,q)$ and \emph{vertical distance}
$L^v_S(p,q)$ of $p,q$: $L^h_S(p,q)=|x_p-x_q|$ and $L^v_S(p,q)=|y_p-y_q|$, where
$h(x_p,y_p)=p$ and $h(x_q,y_q)=q$.

Usually, for every Lorenz like singularity $\rho$, we will take two singular cross-sections $S^+, S^-$, which are on the opposite sides of
$E^{ss}(\rho)\oplus E^u(\rho)$.

For $X\in{\cal X}^1(M^3)$, let $\Lambda$ be a pseudo-Lorenz compact invariant set of $X$. Denote that $\Lambda\cap{\rm Sing}(X)=\{\sigma_1,\cdots,\sigma_k\}$ and
$$
\Sigma=\bigcup_{1\le i\le k}(S_i^+\cup S_i^-),
$$
where $S_i^+$ and $S_i^-$ are two singular cross-sections associated to $\sigma_i$. For each $S_i^\pm$, let
$h_i^\pm:~[-1,1]^2\to\overline{S_i^\pm}$ be the homeomorphism in the definition of singular cross-section.
For every $p\in \Sigma$, there exists a unique coordinate system $h_i^\pm$ for some $i\in\{1,2,\cdots,k\}$ and $\pm\in\{+,-\}$. We denote $h_i^\pm$ by $h_p$. Note that for every $p,p'\in\overline{S_i^\pm}$, $h_p=h_{p'}$.

For each $p\in\Sigma$, if $\{t>0:\phi_t(p)\in\Sigma\}\neq\emptyset$, we can define the first return time
$t_p=\min\{t>0:\phi_t(p)\in\Sigma\}$ and define $F(p)=\phi_{t_p}(p)$. $F: {\rm Dom}(F)\to \Sigma$ is called the \emph{first return map}
associated to $\Sigma$ and $\Lambda$, where
$$
{\rm Dom}(F)=\{p\in\Sigma: \exists t>0 {\rm\ s. t. \ }\phi_t(p)\in\Sigma\}.
$$

If $h_p(x,y)=p$ and $F(p)=q$, denote by
$$
{\widetilde F}(x,y)={\widetilde F}_{p,q}(x,y) =h_q^{-1}\circ F\circ h_p(x,y).
$$
Once $F^n(p)=q_n$ is well-defined, denote by
\begin{eqnarray*}
{\widetilde F}^n(x,y)&=&h_{q_n}^{-1}\circ F^n\circ h_p(x,y),\\
&=& {\widetilde F}_{q_{n-1},q_n}\circ{\widetilde F}_{q_{n-2},q_{n-1}}\circ\cdots \circ {\widetilde F}_{p,q_{1}}(x,y),
\end{eqnarray*}
where $q_j=F^j(p)$, for $j=1, 2, \cdots, n$ and $h_p(x,y)=p$.

Since $S_i^\pm$ is transverse to $X$, ${\rm Dom}(F)$ is open and $t_x$ is upper semi-continuous with $x$.

\begin{Definition}

For $X\in{\cal X}^1(M^3)$, and a pseudo-Lorenz compact invariant set $\Lambda$ of $X$, assume that $\Lambda\cap{\rm Sing}(X)=\{\sigma_1,\cdots,\sigma_k\}$. In the notations above, $(\Sigma, F)$ is a \emph{cross-section system} of $\Lambda$ if the following conditions are satisfied:
\begin{enumerate}

\item $\partial \Sigma\cap \Lambda=\emptyset$, where $\partial \Sigma$ is the boundary of $\Sigma$.

\item For each $p\in\Sigma$, there is $\varepsilon>0$ such that $W^{ss}_{loc}(\phi_{(-\varepsilon,\varepsilon)}(p))\cap\Sigma\subset h_p(\{x\}\times (-1,1))$, where $h_p(x,y)=p$.

\item For each $1\le i\le k$, and $x\in W^s_{loc}(\sigma_i)\cap\Lambda\setminus\{\sigma_i\}$, there exists $t\in{\mathbb R}$ such that $\phi_t(x)\in\ell_i^\pm$ and if $t>0$, then
    $\phi_{[0,t]}(x)\subset W^s_{loc}(\sigma_i)$ and if $t<0$,$\phi_{[t,0]}(x)\subset W^s_{loc}(\sigma_i)$.

\item If $p\in {\rm Dom}(F)$, then $t_r$ is continuous for $r\in h_p(\{x_p\}\times (-1,1))$\footnote{This will imply some adapted property: the iteration of each stable leaf under the first return map is totally contained in the cross section.}.

\item For any $x\in\Lambda\setminus \bigcup_{1\le i\le k}W^s_{loc}(\sigma_i)$, the positive orbit of $x$ will intersect $\Sigma$. Especially,
    $$
    \Lambda\cap \Sigma\setminus \bigcup_{1\le i\le k}W^s_{loc}(\sigma_i)\subset {\rm Dom}(F).
    $$

\end{enumerate}

\end{Definition}

According to the uniform contracting of $E^{ss}$, we have
\begin{Lemma}
Let $(\Sigma,F)$ be a cross-section system of a pseudo-Lorenz set $\Lambda$. Then,
there are $C\ge 1$ and $\lambda\in(0,1)$ such that for any $p\in \Sigma$ with $h_p(x,y)=p$, for any $n\in\NN$, if $F^n(p)$ is well-defined, then
$$
\left|\frac{\partial {\widetilde F}^n}{\partial y}(x,y)\right|\le C\lambda^n.
$$\qed
\end{Lemma}

\begin{Remark}
In the above definition, the local stable manifolds of the singularities also have the exponential contracting property with respect to the flow $\phi_t$. Thus,
$$
\bigcup_{1\le i\le
k}\bigcup_{\pm\in\{+,-\}}\bigcup_{x\in(-1,1)}h_i^{\pm}(\{x\}\times(-1,1))
$$
could be regarded as a stable foliation of the first return map.

\end{Remark}

One can define the horizontal distance $L^h$ and the vertical distance $L^v$ of any two points $x,y$ in $\Sigma$
in the following way:
\begin{itemize}

\item If $x,y$ are in a same singular cross-section $S\subset\Sigma$, then $L^h(x,y)=L^h_{S}(x,y)$ and
$L^v(x,y)=L^v_{S}(x,y)$.

\item If $x,y$ are in different singular cross-sections, then $L^h(x,y)=\infty$ and
$L^v(x,y)=\infty$.

\end{itemize}

For any set $\Gamma\subset\Sigma$, one can define the horizontal diameter $D^h(\Gamma)$ and the vertical
diameter $D^v(\Gamma)$ by:
$$D^h(\Gamma)=\sup_{x,y\in\Gamma}\{L^h(x,y)\},~~~D^v(\Gamma)=\sup_{x,y\in\Gamma}\{L^v(x,y)\}.$$

Morales-Pacifico \cite{MoP03} proved the following theorem:

\begin{Theorem}\label{Thm:moralespacifico}

For $C^1$ generic $X\in{\cal X}^1(M^3)$, if $C(\sigma)$ is a singular hyperbolic Lyapunov stable chain recurrent class, then $C(\sigma)$ is an attractor. As a corollary, $C(\sigma)$ contains periodic orbits.

\end{Theorem}

\begin{Lemma}\label{Lem:notsingularhyperbolic}
For $C^1$ generic $X\in{\cal X}^1(M^3)$ and $\sigma\in{\rm Sing}(X)$, if $C(\sigma)$ is pseudo-Lorenz, then
\begin{itemize}
\item either, $C(\sigma)$ is a homoclinic class;

\item or, there is a $C^1$ neighborhood ${\cal U}_0$
of $X$ such that $C(\sigma_Y)$ is not singular hyperbolic for any $Y\in{\cal U}_0$.

\end{itemize}

\end{Lemma}

\begin{proof}
If for any neighborhood $\cal U$ of $X$, there is $Y\in\cal U$ such that $C(\sigma_Y)$ is singular hyperbolic,
then by the upper semi-continuity of chain recurrent classes, there is a $C^1$ neighborhood ${\cal
U}_Y\subset{\cal U}$ of $Y$ such that for every $Z\in{\cal U}_Y$, $C(\sigma_Z)$ is singular hyperbolic. This implies that for
$X$, we have $C(\sigma)$ is already singular hyperbolic.

Thus, Theorem~\ref{Thm:moralespacifico} implies the first case is true. The fact ends the proof.
\end{proof}

For the study of non-hyperbolic set, Liao \cite{Lia81} and Ma\~n\'e \cite{Man85} gave the notion of minimally
non-hyperbolic set: for any non-hyperbolic set $\Lambda$, there is a compact invariant set
$\Lambda_0\subset\Lambda$ such that every proper compact invariant subset of $\Lambda_0$ is hyperbolic. The
proof is based on Zorn's Lemma. To understand the dynamics of a compact invariant set which is not singular
hyperbolic, sometimes we need to use similar idea.

\begin{Definition}
A nonempty compact invariant set $\Lambda$
is called a \emph{$\cal N$-set}, if $\Lambda$ is not singular hyperbolic, and every proper
compact invariant set in $\Lambda$ is singular hyperbolic.
\end{Definition}

As in the case of minimally non-hyperbolic set, by
Zorn's lemma, for any compact invariant set $\Lambda$, if $\Lambda$ is not singular hyperbolic, then $\Lambda$
contains a $\cal N$-set.

\begin{Lemma}\label{Lem:transitiveinclass}

For $C^1$ generic $X\in{\cal X}^1(M^3)$ and $\sigma\in{\rm Sing}(X)$, if $C(\sigma)$ is pseudo-Lorenz, and not a homoclinic class,
then there is a $C^1$ neighborhood ${\cal U}_0$ of $X$ such that for every $Y\in{\cal U}_0$, $C(\sigma_Y)$ contains a
transitive $\cal N$-set $\Lambda_Y$ such that $\Lambda_Y\nsubseteq{\rm Sing}(Y)$.
\end{Lemma}

\begin{proof}
By Lemma~\ref{Lem:notsingularhyperbolic} and
Lemma~\ref{Pro:continuitypseudoLorenz}, there is a neighborhood ${\cal U}_0$ of $X$ such that for any $Y\in{\cal
U}_0$, one has
\begin{itemize}

\item $C(\sigma_Y,Y)$ is not singular hyperbolic.

\item $C(\sigma_Y,Y)$ admits a partially hyperbolic splitting $T_{C(\sigma_Y,Y)}M^3=E^{ss}\oplus E^{cu}$, where $\dim E^{ss}=1$,
and every singularity in $C(\sigma_Y,Y)$ is Lorenz-like.

\end{itemize}

For each $Y\in{\cal U}_0$, $C(\sigma_Y)$ is not singular hyperbolic. Hence, there exists a $\cal N$-set $\Lambda_Y\subset C(\sigma_Y)$.

We will prove that $\Lambda_Y$ is transitive. If not, for every $x\in\Lambda_Y$, $\alpha(x)$ is a proper subset
of $\Lambda_Y$. As a consequence, $\alpha(x)$ is singular hyperbolic for every $x\in\Lambda_Y$. We have already
known that $\Lambda_Y$ admits a partially hyperbolic splitting $T_{\Lambda_Y}M=E^{ss}\oplus E^{cu}$ since
$C(\sigma_Y)$ is pseudo-Lorenz. If $\alpha(x)$ is singular hyperbolic, then $\lim_{t\to\infty}|{\rm Det}
\Phi^Y_{-t}|_{E^{cu}(x)}|=0$. Since every $x\in\Lambda_Y$ has this property, by a compact argument (e.g., see Lemma \ref{Lem:compact}), one can prove that $\Lambda_Y$ is singular hyperbolic. This contradicts to the assumption that $\Lambda_Y$ is an $\cal N$-set.
\end{proof}

\begin{Corollary}\label{Cor:measure}
Under the assumption of Lemma~\ref{Lem:transitiveinclass}, for every ${\cal N}$-set $\Lambda_Y$, there is an ergodic measure $\mu_Y$ of $\phi_t^Y$ such that the support of $\mu_Y$ is $\Lambda_Y$, and for any $t>0$, one has

$$\int \log|{\rm Det}\Phi_t|_{E^{cu}(x)}|{\rm d}\mu_Y\le 0.$$

\end{Corollary}

\begin{proof}
By Lemma~\ref{Lem:compact}, there is a point $x\in\Lambda_Y$ such that $\log|{\rm Det}\Phi_t|_{E^{cu}(x)}|\le 0$ for any $t\ge 0$. By using a standard method, we can have an invariant measure $\nu$ such that for any $t>0$, one has
$$\int \log|{\rm Det}\Phi_t|_{E^{cu}(x)}|{\rm d}\nu\le 0.$$
By using the ergodic decomposition theorem, there is an ergodic component $\mu$ of $\nu$ such that
$$\int \log|{\rm Det}\Phi_t|_{E^{cu}(x)}|{\rm d}\mu\le 0.$$

\noindent ${\rm supp}(\mu)=\Lambda$: otherwise, ${\rm supp}(\mu)$ is a proper compact invariant set of $\Lambda_Y$; hence it is singular hyperbolic, which implies that the inequality above is false.
\end{proof}

We know some structures about minimally non-hyperbolic \emph{non-singular} set for $C^2$ vector fields when the set
admits a dominated splitting w.r.t. the linear Poincar\'e flow, which is the main theorem in \cite{ArR03}.

\begin{Theorem}\label{Thm:pujalssambarino}

Let $X\in{\cal X}^2(M^3)$. If a compact invariant set $\Lambda$ satisfies the following properties:
\begin{itemize}

\item $\Lambda$ is transitive.

\item $\Lambda\cap{\rm Sing}(X)=\emptyset$.

\item The linear Poincar\'e flow $\psi_t$ admits a dominated splitting on ${\cal N}_{\Lambda}$.

\item Every periodic point in $\Lambda$ is hyperbolic.

\end{itemize}

Then,

\begin{itemize}

\item either $\Lambda$ is hyperbolic;

\item or $\Lambda=\TT^2$, $\Lambda$ is a normally hyperbolic set with respect to $\Phi_t$, and $\phi_t|_\Lambda$ is equivalent to an irrational flow.

\end{itemize}

\end{Theorem}

\begin{Corollary}\label{Cor:containsingularity}
Under the assumption of Lemma~\ref{Lem:transitiveinclass}, every ${\cal N}$-set $\Lambda_Y$ contains a singularity for every $C^2$ weak Kupka-Smale vector field $Y\in{\cal U}_0$

\end{Corollary}
\begin{proof}

This is just an application of Theorem~\ref{Thm:pujalssambarino} because $C(\sigma_Y)$ cannot contain a normally hyperbolic torus.
\end{proof}

\begin{Remark}
This is another place that we need to use the assumption of ``weak Kupka-Smale'' besides the usage of the connecting lemma for pseudo orbits.
\end{Remark}

Let us concentrate on the proof of Theorem~\ref{Thm:partialtoperiodic}. We need to fix a neighborhood such that we can do perturbations freely in the neighborhood. Under the assumptions of Theorem~\ref{Thm:partialtoperiodic}, if the conclusion is not true, then there is a $C^1$ neighborhood ${\cal U}^*$ such that

\begin{itemize}
\item For every $Y\in{\cal U}^*$, $C(\sigma_Y)$ is not singular hyperbolic.

\item Since chain recurrent classes is upper semi-continuous by
Lemma~\ref{Lem:continuitychainclass}, one has for any $Y\in {\cal U}^*$, $C(\sigma_Y)$ has a partially hyperbolic splitting $T_{C(\sigma_Y)} M^3=E^{ss}\oplus E^{cu}$ with respect to the tangent flow.

\item By the continuity of eigenvalues of singularity and Lemma~\ref{Lem:twochainrecurrentclass}, every singularity in $C(\sigma_Y)$ is a continuation of a singularity in $C(\sigma)$, and it is Lorenz-like.

\item $C(\sigma_Y)$ is still Lyapunov stable by Lemma~\ref{Lem:stablepropertyofquasiattractor} if $Y$ is weak Kupka-Smale.

\item $C(\sigma_Y)$ is not singular hyperbolic by Lemma~\ref{Lem:notsingularhyperbolic}. As in Lemma~\ref{Lem:transitiveinclass}, $C(\sigma_Y)$ contains a transitive $\cal N$-set inside, which is not reduced to a singularity.
\end{itemize}

\begin{Proposition}\label{Pro:crosssectionforpseudoLorenz}
For $C^1$ generic vector field  $X\in{\cal U}^*$, and $\sigma\in{\rm Sing}(X)$, if $C(\sigma)$ is pseudo-Lorenz and contains no periodic orbit, then $C(\sigma)$ admits a cross-section system $(\Sigma,F)$.
\end{Proposition}

\begin{proof}

{\bf The preparation: curves in the stable manifolds.\ } Assume that
$$
C(\sigma)\cap{\rm Sing}(X)=\{\sigma_1, \sigma_2, \cdots, \sigma_k\}.
$$
For each singularity $\sigma_i$, one can choose a local
chart $\phi^s_i:\RR^2\to W^s_{loc}(\sigma_i)$ such that
\begin{itemize}

\item $\phi^s_i(\{0\}\times\RR)=W^{ss}_{loc}(\sigma_i)$,

\item $\phi^s_i(\RR\times\{0\})$ is an invariant central manifold.

\end{itemize}
Take two curves $\gamma_{i,1},\gamma_{i,2}:\RR\to W^s(\sigma_i)$ as the images of $y=x$ and $y=-x$ under $\phi^s_i$ such that $\gamma_{i,j}(0)=\sigma_i$ for $j=1,2$.

\begin{Claim}
There is $\rho_i>0$ such that $C(\sigma)\cap \phi^s_i([-\rho_i,\rho_i]^2)\subset \phi_i^s(\{(x,y):|x|<
|y|\})$. As a corollary, $\gamma_{i,1}([-\rho_i,\rho_i])\cap C(\sigma)=\{\sigma_i\}$ and
$\gamma_{i,2}([-\rho_i,\rho_i])\cap C(\sigma)=\{\sigma_i\}$.

\end{Claim}

\begin{proof}[Proof of the claim]
If the claim is not true, there are $x_n\in\phi_i^s(\{(x,y):|x|\ge |y|\})\cap C(\sigma)$ such that
$\lim_{n\to\infty}x_n=\sigma_i$ and $x_n\neq\sigma_i$. The negative iterations of $x_n$ are still in
$C(\sigma)$. Choose a small neighborhood $B_i$ of $\sigma_i$. Let
$$t_n=\sup\{t:~\phi_{-s}(x_n)\in B_i,~\forall 0\le s\le t\}.$$
We have that $t_n\to\infty$ as $x_n\to\sigma_i$. Let $a$ be an accumulation point of $\phi_{-t_n}(x_n)$. Then $\phi_t(a)\in \{(x,y):|x|\ge |y|\}$ for $t\ge 0$. Hence $a\in C(\sigma)\cap W^{ss}(\sigma_i)\cap \partial B_i$. This contradicts to the fact that $W^{ss}(\sigma_i)\cap
C(\sigma)=\{\sigma_i\}$.
\end{proof}

{\bf The first step of the construction:} Let
$\rho=\min\{\rho_i: 1\le i\le k\}$. For each $\sigma_i$, there are two connected components $\Theta_i^{\pm}$ of
$W^s_{loc}(\sigma_i)\setminus W^{ss}_{loc}(\sigma_i)$. In the following, we will use $\Theta_i^+$ to construct $S_i^+$, while the other case can be constructed similarly.

There are two
points $x_{i,1}\in\gamma_{i,1}\cap\Theta_i^+$ and $x_{i,2}\in\gamma_{i,2}\cap\Theta_i^+$ such that
\begin{itemize}
\item $x_{i,1},x_{i,2}\notin C(\sigma)$.

\item $x_{i,1}\in W^{ss}_{loc}(x_{i,2})$.

\end{itemize}

Thus there is a cross-section $\widetilde{S}_i^+=h_i^+((-1,1)^2)$, where $h_i^+:[-1,1]^2\to\overline{\widetilde{S}_i^+}$ is a homeomorphism, such that

\begin{itemize}

\item $h_i^+((0,-1))=x_{i,1}$, $h_i^+((0,1))=x_{i,2}$ and $h_i^+(\{0\}\times(-1,1))$ is a connected part of
a strong stable manifold of $\phi_t$.

\item $h_i^+((-1,1)\times\{-1,1\})\cap C(\sigma)=\emptyset$.

\item $\widetilde{S}_i^+$ is foliated by strong stable foliation in the following sense: for each $x\in \widetilde{S}_i^{+}$, one define ${\cal F}^s(x)$ to be the connected component of
$\bigcup_{t\ge0}\phi_t(W^{ss}_{\varepsilon_1}(x))\cap \widetilde{S}_i^{+}$, $\widetilde{S}_{i}^{+}$ can be foliated by
${\cal F}^s$. Moreover, $h_i^+(\{z\}\times(-1,1))$ is a leaf of the strong stable foliation.

\item For any arbitrarily small number $\alpha>0$, one can require that $D^h(\widetilde{S}_i^+)<\alpha$.

\item $\cup_{x\in(-1,1)}h_i^+(\{x\}\times(-1,1))$ is a family of $C^1$ curves, and as a $C^1$ family, it varies continuously with respect to $x$.
\end{itemize}

One can construct $\widetilde{S}_i^-$ for $\Theta_i^-$ also.
%
%
%

\paragraph{Refine the construction:}We take a neighborhood $U=B_\delta(C(\sigma))$ with $\delta$ small enough such that $B_\delta(\sigma_i)$ and $B_\delta(\sigma_j)$ are disjoint for $i\not=j$ and $\overline{U}$ is disjoint from $h_i^\pm((-\beta_0,\beta_0)\times\{-1,1\})$ for any $i$ and any
$\pm\in\{+,-\}$. Denote $S_i^\pm(\beta)=h_i^\pm((-\beta,\beta)\times(-1,1))$ for $\beta\in(0,\beta_0]$.
Denote by
$$
\Sigma(\beta)=\bigcup_{1\le i\le k,\ \ \pm\in\{+,-\}}S_i^{\pm}(\beta).
$$
As before, consider the first return map $F$ with respect to $\Sigma(\beta)$.

\begin{Claim} If $\beta$ is small enough, for any $x\in {\rm Dom}(F)$, $t_x=\inf\{t>0: \phi_t(x)\in\Sigma(\beta)\}$ is continuous on ${\cal F}^s(x)$, the strong stable leaf of $x$ with respect to $F$.
\end{Claim}

\begin{figure}[h]
\centering
\includegraphics[width=0.60\textwidth]{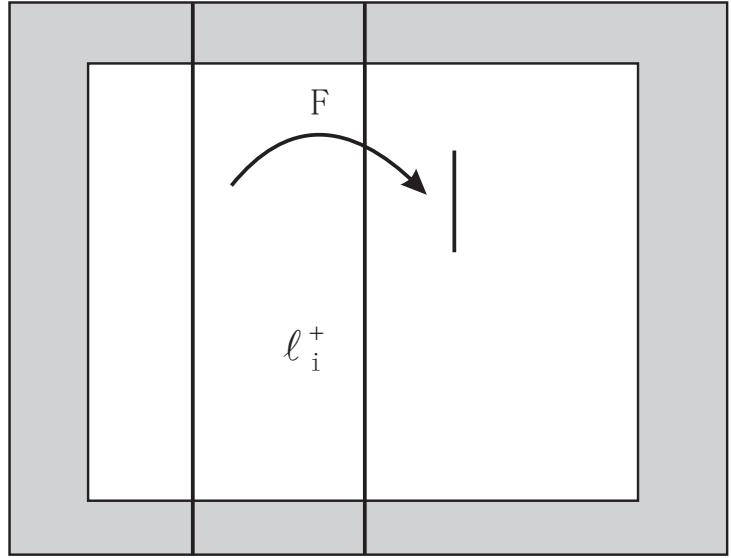}
\caption{Cross-section and return map: $\Lambda$ does not intersect the shaded area and the image of any strong stable leaf under the return map does not intersect the shaded area.}
\end{figure}

\begin{proof}[Proof of the claim]
Otherwise, $\phi_t({\cal F}^s(x))$ will intersect some $h_i^\pm((-\beta_0,\beta_0)\times\{-1,1\})$ for some $t>0$. Since $C(\sigma)$ is Lyapunov stable, for $\epsilon>0$ small enough, for any $t\ge 0$, $\phi_t(B_\epsilon(\sigma_i))\subset U$. But for $\beta$ small enough, for any $x\in {\rm Dom}(F)$, there exists $T>0$ such that $\phi_T(x)\in B_\epsilon(\sigma_i)$ for some $i$ before it first returns to $\Sigma(\beta)$ since $S_i^{\pm}(\beta)$ is very close to the stable manifold of $\sigma_i$. So for $t>T$, $\phi_t(x)\in U$. Especially, $\phi_t(x)\cap h_i^\pm((-\beta_0,\beta_0)\times\{-1,1\})=\emptyset$ for any $i$ and any $\pm\in\{+,-\}$. This contradiction proves the claim.
\end{proof}
Since $X$ is a $C^1$ generic vector field, $C(\sigma)$ can be accumulated by periodic orbits. Since $X\in{\cal U}_0$, $C(\sigma)$ contains no hyperbolic periodic orbit.

\begin{Claim}
Given $\beta\in(0,\beta_0]$, for every $i=1,2,\cdots,k$ and $\pm\in\{+,-\}$,
$$J_i^\pm=\{z\in (-\beta,\beta): h_i^\pm(\{z\}\times(-1,1))\cap C(\sigma)=\emptyset\}$$ is open and dense
in $(-\beta,\beta)$.
\end{Claim}

\begin{proof}[Proof of the claim]

We only have to prove that: Given $a,b\in (-\beta,\beta), a<b$, there exists $z\in (a,b)$ such that
$h_i^\pm(\{z\}\times(-1,1))\cap C(\sigma)=\emptyset$. Otherwise, since $C(\sigma)$ can be accumulated by
periodic orbits, there is a periodic point $p\in S_i^\pm$ close to $h_i^\pm(\{(a+b)/2\}\times(-1,1))$. Thus
$W^{ss}_{\varepsilon_0}(p)\cap C(\sigma)\neq\emptyset$. This contradicts to the fact that $C(\sigma)\cap {\rm Per}(X)=\emptyset$.
\end{proof}

For $\beta>0$ small enough,
$$
G=\bigcap_{1\le i\le k, \pm\in\{+,-\}}(0,\beta)\cap J_i^\pm \cap (-J_i^\pm)
$$
is open and dense in $(0,\beta)$.

So, take $\beta'\in G$ and let $\Sigma=\Sigma(\beta')$. After scaling, we may assume coordinate mappings $h_i^\pm$ are defined on $(-1,1)^2$. Then $(\Sigma,F)$ satisfies item 1)-4) in the definition of cross-section system.

\paragraph{The domain of $F$:} Since we assume that $C(\sigma)$ contains no periodic orbit, we have that $\omega(x)$ contains a singularity for any $x\in C(\sigma)$. This implies that for every $x\in C(\sigma)$, if $x$ is not in the local stable manifold of some singularity, then the positive iterations of $x$ will intersect $\Sigma$. This finishes the proof of the existence of cross-section system.
\end{proof}

By summarizing the construction as in the above proof, we first find some ``large'' cross-section $\Sigma(1)$, then we just take some smaller part $\Sigma(\beta')$ which is modified by the strong stable foliation. Since the local strong stable manifolds are continuous with respect to the vector fields, for any $Y$ $C^1$ close to $X$, the intersection of the strong stable manifolds $W^{ss}_{loc}(z,Y)$ of $z\in \Sigma(\beta')$ and $\Sigma(1)$ is close to $\Sigma(\beta')$. Thus, the cross-section system has some continuous property from the proof of the above proposition.
\begin{Corollary}\label{Cor:crosssectionforpseudoLorenz}
For $C^1$ generic vector field $X\in{\cal U}^*$, let $C(\sigma)$ be a chain recurrent class satisfying the
assumption of Theorem~\ref{Thm:partialtoperiodic}. Then there is a $C^1$ neighborhood ${\cal U}\subset{\cal U}^*$
of $X$, such that for any $Y\in\cal U$, $C(\sigma_Y)$ admits a cross-section system $(\Sigma_Y,F_Y)$.
Moreover, when $Y\to X$, we have $\Sigma_Y\to\Sigma$.
%
\end{Corollary}

\begin{proof}
As explained above, we take

$$\Sigma_Y=(\bigcup_{z\in\Sigma_X}\phi_{[-\varepsilon(z),\varepsilon(z)]}(W^{ss}_{loc}(z,Y)))\cap \Sigma(1).$$

Moreover, one can defined the first return map $F_Y$ by using this cross section.

By the continuity of the local strong stable manifolds w.r.t. the vector fields, we have that $\Sigma_Y$ is close to $\Sigma=\Sigma_X$ when $Y$ is close to $X$. Since $C(\sigma_Y)$ is continuous w.r.t. $Y$, we have that $C(\sigma_Y)\cap \partial\Sigma_Y=\emptyset.$ By the definition of $\Sigma_Y$, Item 2 of the cross section system is satisfied. Item 3 is true because that $C(\sigma_Y)\cap W^s_{loc}(\sigma_Y)$ is continuous w.r.t. $Y$. Item 4 is true if the return time is long, which can be guaranteed if the cross section is thin. We have Item 5  because

\begin{itemize}
\item The return time is uniformly continuous when the point is not close to the local stable manifolds of the singularities.

\item When the point is close to the local stable manifolds of the singularities, the return will follow the unstable manifolds of the singularities, which are stably contained in the new cross-section.
\end{itemize}
\end{proof}

Now we take a $C^1$ neighborhood ${\cal U}^{**}\subset{\cal U}^*$ of $X$
such that ${\cal U}^{**}$ verifies Corollary~\ref{Cor:crosssectionforpseudoLorenz}.

\subsection{Homoclinic orbits of singularities}

Now we are going to prove Proposition~\ref{Pro:roughsinkbasin}: for generic $X\in{\cal X}^1(M^3)$, if $C(\sigma)$ is pseudo-Lorenz, then either $C(\sigma)$ is a singular hyperbolic attractor; or for any $C^1$ neighborhood $\cal U$ of $X$, there is a weak Kupka-Smale vector field $Y\in\cal U$ such that $Y$ has a periodic sink $\gamma$ and $\overline{{\rm Basin}(\gamma)}\cap C(\sigma_Y)\neq\emptyset$. As the end of Section~\ref{Sec:reduction}, once we know Proposition~\ref{Pro:roughsinkbasin} is true, we will get the final result.

For each Lorenz-like singularity $\rho$ and a small neighborhood $U$ of this singularity, the local stable manifold of $\rho$ divides $U$ into two parts: the left part and the right part. Thus the unstable manifold $W^u(\rho)$ has two separatrice: the left separatrix and the right separatrix.

For each $Y\in{\cal U}^{**}$, one defines $n(Y)$ to be the number of homoclinic orbits of singularities contained in $C(\sigma_Y)$. Since there are only $k$ singularities in $C(\sigma)$, $n(Y)\le 2k$. Let
$$n=\max\{n(Y):~Y\in{\cal U}_1,~Y~{\rm is}~C^2~{\rm and}~{\rm~is~weak~Kupka-Smale}\}.$$
Let ${\cal M}_n\subset{\cal U}^{**}$ be the set of $C^2$ weak Kupka-Smale vector fields $Y$ with $n(Y)=n$.

\begin{Lemma}\label{Lem:homoclinicintransitive}
For $Y\in{\cal M}_n$ and for the transitive ${\cal N}$-set $\Lambda_Y\subset C(\sigma_Y)$ as in
Lemma~\ref{Lem:transitiveinclass}, if $\sigma_{i,Y}\in\Lambda_Y$, then $\sigma_{i,Y}$ has a homoclinic orbit
$\Gamma_i^\pm$.

\end{Lemma}
\begin{proof}
Suppose on the contrary, $\Lambda_Y$ contains a $\sigma_{i,Y}$, but $C(\sigma_Y)$
does not contain a homoclinic orbit of $\sigma_{i,Y}$. Since $\Lambda_Y$ is transitive, $\Lambda_Y\cap
W^s(\sigma_{i,Y})\setminus\{\sigma_{i,Y}\}\neq\emptyset$ and $\Lambda_Y\cap
W^u(\sigma_{i,Y})\setminus\{\sigma_{i,Y}\}\neq\emptyset$. Take $x^s\in \Lambda_Y\cap
W^s(\sigma_{i,Y})\setminus\{\sigma_{i,Y}\}$ and $x^u\in \Lambda_Y\cap
W^u(\sigma_{i,Y})\setminus\{\sigma_{i,Y}\}$. By the assumptions, $x^s$ and $x^u$ are not in a homoclinic orbit.

\paragraph{Construction of perturbation boxes.} For $Y\in{\cal M}_n$, one can choose $\varepsilon>0$,
such that $B(Y,\varepsilon)\subset {\cal U}^{**}$. By Lemma~\ref{Lem:wenxia}, one can choose $L>0$ and
neighborhoods $\widetilde{W}_{x^s}\subset W_{x^s}$ of $x^s$, $\widetilde{W}_{x^u}\subset W_{x^u}$ of $x^u$ as in
Lemma~\ref{Lem:wenxia} such that
\begin{itemize}
\item $W_{L,x^s}\cap W_{L,x^u}=\emptyset$.

\item $\Lambda_Y\setminus (W_{L,x^s}\cup W_{L,x^u})\neq\emptyset$.

\item $W_{L,x^s}\cup W_{L,x^u}$ is disjoint from any other homoclinic orbits of singularities.

\item $W_{L,x^s}\cup W_{L,x^u}$ is disjoint from $\sigma_{i,Y}$.

\end{itemize}

\paragraph{Choosing the orbits.} Since $\Lambda_Y$ is transitive, there is $z\in \Lambda_Y\setminus (W_{L,x^s}\cup W_{L,x^u})$ such that $\alpha(z)=\omega(z)=\Lambda$. Choose $t_1,t_2>0$ such that $\phi_{-t_1}^Y(z)\in \widetilde{W}_{x^s}$ and $\phi_{t_2}^Y(z)\in \widetilde{W}_{x^u}$. Choose $t_s>0$ and $t_u>0$ such that $\phi^Y_{t_s}(x^s)\notin W_{L,x^s}\cup W_{L,x^u}$ and $\phi^Y_{-t_u}(x^u)\notin W_{L,x^s}\cup W_{L,x^u}$.

\paragraph{Connecting the orbit from $\phi^Y_{-t_2}(z)$ to $\phi^Y_{t_s}(x^s)$.} Since the negative orbit of $\phi^Y_{t_s}(x^s)$ and the positive orbit of $\phi^Y_{-t_2}(z)$ both enter $\widetilde{W}_{x^s}$, by using Lemma~\ref{Lem:wenxia}, there is $Y_1$ which is $\varepsilon$-close to $Y$ such that
\begin{itemize}
\item There is $T_1>0$ such that $\phi^{Y_1}_{-T_1}(\phi^Y_{t_s}(x^s))=\phi^Y_{-t_2}(z)$.

\item $Y_1(x)=Y(x)$ for any $x\in M^3\setminus W_{L,x^s}$.

\end{itemize}

As a corollary, we have
\begin{itemize}

\item Any homoclinic orbit of $Y$ is still a homoclinic orbit of $Y_1$.

\item $\sigma_{i,Y}$ is still a singularity of $Y_1$, $\phi^Y_{t_s}(x^s)$ is still in the stable manifold of
$\sigma_{i,Y}$ and $\phi^Y_{-t_u}(x^u)$ is still in the unstable manifold of $\sigma_{i,Y}$ with respect to $Y_1$.
\end{itemize}

\paragraph{Connecting the orbit from $\phi^Y_{-t_u}(x^u)$ to $\phi^Y_{t_s}(x^s)$.} Since
$\phi^{Y_1}_{-T_1}(\phi^Y_{t_s}(x^s))=\phi^Y_{-t_2}(z)$ is contained in $\widetilde{W}_{x^u}$, by using
Lemma~\ref{Lem:wenxia} again, there is $Y_2$  which is $\varepsilon$-close to $Y_1$ such that
\begin{itemize}

\item There is $T_2>0$ such that $\phi^{Y_1}_{-T_2}(\phi^Y_{t_s}(x^s))=\phi^Y_{-t_u}(x^u)$.

\item $Y_2(x)=Y_1(x)$ for any $x\in M^3\setminus W_{L,x^u}$.

\end{itemize}

As a corollary, we have
\begin{itemize}
\item $Y(x)=Y_2(x)$ for any $x\in M\setminus (W_{L,x^s}\cup W_{L,x^u})$.

\item Any homoclinic orbit of $Y$ is still a homoclinic orbit of $Y_1$.

\item $\phi^Y_{t_s}(x^s)$ is in a homoclinic orbit of $\sigma_{i,Y}$.

\end{itemize}

Since $Y_2\in{\cal U}_1$ and $Y\in{\cal M}_n$, we have that $Y_2$ has $n+1$ homoclinic orbits of singularities in $C(\sigma_Y)$, which contradicts to the maximality of $n$.
\end{proof}

\begin{Definition}
A point $p\in M$ is called a \emph{typical point} of a probability ergodic measure $\mu$ of a vector field $Y$, if the following conditions are satisfied:
\begin{enumerate}
\item $p$ is strongly closable;
\item $\omega(p)={\rm supp}(\mu)$;
\item for every continuous function $f: M\to \mathbb{R}$,
$$
\lim_{T\to+\infty}\frac 1T\int_0^Tf(\phi_t^Y(p)){\rm d} t=\int f(x){\rm d}\mu(x).
$$
\end{enumerate}
\end{Definition}

According to Ergodic Closing Lemma and Birkhoff Ergodic Theorem, the set of typical points of $\mu$ has $\mu$-full measure.

\begin{Proposition}\label{Pro:generalresult}

For any $Y\in{\cal M}_n$,and $\sigma_{i,Y}\in \Lambda_Y\cap {\rm Sing}(Y)$, assume that

\begin{enumerate}

\item The right separatrix $\Gamma$ of $\sigma_{i,Y}$ is a homoclinic orbit. Let $x\in W^u_{loc}(\sigma_{i,Y})\cap \Gamma$ and denote
$$
{\cal T}_{x}=\{t>0:\phi^Y_{t}(x)\in\Sigma\}=\{t_1<t_2<\cdots<t_{N(x)}\},$$
where $N(x)=\#{\cal T}_x$. We may assume that
$\phi^Y_{t_{N(x)}}(x)\in S_i^+$.

\item There are\footnote{Notice that here we use $N$ but not $N(x)$, because we need to deal with some general case.} $N\in\NN$ and $\beta\in(0,1)$ such that
$$F^{N}(h_i^+((0,\beta)\times(-1,1)))\subset S_i^+.$$

\item Let $\mu_Y$ be the ergodic measure as in Corollary~\ref{Cor:measure}. If $(\Lambda_Y,\mu_Y)$ accumulates $\phi^Y_{t_{N(x)}}(x)$ on the right: there is a sequence of typical points
$x_n\in h_i^+((0,1)\times(-1,1))$ of $\mu_Y$ such that $\lim_{n\to\infty}x_n=\phi^Y_{t_{N(x)}}(x).$

\end{enumerate}

Then for any $C^1$ neighborhood ${\cal V}$ of $Y$, there is a weak Kupka-Smale vector field $Z\in\cal V$ such that $Z$ has a periodic sink, whose basin accumulates on $C(\sigma_Z)$. In other words, Proposition~\ref{Pro:roughsinkbasin} is true in this case.

\end{Proposition}

\begin{proof}

Since every singularity in $C(\sigma_{Y})$ is Lorenz-like, for $\beta_0$ small enough, we have
$F^{N(x)}(h_i^+(\{\beta_0\})\times(-1,1)))\subset h_i^+(\{\beta_1\}\times(-1,1)))$ for some $\beta_1\in(\beta_0,1)$.

\begin{Claim}

There is a weak Kupka-Smale vector field $Z$ arbitrarily $C^1$-close to $Y$ such that $Z$ has a periodic sink $\gamma$ with the following properties:
\begin{itemize}

\item $\gamma$ intersects $S_i^{+,r}$, and one can choose the intersection points arbitrarily close to $\phi^Y_{t_{N(x)}}(x)$.

\item Every homoclinic orbit of singularities of $Y$ is still a homoclinic orbit of singularities of $Z$.

\end{itemize}

\end{Claim}

\begin{proof}[Proof of the claim]

Given $T>0$, define $f_T(x)=\log|{\rm Det}\Phi_T^Y|_{E^{cu}(x)}|$ for any $x\in C(\sigma_Y)$. One knows that $f_T$ is a continuous function on $C(\sigma)$. Since $E^{cu}$ can be extended continuously in a small neighborhood of $C(\sigma)$, $f_T$ can be also extended continuously. Denote by $\widehat{E}^{cu}$ and $\widehat{f}_T$ the extension of $E^{cu}$ and $f_T$ respectively. Note that we don't require that $\widehat{E}^{cu}$ is invariant. By the property of $\mu_Y$, one has $\int f_T{\rm d}\mu_Y\le 0$

Since $\mu_Y$ is ergodic and ${\rm supp}(\mu_Y)=\Lambda$, one has that the set of homoclinic orbits of singularities has zero measure w.r.t. $\mu_Y$. Choose a typical point $x$ of $\mu_Y$, which is very close to $\phi^Y_{t_{N(x)}}(x)$ and which is in $S_i^{+,r}$. Since $x$ is typical, one can assume that $x$ is a strong closable point by Lemma~\ref{Lem:ergodicclosing}. Now by Corollary~\ref{Cor:subsetisrobust}, for any $\varepsilon>0$, there is $Z$ $\varepsilon$-close to $Y$ such that
\begin{itemize}

\item Every homoclinic orbit of singularities of $Y$ is still a homoclinic orbit of singularities of $Z$.

\item $Z$ has a periodic $\gamma$ such that
$$\left|\int \widehat{f}_T {\rm d}\delta_{\gamma}-\int \widehat{f}_T{\rm d}\mu_Y\right|<\varepsilon.$$

\end{itemize}

Since $\gamma$ is close to $\Lambda_Y$, one knows that $\gamma$ admits a partially hyperbolic splitting $T_{\gamma} M^3=E^{ss,Z}\oplus E^{cu,Z}$. By the property of dominated splittings, for any $x\in\gamma$, one has that $\widetilde{d}(\widehat{E}^{cu}(x),E^{cu,Z}(x))={\rm O}(\varepsilon)$, where ${\rm O}(\varepsilon)$ means that there is a constant $C>0$ such that ${\rm O}(\varepsilon)/\varepsilon\le C$ for $\varepsilon$ small enough.

Thus, one has that $|\int \log|{\rm Det}\Phi_T^Z|{\rm d}\delta_{\gamma}-\int f_T {\rm d}\mu_Y|={\rm O}(\varepsilon)$. As a consequence,
$$\int \log|{\rm Det}\Phi_T^Z|{\rm d}\delta_{\gamma}\le {\rm O}(\varepsilon).$$
Thus, by using the Franks Lemma (Lemma~\ref{Lem:Franks}), by an extra small perturbation in an arbitrarily small neighborhood of $\gamma$, one gets that $\gamma$ is a periodic sink. Finally, by Theorem~\ref{Thm:xiruibin}, one can assume that $Z$ is weak Kupka-Smale.
\end{proof}

Since $Z$ is close enough to $Y$, $Z$ has a cross-section system close to $\Sigma$ by Corollary~\ref{Cor:crosssectionforpseudoLorenz}. For simplicity, we still denote by $\Sigma$ the cross-section system of $Z$.

\begin{Claim}
There is $\beta_Z>0$ such that for any $\beta\in (0,\beta_Z]$, there is $N(\beta)\in\NN$ such that ${
F}_Z^{N(\beta)}({h}_i^+((0,\beta)\times(-1,1)))\supset {h}_i^+((0,\beta_Z]\times(-1,1))$.
\end{Claim}

\begin{proof}

This is true because the singularities is Lorenz-like: for the return map near the local stable manifold of the singularities, the horizontal direction will be expanded. Moreover, the local stable manifold of the singularities served as ``fixed points'' of the return map.
\end{proof}

We need to consider the relation between the intersections of $\gamma$ and $\Gamma$ with $S_j^{\pm}$.

For each $j$ and $\pm\in\{+,-\}$, denote
$$
{\cal T}_{j,\gamma}^\pm=\{t_\gamma:~\gamma\cap h_j^\pm(\{t_\gamma\}\times[-1,1])\neq\emptyset\}.
$$
$$
{\cal T}_{j,\Gamma}^\pm=\{t_\Gamma:~\Gamma\cap h_j^\pm(\{t_\Gamma\}\times[-1,1])\neq\emptyset\}.
$$

$$
a_j^{\pm,l}=\max{{\cal T}}_{j,\gamma}^\pm\cap[-1,0),~~~a_j^{\pm,r}=\min{{\cal T}}_{j,\gamma}^\pm\cap(0,1].
$$
$$
b_j^{\pm,l}=\max{{\cal T}}_{j,\Gamma}^\pm\cap[-1,0),~~~b_j^{\pm,r}=\min{{\cal T}}_{j,\Gamma}^\pm\cap(0,1].
$$
${\cal T}_{j,\Gamma}^\pm$ and ${\cal T}_{j,\Gamma}^\pm$ may be empty for some $j,\pm$. Since $\gamma$ can be chosen arbitrarily close to $\Gamma$ and obtained by a typical orbit, we may assume that if ${\cal T}_{j,\Gamma}^\pm$ is nonempty, then ${\cal T}_{j,\gamma}^\pm$ is also nonempty. Moreover, consider the iterates of the right direction $R$ attached to $\Gamma$ in $S_i^+$. We may assume that
\begin{enumerate}
\item If $\Gamma$ intersects $S_j^{\pm,l}$, and the corresponding iterate of $R$ is still the right direction in $S_j^{\pm}$, then $b_j^{\pm,l}<a_j^{\pm,l}$.
\item If $\Gamma$ intersects $S_j^{\pm,r}$, and the corresponding iterate of $R$ is the left direction in $S_j^{\pm}$, then $b_j^{\pm,r}>a_j^{\pm,r}$.
\end{enumerate}

Since $Z$ is weak Kupka-Smale, $C(\sigma_Z)$ is Lyapunov stable. This implies that for given $r\in\mathbb{N}$, ${F}_Z^r({h}_i^+((0,\beta)\times(-1,1)))$ is well defined for $\beta$ small enough since the boundary of $\Sigma_Z$ dose not intersect $C(\sigma_Z)$.

We will deal with the following two cases:
\begin{enumerate}
\item For any $\beta\in(0,\beta_Z]$, there exists some $r\ge N(\beta)$, such that ${F}_Z^r({h}_i^+((0,\beta_Z)\times(-1,1)))$ intersects $l_j^\pm$ for some $j, \pm$.
\item For any $r>0$, ${F}_Z^r({h}_i^+((0,\beta_Z)\times(-1,1)))$ will never intersect $l_j^\pm$ for any $j, \pm$.
\end{enumerate}

{\bf Case 1.} We may assume that $r$ is the smallest positive integer such that ${F}_Z^r({h}_i^+((0,\beta_Z)\times(-1,1)))$ intersects $l_j^\pm$ for some $j, \pm$. Without loss of generality, we may assume that
$$
\lim_{t\to 0+}{F}_Z^r({h}_i^+(\{t\}\times(-1,1))) \in S_j^{\pm,l}.
$$
Since $C(\sigma_Z)$ is Lyapunov stable, the corresponding iterate $F^r_Z(R)$ of $R$ is still the right direction in $S_j^{\pm}$. And hence
${F}_Z^r({h}_i^+((0,\beta)\times(-1,1)))$ intersects ${h}_j^\pm(\{a_J^{\pm,l}\}\times(-1,1))$, which are contained in the basin of the sink $\gamma$. The conclusion follows by letting $\beta\to 0$.

{\bf Case 2.} For any $r>0$, ${F}_Z^r({h}_i^+((0,\beta_Z]\times(-1,1)))$ will never intersect $l_j^\pm$ for any $j, \pm$.

This will imply that  for every $r>0$, ${F}_Z^r({h}_i^+((0,\beta_Z]\times(-1,1)))$  is connected.

For any $\beta_0\in (0,\beta_Z]$, there is an increasing
sequence $\{\beta_r\}_{r\in\NN}\subset(0,1)$ such that ${
F}^r_Z({h}_i^+(\{\beta_0\}\times(-1,1)))\subset {h}_i^+(\{\beta_r\}\times(-1,1))$ for any
$r\in\NN$. We assume that $\beta'=\lim_{r\to\infty}\beta_r$. By the continuity of ${F_Z}$ and
${h}_i^+$, one has
${F}_Z({h}_i^+(\{\beta'\}\times(-1,1)))\subset {h}_i^+(\{\beta'\}\times(-1,1))$.
 Since $h_i^+(\{\beta'\}\times(-1,1))$ is a stable
leaf of $\widetilde{F}$, one has $h_i^+(\{\beta'\}\times(-1,1))$ contains a periodic point of $Z$, which is a hyperbolic sink of $Z$. This implies that
${h}_i^+((0,\beta_Z]\times(-1,1))$ is in a basin of the sink.
\end{proof}

\begin{Remark}

In the above proof, since we are in the non-generic case, we need to do one extra perturbation to get a sink in a suitable position. This is one of the main difficulties in this paper.

\end{Remark}

A direct application of Proposition~\ref{Pro:generalresult} would be the following result:

\begin{Proposition}\label{Pro:righttoleft}
For any $Y\in{\cal M}_n$,and $\sigma_{i,Y}\in \Lambda_Y\cap {\rm Sing}(Y)$, assume that the right separatrix of $W^u(\sigma_{i,Y})$ is a homoclinic orbit $\Gamma_i$ of $\sigma_{i,Y}$. Let $x\in W^u_{loc}(\sigma_{i,Y})\cap \Gamma_i$ and denote by
$$
{\cal T}_{x}=\{t>0:\phi^Y_{t}(x)\in\Sigma\}=\{t_1<t_2<\cdots<t_{N(x)}\},$$
where $N(x)=\#{\cal T}_x$. We may assume that
$\phi^Y_{t_{N(x)}}(x)\in S_i^+$.

Let $\mu_Y$ be the ergodic measure as in Corollary~\ref{Cor:measure}. If $(\Lambda_Y,\mu_Y)$ accumulates $\phi^Y_{t_{N(x)}}(x)$ on the right: there is a sequence of typical points
$x_n\in h_i^+((0,1)\times(-1,1))$ of $\mu_Y$ such that $\lim_{n\to\infty}x_n=\phi^Y_{t_{N(x)}}(x)$, then

\begin{itemize}

\item either, for $t>0$ small enough we have that
$$
F^{N(x)}(h_i^+((0,t)\times(-1,1)))\subset h_i^+((-1,0)\times(-1,1)).
$$
\item or, for any neighborhood $\cal V$ of $Y$, there is a weak Kupka-Smale vector field $Z\in\cal V$ such that $Z$ has a periodic sink $\gamma$ and $\overline{{\rm Basin}(\gamma)}\cap C(\sigma_Z)\neq\emptyset$.

\end{itemize}
\end{Proposition}
\begin{figure}[h]
\centering
\includegraphics[width=0.60\textwidth]{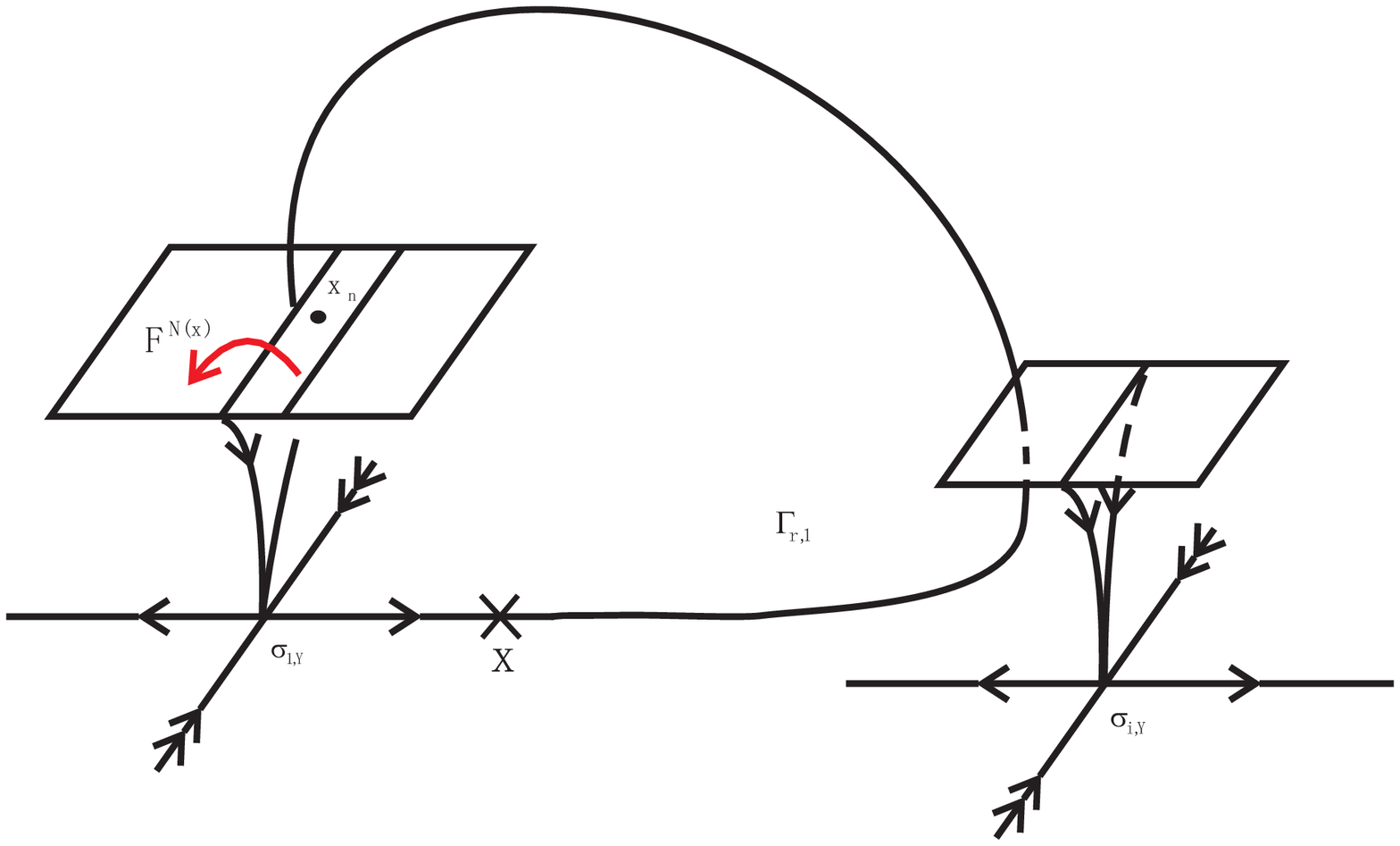}
\caption{}
\end{figure}
\begin{proof}
First, notice that $N(x)=\#{\cal T}_x$ is finite.  For $t_0>0$ small enough,
$F^{N(x)}(h_i^+((0,t_0]\times(-1,1)))$ is a connected set contained in $S_i^+$. If$F^{N(x)}(h_i^+((0,t_0]\times(-1,1)))\subset h_i^+((0,1)\times(-1,1))$, then it is true for any $t\in(0,t_0)$. Hence, by Proposition~\ref{Pro:generalresult}, we know that for any neighborhood $\cal V$ of $Y$, there is a weak Kupka-Smale vector field $Z\in\cal V$ such that $Z$ has a periodic sink $\gamma$ and $\overline{{\rm Basin}(\gamma)}\cap C(\sigma_Z)\neq\emptyset$.
\end{proof}

\begin{Remark}
For generic $X$, as in the proof of Proposition~\ref{Pro:crosssectionforpseudoLorenz}, the projection of $C(\sigma)\cap\Sigma$ to the horizontal direction along the strong stable foliation is no where dense. But $Y\in {\cal M}_n$ is not generic. This is why we need to consider a transitive subset $\Lambda_Y$ and the ergodic measure supported on $\Lambda$.
\end{Remark}

\begin{Corollary}\label{Cor:righttoleft}

For any $Y\in{\cal M}_n$,and $\sigma_{i,Y}\in \Lambda_Y\cap {\rm Sing}(Y)$, assume that the right separatrix of $W^u(\sigma_{i,Y})$ is a homoclinic orbit $\Gamma_i$ of $\sigma_{i,Y}$. Let $x\in W^u_{loc}(\sigma_{i,Y})\cap \Gamma_i$ and denote
$$
{\cal T}_{x}=\{t>0:\phi^Y_{t}(x)\in\Sigma\}=\{t_1<t_2<\cdots<t_{N(x)}\},$$
where $N(x)=\#{\cal T}_x$. If $\phi_{t_{N(x)}}^Y(x)$ is contained in $\Lambda_Y$, then $(\Lambda_Y,\mu_Y)$ accumulates $\phi^Y_{t_{N(x)}}(x)$ on the left.
\end{Corollary}
\begin{proof}
If $(\Lambda_Y,\mu_Y)$ accumulates $\phi^Y_{t_{N(x)}}(x)$ on the left, it is finished. Otherwise, $(\Lambda_Y,\mu_Y)$ will accumulate $\phi^Y_{t_{N(x)}}(x)$ on the right. Then according to Proposition \ref{Pro:righttoleft}, for $t>0$ small enough we have that
$$
F^{N(x)}(h_i^+((0,t)\times(-1,1)))\subset h_i^+((-1,0)\times(-1,1)).
$$
This implies that the typical points in $\Lambda_Y$ also accumulate $\phi^Y_{t_{N(x)}}(x)$ on the left.
\end{proof}

In the following, we will assume:

\noindent {\bf There is a neighborhood ${\cal V}^*$ of $Y$ such that for any $Z\in{\cal V}^*$, the basin of any sink of $Z$ cannot accumulate on $C(\sigma_Z)$.}

\begin{Corollary}\label{Cor:leftinside}

For any $Y\in{\cal M}_n$,and $\sigma_{i,Y}\in \Lambda_Y\cap {\rm Sing}(Y)$, if the right separatrix of $W^u(\sigma_{i,Y})$ is a homoclinic orbit $\Gamma_i$ of $\sigma_{i,Y}$, then the left separatrix of $W^u(\sigma_{i,Y})$ is contained in $\Lambda_Y$.
\end{Corollary}

\begin{proof}
If $\Gamma_i\subset\Lambda_Y$, then $(\Lambda_Y,\mu_Y)$ accumulates $\phi^Y_{t_{N(x)}}(x)$ on the left, which implies the left separatrix of $W^u(\sigma_{i,Y})$ is contained in $\Lambda_Y$ by Corollary~\ref{Cor:righttoleft}; otherwise, since $W^u(\sigma_{i,Y})\setminus{\sigma_{i,Y}}\cap\Lambda_Y\neq\emptyset$, we have that the left separatrix of $W^u(\sigma_{i,Y})$ is contained in $\Lambda_Y$.
\end{proof}

\begin{Corollary}\label{Cor:twohomoclinicinside}
For any $Y\in{\cal M}_n$, and the transitive $\cal N$-set $\Lambda_Y\subset C(\sigma_Y)$, if $\sigma_{i,Y}\in\Lambda_Y$ has two homoclinic orbits, then both the two homoclinic orbits are contained in $\Lambda_Y$.
\end{Corollary}
\begin{proof}
Since $\Lambda_Y$ is transitive and $\sigma_{i,Y}\in\Lambda_Y$, one separatrix of $W^u(\sigma_{i,Y})$ is contained in $\Lambda_Y$. We may assume that the right separatrix is contained in $\Lambda_Y$. According to Corollary \ref{Cor:leftinside}, the left separatrix of $W^u(\sigma_{i,Y})$ is contained in $\Lambda_Y$, which means the left homoclinic orbit associated to $\sigma_{i,Y}$ is also contained in $\Lambda_Y$.
\end{proof}
\begin{Proposition}\label{Pro:intersectionontheotherside}
Given $Y\in{\cal M}_n$,  assume that $\Lambda_Y$ contains a homoclinic orbit $\Gamma_i$ of $\sigma_{i,Y}$. If $\Gamma_i\cap \ell_i^+\neq\emptyset$, then $\Lambda_Y\cap \ell_i^-\neq\emptyset$, where $\ell_i^\pm=W^s_{loc}(\sigma_i)\cap S_i^{\pm}$, $\pm\in\{+,-\}$.
\end{Proposition}

\begin{proof}
Suppose on the contrary,  $\Lambda_Y\cap W^s_{loc}(\sigma_{i,Y})\cap S_i^-=\emptyset$. Without loss of generality, we assume that $\Gamma_i$ is the right separatrix of $W^u(\sigma_{i,Y})$. Since $\Lambda_Y\cap W^s_{loc}(\sigma_{i,Y})\neq\emptyset$, we
will have $\Lambda_Y\cap W^s_{loc}(\sigma_{i,Y})\cap S_i^+\neq\emptyset$. Since $\Gamma_i\subset \Lambda_Y$ and the ergodic measure $\mu_Y$ satisfies ${\rm supp}(\mu_Y)=\Lambda_Y$, $\Gamma_i$ can be approximated by typical points of $\mu_Y$. As before, denote by
$$
{\cal T}_{x}=\{t_i>0:\phi^Y_{t_i}(x)\in\Sigma\}=\{t_1<t_2<\cdots<t_{N(x)}\},
$$
where $N(x)=\# {\cal T}_{x}$. We have that $\phi^Y_{t_{N(x)}}(x)\in S_i^+$. For the return map $F$,
there are two cases:

\begin{enumerate}

\item The orientation preserving case: For $t>0$ small enough, one has $F^{N(x)}(h_i^+((0,t)\times(-1,1)))\subset S_i^{+,r}$.

\item The orientation reversing case: For $t>0$ small enough, one has $F^{N(x)}(h_i^+((0,t)\times(-1,1)))\subset S_i^{+,l}$.

\end{enumerate}

In case 1, one will have

\begin{itemize}
\item For $t>0$ small enough, one has $F^{N(x)}(h_i^-((0,t)\times(-1,1)))\subset S_i^{+,l}$.

\item Typical points of $\mu_Y$ in $\Lambda_Y$ can not accumulate $\Gamma_i\cap W^s_{loc}(\sigma_{i,Y})$ on the right by Proposition~\ref{Pro:righttoleft}.

\end{itemize}

Thus, typical points of $\mu_Y$ in $\Lambda_Y$ accumulates $\Gamma_i\cap W^s_{loc}(\sigma_{i,Y})$ on the left. What we need to prove is that: for every point in $\Lambda_Y$ which is close to $x$, then its backward iteration will intersect $S_i^+\cup S_i^-$.

\begin{Claim}
Let $x_n\in\Lambda_Y\cap S_i^{+,l}$ such that $x_n\to \phi^Y_{t_{N(x)}}(x)\in l_i^+$. Then for $n$ large enough $x_n\in F^{N(x)}(h_i^-((0,t)\times(-1,1)))$.
\end{Claim}

\begin{figure}[h]
\centering
\includegraphics[width=0.60\textwidth]{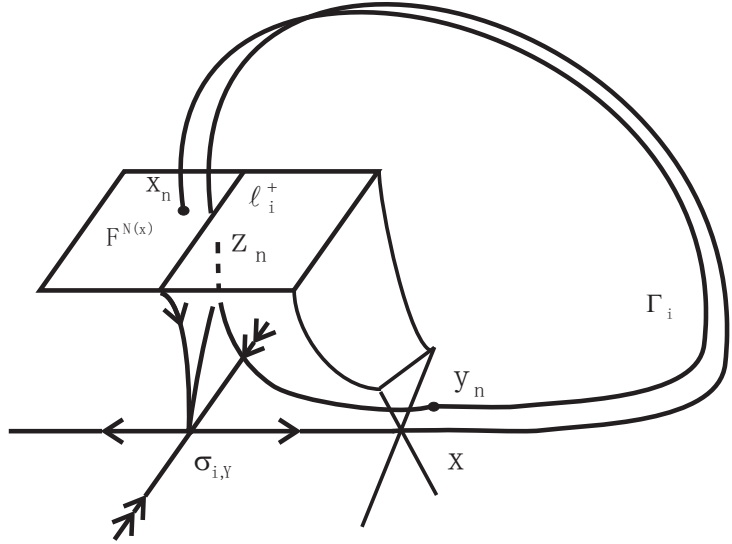}
\caption{Illustration of the Claim}
\end{figure}

\begin{proof}[\bf Proof of the claim:\ ]

Let $U$ be a small ball neighborhood of $\sigma_{i,Y}$. We may assume that $x\in U$ and $B=U\cap W^s_{loc}(\sigma_{i,Y})$ is a 2-disc, which is a neighborhood $\sigma_{i,Y}$ in $W^s_{loc}(\sigma_{i,Y})$.
Moreover, we may assume that $\ell_i^{\pm}\subset B$. According to the construction of cross-section system, we know that

$$\left(B\setminus\overline{L_i^+\cup L_i^-}\right)\cap C(\sigma_Y)=\emptyset,$$ where $$L_i^\pm=\bigcup_{p\in{\ell_i^\pm}}\phi_{[t_p,+\infty)}^Y(p), \quad t_p=\min\{t: \phi_{[t,0]}^Y(p)\subset B \}.$$

Denote by $y_n= \phi^Y_{-t_{N(x)}}(x_n)$. Then $\lim_{n\to\infty}y_n=x$. Let $\tau_n=\min\{t: \phi_{[t,0]}^Y(y_n)\subset U \}.$ Then $\lim_{n\to\infty}\tau_n=-\infty$. Let $z_n=\phi_{\tau_n}^Y(y_n)$. Then $z_n\in B\cap\partial U\cap\Lambda_Y$. Assume that $z=\lim_{n\to\infty}z_n$. Hence $z\in\partial B\cap \Lambda$. The positive orbit of $z$ will intersect $\ell_i^+\cup\ell_i^-$, where the claim follows.
\end{proof}

According to the claim, we get a contradiction in Case 1.

In case 2, for $t>0$ small enough, one has $F^{N(x)}(h_i^-((0,t)\times(-1,1)))\subset S_i^{+,r}$.
No matter typical points of $\mu_Y$ in $\Lambda_Y$ accumulates $\Gamma_i\cap W^s_{loc}(\sigma_{i,Y})$ on the left or on the right, for any $t>0$ small enough, $\Lambda_Y\cap
h_i^-((0,t)\times(-1,1))\neq\emptyset$. Since $\Lambda_Y$ is compact, one has $\Lambda_Y\cap\ell_i^-\neq\emptyset$.
\end{proof}

\begin{Proposition}\label{Pro:notwohomoclinic}
For any $Y\in{\cal M}_n$, and the transitive set $\Lambda_Y\subset C(\sigma_Y)$, if $\sigma_{i,Y}\in\Lambda_Y$, then $\sigma_{i,Y}$ cannot have two homoclinic orbits.

\end{Proposition}

\begin{proof}
Suppose on the contrary, some $\sigma_{i,Y}\in\Lambda_Y$ has two homoclinic orbits. By Corollary~\ref{Cor:twohomoclinicinside}, both homoclinic orbits are contained in $\Lambda_Y$. There are two cases:

\begin{enumerate}
\item[A.] Two-side case: one homoclinic orbit intersects $\ell_i^+$, and the other intersects $\ell_i^-$.
\item[B.] One-side case: both homoclinic orbits intersect $\ell_i^+$ (or $\ell_i^-$).
\end{enumerate}

We may have many subcases. For completeness, we will list all the cases here. We
assume that $x_r\in W^u_{loc}(\sigma_{i,Y})$ is on the right separatrix and $x_l\in W^u_{loc}(\sigma_{i,Y})$ the left separatrix. For $x_r$ and $x_l$, one can define
\begin{eqnarray*}
{\cal T}_{r}&=&\{t>0:\phi^Y_{t}(x_r)\in\Sigma\}=\{t_1<t_2<\cdots<t_{N_r}\},\\
{\cal T}_{l}&=&\{t>0:\phi^Y_{t}(x_l)\in\Sigma\}=\{\tau_1<\tau_2<\cdots<t_{N_l}\},
\end{eqnarray*}
where
$$
N_r=\#{\cal T}_r,\qquad N_l=\#{\cal T}_l.
$$

\begin{enumerate}

\item The right separatrix of the
unstable manifold intersect $S_i^+$ and the left separatrix of the unstable manifold intersect $S_i^-$ (symmetrically, the right separatrix of the
unstable manifold intersect $S_i^-$ and the left separatrix of the unstable manifold intersect $S_i^+$).
\begin{itemize}

\item Orientation-reversing: for small $s>0$, $F^{N_r}(h_i^+((0,s)\times(-1,1)))\subset S_i^{+,l}$ and $F^{N_l}(h_i^-((-s,0)\times(-1,1)))\subset S_i^{-,r}.$

\item Orientation-preserving: for small $s>0$, $F^{N_r}(h_i^+((0,s)\times(-1,1)))\subset S_i^{+,r}$ and $F^{N_l}(h_i^-((-s,0)\times(-1,1)))\subset S_i^{-,l}.$

\item Mixing: for small $s>0$, $F^{N_r}(h_i^+((0,s)\times(-1,1)))\subset S_i^{+,r}$ and $F^{N_l}(h_i^-((-s,0)\times(-1,1)))\subset S_i^{-,r}$; or $F^{N_r}(h_i^+((0,s)\times(-1,1)))\subset S_i^{+,l}$ and $F^{N_l}(h_i^-((-s,0)\times(-1,1)))\subset S_i^{-,l}.$

\end{itemize}

\item Both the right separatrix and the left separatrix of the
unstable manifolds intersect $S_i^+$ (symmetrically, both the right separatrix and the left separatrix of the
unstable manifolds intersect $S_i^-$).
\begin{itemize}

\item Orientation-reversing: for small $s>0$, $F^{N_r}(h_i^+((0,s)\times(-1,1)))\subset S_i^{+,l}$ and $F^{N_l}(h_i^+((-s,0)\times(-1,1)))\subset S_i^{+,r}.$

\item Orientation-preserving: for small $s>0$, $F^{N_r}(h_i^+((0,s)\times(-1,1)))\subset S_i^{+,r}$ and $F^{N_l}(h_i^+((-s,0)\times(-1,1)))\subset S_i^{+,l}.$

\end{itemize}

\end{enumerate}

But in any case, the assumption of Proposition~\ref{Pro:generalresult} will be satisfied.
We only consider one case for instance. For the two-side case, without loss of generality, one can assume that the right separatrix of the
unstable manifold intersect $S_i^+$ and the left separatrix of the unstable manifold intersect $S_i^-$.


We considers the following case:

\begin{itemize}
\item For any $s>0$ small enough,  $F^{N_r}(h_i^+((0,s)\times(-1,1)))\subset S_i^{+,l}$ and $F^{N_l}(h_i^-((-s,0)\times(-1,1)))\subset S_i^{-,r}.$
\end{itemize}

Then for any $\pm\in\{+,-\}$, we have $F^{2(N_r+N_l)}(h_i^+((0,s)\times(-1,1)))\subset h_i^+((0,1)\times(-1,1))$ and $F^{2(N_r+N_l)}(h_i^+((-s,0)\times(-1,1)))\subset h_i^+((-1,0)\times(-1,1))$.The typical point of $\mu_Y$ will accumulate local stable manifold of $\sigma_i$. Then we can apply Proposition~\ref{Pro:generalresult} to get a contradiction.

\end{proof}

Now we can give a contradiction and finish the proof of Theorem~\ref{Thm:partialtoperiodic}.

\paragraph{The contradiction.} For any $Y\in{\cal M}_n$, and the transitive set $\Lambda_Y\subset
C(\sigma_Y)$, according to Theorem \ref{Thm:pujalssambarino}, $\Lambda_Y$ contains a singularity. Take a singularity $\rho=\rho_Y\in\Lambda_Y$. By Lemma \ref{Lem:homoclinicintransitive}, $\rho$ has a homoclinic orbit $\Gamma$. We assume that $\Gamma$ is the right
separatrix of $W^u(\rho)$. By Proposition \ref{Pro:notwohomoclinic}, the left separatrix of $W^u(\rho)$ is not a homoclinic orbit.

If $\Gamma\subset\Lambda_Y$, by Corollary~\ref{Cor:righttoleft}, the left separatrix of $W^u(\rho)$ is contained in $\Lambda_Y$. We assume that $\Gamma\cap \ell^+\neq\emptyset$. By Proposition \ref{Pro:intersectionontheotherside}, $\Lambda_Y\cap \ell^-\neq\emptyset$.

For any neighborhood ${\cal U}$ of $Y$, by using Lemma~\ref{Lem:wenxia} to a transitive orbit, we can get a $C^2$ vector field $Z\in{\cal U}$ without perturbing the homoclinic orbit of singularities of $Y$, which connects the orbit of the left separatrix of $W^u(\rho)$ and $\ell^-$. Then $Z$ has one more
homoclinic orbit of $\rho$. And then by Theorem~\ref{Thm:xiruibin}, there is a $C^2$ weak Kupka-Smale $Z'\in{\cal U}$ such that $Z'$ has at least $(n+1)$ homoclinic orbits of singularities. This fact contradicts the maximality of $n$. So, we proved that $\Gamma \not\subset\Lambda_Y$. Then by a similar argument, for any neighborhood ${\cal U}$ of $Y$, there is a $C^2$ weak Kupka-Smale $Z\in{\cal U}$ which connects the orbit of the left separatrix of $W^u(\rho)$ and some point $a\in(\ell^+\cup\ell^-)\cap \Lambda_Y$. Since $a\not\in \Gamma$, $Z$ also contains at least $(n+1)$ homoclinic orbit of singularities. This finishes the proof of Theorem \ref{Thm:partialtoperiodic}.

\subsection*{Acknowledgements.}

We thank L. Wen for his comments and encouragements, S. Crovisier who listened the proof carefully. We also thank C. Bonatti, Y. Shi, Y. Zhang and X. Wang for useful discussion.

\vskip 5pt

\noindent Shaobo Gan

\noindent Laboratory of Mathematics and Applied Mathematics (LMAM),

\noindent School of Mathematical Sciences

\noindent Peking University, Beijing, 100871, P. R. China

\noindent gansb@math.pku.edu.cn

\vskip 5pt

\noindent Dawei Yang

\noindent School of Mathematics,

\noindent Jilin University, Changchun, 130012, P.R. China

\noindent yangdw1981@gmail.com


\begin{thebibliography}{1d}


\bibitem{ABS77}
V. Afra\u\i movi\v c, V. Bykov, and L. Silnikov, The origin and structure of the Lorenz attractor, {\it Dokl.
Akad. Nauk SSSR}, {\bf 234} (1977), 336--339.

\bibitem{Ano05}D. V. Anosov, Dynamical systems in the 1960s: The hyperbolic revolutions, {\it Methematical events of the
twentieth century}, Springer-Verlag, 2005.

\bibitem{Arn01}M. Arnaud, Cr\'eation de connexions en topologie $C^1$, {\it  Ergodic Theory Dynam.
Systems}, {\bf21}(2001), 339-381.

\bibitem{ArP10}V. Araujo and M. Pacifico, Three-dimensional flows, Springer-Verlag, 2010.

\bibitem{ArR03}A. Arroyo and F. Rodriguez Hertz, Homoclinic bifurcations and uniform hyperbolicity for threedimensional
flows, {\it Ann. Inst. H. Poincar\'e-Anal. NonLin\'eaire}, {\bf20}(2003), 805-841.

\bibitem{Bir35}G. D. Birkhoff, Nouvelles recherches sur les syst\`emes dynamiques, {\it Mem. Pont. Acad. Sci. Nov. Lyncaei}, {\bf 53} (1935), 85-216.
\bibitem{BoC04}C. Bonatti and S. Crovisier, R\'ecurrence et
g\'en\'ericit\'e, {\it Invent. Math.}, {\bf 158} (2004), 33-104.

\bibitem{BoD02}
C. Bonatti and L. D\'iaz, On maximal transitive sets of generic
diffeomorphisms, {\it Inst. Hautes \'{e}tudes Sci. Publ. Math.},
{\bf96}(2002), 171--197.

\bibitem{BDP03} C. Bonatti, L. J. D\'iaz and E. R. Pujals, A
$C^1$-generic dichotomy for diffeomorphisms: Weak forms of
hyperbolicity or infinitely many sinks or sources, {\it Annals of
Math.} {\bf 158} (2003), 355-418.

\bibitem{BDV05}
C. Bonatti L. D\'iaz and M. Viana, Dynamics beyond uniform
hyperbolicity. A global geometric and probabilistic perspective.
\emph{Encyclopaedia of Mathematical Sciences}, {\bf 102}.
Mathematical Physics, III. Springer-Verlag, Berlin,(2005). xviii+384
pp.

\bibitem{BGW07}C. Bonatti, S. Gan and L. Wen,  On the existence of
non-trivial homoclinic classes, {\it  Ergodic Theory Dynam.
Systems}, {\bf 27}(2007), 1473--1508.


\bibitem{BGY11}C. Bonatti, S. Gan and D. Yang, Dominated chain recurrent classes with singularieis, accepted by {\it Ann. Sc. Norm. Super. Pisa}, 2012.

\bibitem{BGV06}C. Bonatti, N. Gourmelon and T. Vivier, Perturbations of the derivative along periodic orbits, {\it  Ergodic Theory Dynam.
Systems}, {\bf 26}(2006), 1307--1337.


\bibitem{Con78}C. Conley, {\it Isolated invariant sets and Morse
index}, CBMS Regional Conference Series in Mathematics, {\bf38}, AMS
Providence, R.I., (1978).

\bibitem{Cro06}S. Crovisier, Periodic orbits and chain transitive
sets of $C^1$-diffeomorphisms, {\it  Publ. Math. Inst. Hautes
\'Etudes Sci.}, {\bf 104}(2006), 87--141.

\bibitem{Cro10}S. Crovisier, Birth of homoclinic intersections: a model for
the central dynamics of partially hyperbolic systems, {\it Ann. Math.}, {\bf 172}(2010), 1641-1677.


\bibitem{CSY11}S. Crovisier, M. Sambarino and D. Yang, Partial hyperbolicity and homoclinic tangencies, arXiv:1103.0869.



\bibitem{Fra71} J. Franks, Neccssary conditions for stabilty of
diffeomorphisms, {\it Trans. Amer. Math. Soc.}, {\bf 158} (1971),
301-308.

\bibitem{Guc76}
J. Guchenheimer, A strange, strange attractor, {\it The Hopf bifurcation  theorems and its applications (Applied
Mathematical Series, {\bf 19})}, Springer-Verlag, 1976, 368--381.

\bibitem{GuW79}
J. Guchenheimer and R. Williams, Structural stability of Lorenz attractors, {\it  Inst. Hautes \'Etudes Sci.
Publ. Math.}, {\bf 50} (1979), 59--72.

\bibitem{Hay97}S. Hayashi, Connecting invariant manifolds and the solution of the $C^1$-stability and
$\Omega$-stability conjectures for flows, {\it Ann of Math.} {\bf145}(1997), 81-137.


\bibitem{HPS77} M. W. Hirsch; C. C. Pugh; M. Shub, Invariant manifolds, {\it Lecture
Notes in Mathematics}, {\bf 583} Springer-Verlag 1977.



\bibitem{Kup64}I. Kupka, Contribution \`a la th\'eorie des champs g\'en\'eriques,
{\it Contrib. Differ. Equ.}, {\bf 2}(1963), 457每484 and {\bf
3}(1964), 411-420.

\bibitem{LGW05}M. Li, S. Gan and L. Wen, Robustly transitive singular sets via
approach of extended linear Poincar\'e flow, {\it Discrete Contin.
Dyn. Syst. }{\bf 13 }(2005), 239-269.

\bibitem{Lia79}S. Liao, A basic property of a certain class of differential
systems, (Chinese) {\it Acta Math. Sinica}, {\bf22}(1979), 316-343. (in chinese)

\bibitem{Lia80}S. Liao, On the stability conjecture, {\it Chinese Annals
of Math.}, {\bf 1}(1980), 9-30.

\bibitem{Lia81}S. Liao, Obstruction sets II, {\it Acta Sci. Natur.
Univ. Pekinensis}, {\bf 2} (1981), 1-36. (in chinese)

\bibitem{Lia85}S. Liao, Certain uniformity properties of differential systems and a generalization of an existence theorem for periodic orbits, {\it Acta Sci. Natur.
Univ. Pekinensis}, {\bf 2} (1981), 1-19. (in chinese)

\bibitem{Lia89}S. Liao, On $(\eta,d)$-contractible orbits of vector
fields, {\it Systems Science and Mathematical Sciences},
{\bf2}(1989), 193-227.

\bibitem{Lor63}
E. N. Lorenz, Deterministic nonperiodic flow, {\it J. Atmosph. Sci.}, {\bf 20} (1963), 130--141.

\bibitem{Man82}R. Ma{\~n}{\'e}, An ergodic closing lemma, {\it Ann. Math.},
{\bf 116}(1982), 503-540.

\bibitem{Man85}R. Ma\~n\'e, Hyperbolicity, sinks and measure in one-dimensional dynamics, {\it Comm. Math. Phys.},
{\bf 100}(1985), 495-524.




\bibitem{MoP02}
C. Morales and M. Pacifico, Lyapunov stability of $\omega$-limit
sets, {\it Discrete Contin. Dyn. Syst.}, {\bf8}(2002), 671-674.

\bibitem{MoP03}C. Morales and M. Pacifico,  A dichotomy for
three-dimensional vector fields, {\it Ergodic Theory and Dynamical
Systems}, {\bf23}(2003), 1575-1600.

\bibitem{MPP98}
C. Morales, M. Pacifico, and E. Pujals, On $C\sp 1$ robust singular transitive sets for three-dimensional flows,
{\it C. R. Acad. Sci. Paris}, {\bf 326} (1998), 81--86.

\bibitem{MPP04}
C. Morales, M. Pacifico, and E. Pujals, Robust transitive singular sets for 3-flows are partially hyperbolic
attractors or repellers, {\it  Ann. of Math.}, {\bf160} (2004), 375--432.

\bibitem{New70}S. Newhouse, Nondensity of axiom A(a) on $S^2$, {\it Global analysis I, Proc. Symp. Pure Math. AMS},
{\bf14}(1970), 191-202.

\bibitem{New74}S. Newhouse, Diffeomorphisms with infinitely many
sinks, {\it Topology}, {\bf 13} (1974), 9--18.

\bibitem{New79}S. Newhouse, The abundance of wild hyperbolic sets
and nonsmooth stable sets for diffeomorphisms, {\it Inst. Hautes
\'Etudes Sci. Publ. Math.}, {\bf50} (1979), 101--151.

\bibitem{Ohn80}T. Ohno, A weak equivalence and topological entropy, {\it Publ. Res. Inst. Math. Sci.}, {\bf16}(1980), 289-298.

\bibitem{Pal88}J. Palis, On the $C^1$ $\Omega$-stability conjecture. {\it  Publ. Math. Inst. Hautes
\'Etudes Sci.}, {\bf 66}(1988), 211-215.

\bibitem{Pal91}J. Palis, Homoclinic bifurcations, sensitive-chaotic dynamics and strange attractors, {\it Dynamical systems and related topics (Nagoya, 1990)}, 466-472,
Adv. Ser. Dynam. Systems, 9,{\it  World Sci. Publ., River Edge, NJ}, 1991.



\bibitem{Pal00}J. Palis, A global view of dynamics and a conjecture of
the denseness of finitude of attractors, {\it Ast\'erisque}, {\bf
261}(2000), 335--347.

\bibitem{Pal05}J. Palis, A global perspective for non-conservative dynamics,
{\it Ann. Inst. H. Poincar\'e Anal. Non Lin\'eaire}, {\bf22}(2005),
485-507.

\bibitem{Pal08}J. Palis, Open questions leading to a global perspective in
dynamics, {\it Nonliearity}, {\bf21}(2008), 37-43.

\bibitem{PdM82}J. Palis, and W. de Melo, Geometric theory of dynamical systems -- an introduction, Springer-Verlag (1982).

\bibitem{Pei62}M. M. Peixoto, Structural stability on two-dimensional manifolds, {\it Topology}, {\bf 1}(1962), 101-120.

\bibitem{Pli72}V. Pliss, A hypothesis due to Smale, {\it Diff. Eq.} {\bf
8} (1972), 203-214.

\bibitem{Poi90}H. Poincar\'e, Sur le probl\`eme des trois corps et les \'equations de la dynamique,
{\it Acta Math.}, {\bf13}(1890), 1-270.

\bibitem{Poi99}H. Poincar\'e, {\it Les m\'ethodes nouvelles de la m\'ecanique c\'eleste, Tome III},
Gauthier-Villars (1899).

\bibitem{Pug67}C. Pugh, The closing lemma, {\it Amer. J. Math.}, {\bf89}1967, 956--1009.


\bibitem{PuS00} E. Pujals and M. Sambarino, Homoclinic tangencies and
hyperbolicity for surface diffeomorphisms, {\it Annals of Math.},
{\bf 151} (2000), 961-1023.

\bibitem{Shi05}L. P. Shil'nikov, Homoclinic Trajectories: From Poincar\'e to the present, {\it Mathematical events
of the tweentieth century}, Springer-Verlag, 2005.

\bibitem{Sma63}S. Smale, Stable manifolds for differential equations and
diffeomorphisms, {\it Ann. Sc. Norm. Super. Pisa}, {\bf17}(1963), 97-116.

\bibitem{Sma65}S. Smale, Diffeomorphisms with many periodic points. In: {\it Differential and Combinatorial Topology}.
Princeton University Press, 1965, 63-80 (Princeton Math. Ser., 24).

\bibitem{SYZ09}W. Sun, T. Young and Y. Zhou, Topological entropies of equivalent smooth flows, {\it Trans. Amer. Math. Soc.}, {\bf361}(2009), 3071-3082.

\bibitem{Tho90}R. Thomas, Topological entropy of fixed-point free flows, {\it Trans. Amer. Math. Soc.}, {\bf319}(1990), 601-618.

\bibitem{Wen96}L. Wen, On the $C^1$-stability conjecture for flows, {\it J. Differential Equations}, {\bf129}(1996), 334-357.

\bibitem{Wen02} L. Wen, Homoclinic tangencies and dominated
splittings, {\it Nonlinearty} {\bf 15} (2002), 1445-1469.

\bibitem{Wen04}L. Wen, Generic diffeomorphisms away from homoclinic
tangencies and heterodimensional cycles, {\it Bull. Braz. Math. Soc.
(N.S.)}, {\bf 35}(2004), 419-452.


%

\bibitem{WeX00}L. Wen and Z. Xia, $C^1$ connecting lemmas, {\it Trans.
Am. Math. Soc.}, {\bf352}(2000), 5213-5230.

\bibitem{Wen08} L. Wen, The selecting lemma of Liao, {\it Disc. Cont.
Dynam. Syst.}, {\bf20}(2008), 159-175.

\bibitem{Xir05}R. Xi, Master thesis in Peking University, 2005.
\end{thebibliography}
\end{document}